\documentclass[12pt]{article}
\usepackage[utf8]{inputenc}
\usepackage[english]{babel}
\usepackage[margin=1in]{geometry}
\usepackage{color}
\usepackage{enumerate}
\usepackage{amsmath, amsthm, amssymb, amsfonts}
\usepackage{bbm}
\usepackage{mathtools}
\usepackage[all,cmtip]{xy}
\usepackage{todonotes}
\usepackage{hyperref}
\allowdisplaybreaks
\newtheorem{theorem}{Theorem}[section]
\newtheorem{corollary}[theorem]{Corollary}
\newtheorem{lemma}[theorem]{Lemma}
\newtheorem*{lma}{Lemma}
\newtheorem*{claim*}{Claim}

\newtheorem{proposition}[theorem]{Proposition}
\theoremstyle{definition}
\newtheorem{definition}[theorem]{Definition}
\newtheorem*{note}{Note}

\newtheorem{remark}{Remark}

\newcommand{\Z}{\mathbb{Z}}
\newcommand{\F}{\mathbb{F}}

\newcommand{\ed}{\operatorname{ed}}
\newcommand{\syl}{\operatorname{Syl}}

\newcommand{\trdeg}{\operatorname{trdeg}}

\newcommand{\mbf}{\mathbf}

\newcommand{\Id}{\text{Id}}

\newcommand{\divides}{\bigm|}

\newenvironment{theorem*}[2][Theorem]{\begin{trivlist}
\item[\hskip \labelsep {\bfseries #1}\hskip \labelsep {\bfseries #2}]}{\end{trivlist}}
\newenvironment{lemma*}[2][Lemma]{\begin{trivlist}
\item[\hskip \labelsep {\bfseries #1}\hskip \labelsep {\bfseries #2}]}{\end{trivlist}}
\newenvironment{corollary*}[2][Corollary]{\begin{trivlist}
\item[\hskip \labelsep {\bfseries #1}\hskip \labelsep {\bfseries #2}]}{\end{trivlist}}
\begin{document}
\title{\texorpdfstring{The essential $2$-dimension of the linear groups}{The essential 2-dimension of the linear groups}}
\author{Hannah Knight\thanks{This work was supported in part by NSF Grant Nos. DMS-1811846 and DMS-1944862 and NSF Award No. 2302822.}\\ Mathematics, UCLA, Los Angeles, CA, USA\\ {\color{blue}\href{mailto:hknight1@uci.edu}{hknight1@uci.edu}}} 
\date{}

\maketitle

\vspace{-1cm}

\begin{abstract}

In this paper, we compute the essential $2$-dimension when the defining prime is odd of the general linear groups, the projective general linear groups, the special linear groups when $n$ is odd or $n = 2$, as well as the special linear groups and quotients of it (such as the projective special linear groups) in the case case $q \equiv 1 \mod 4$, $s = v_2(q-1)$, and $\Gamma = \text{Gal}(k(\zeta_{2^s})/k)$ is trivial.

\end{abstract}

%\tableofcontents

%\bigskip

\section{Introduction}

The goal of this paper is to compute the essential $2$-dimension of the linear groups when the defining prime is odd. Fix a field $k$. The essential dimension of a finite group $G$, denoted $\ed_k(G)$, is the smallest number of algebraically independent parameters needed to define a Galois $G$-algebra over any field extension $F/k$ (or equivalently $G\text{-torsors over }\text{Spec}F$). In other words, the essential dimension of a finite group $G$ is the supremum taken over all field extensions $F/k$ of the smallest number of algebraically independent parameters needed to define a Galois $G$-algebra over $F$.  The essential $2$-dimension of a finite group, denoted $\ed_k(G,2)$, is similar except that before taking the supremum, we allow finite extensions of $F$ of odd degree and take the minimum of the number of parameters needed.  In other words, the essential $2$-dimension of a finite group is the supremum taken over all fields $F/k$ of the smallest number of algebraically independent parameters needed to define a Galois $G$-algebra over a field extension of $F$ of odd degree.  See Section \ref{edbackground} for more formal definitions. See also \cite{BR} and \cite{KM} for more detailed discussions. For a discussion of some interesting applications of essential dimension and essential $p$-dimension, see \cite{Reich}.

What is the essential dimension of the finite simple groups? This question is quite difficult to answer. A few results for small groups (not necessarily simple) have been proven. For example, it is known that $\ed_k(S_5) = 2$, $\ed_k(S_6) = 3$ for $k$ of characteristic not $2$ \cite{BF}, and $\ed_k(A_7) = \ed_k(S_7) = 4$ in characteristic $0$ \cite{Dun}.  It is also known that for $k$ a field of characteristic $0$ containing all roots of unity, $\ed_k(G) = 1$ if and only if $G$ is isomorphic to a cyclic group $\Z/n\Z$ or a dihedral group $D_m$ where $m$ is odd (\cite{BR}, Theorem 6.2). Various bounds have also been proven. See \cite{BR}, \cite{Mer}, \cite{Reich},\cite{Mor}, among others. For a nice summary of the results known in 2010, see \cite{Reich}.

We can find a lower bound to this question by considering the corresponding question for essential $p$-dimension. In my first article (\cite{Kni}), I calculated the essential $p$-dimension of the split finite quasi-simple my second article, I calculated the essential $l$-dimension of the groups at a prime $l$, where $l \neq 2$ and $l \neq p$ (where $p$ is the defining prime). In this article, I calculate the essential $2$-dimension of the groups with $2 \neq p$ (where $p$ is the defining prime).

\begin{definition}
Let $v_2(n)$ denote the largest integer $i$ such that $2^i \divides n$. Let $\zeta_n$ denote a primitive $n$th root of unity. 
\end{definition}

\begin{theorem}\label{mainthm} 

\begin{enumerate}[(1)]  Let $p \neq 2$ be a prime and $q = p^r$. Let $k$ be a field with $\text{char } k \neq 2$. Let $$s = \begin{cases} v_2(q-1), &q \equiv 1 \mod 4\\
v_2(q+1) + 1, &q \equiv 3 \mod 4 \end{cases}.$$  
Let $\epsilon = \zeta_{2^s}$ in $k_\text{sep}$ and let $\Gamma = \text{Gal}(k(\epsilon)/k)$. Then

\item (Theorem \ref{GLn2}) If $q = 1 \mod 4$, then
$$\ed_k(GL_n(\F_q),2) = n[k(\zeta_{2^s}):k].$$

\item (Theorem \ref{GLn2'}) If $q = 3 \mod 4$, then
\begin{align*}\ed_k(GL_n(\F_q), 2) &= \begin{cases} 2m[k(\epsilon - \epsilon^{-1}):k], &n = 2m\\ed_k(GL_{2m}(\F_q),2) + 1, &n = 2m+1
\end{cases}\\
&=\begin{cases} m[k(\epsilon):k], &n = 2m, \text{ } \Gamma = \langle 5^{2^{i}}, -1 \rangle \text{ or } \langle -5^{2^{i}} \rangle \text{ for } i = 0, \dots, s-3\\
&\text{ or equivalently }[2^{s-1}-1] \in \Gamma\\
2m[k(\epsilon):k], &n = 2m, \text{ } \Gamma =\langle -1 \rangle \text{ or } \langle 5^{2^i} \rangle \text{ for } i = 0,\dots, s-2\\
&\text{ or equivalently }[2^{s-1}-1] \notin \Gamma\\
\ed_k(GL_{2m}(\F_q),2) + 1, &n = 2m+1
\end{cases}.
\end{align*}

\item (Theorem \ref{PGLn2})
If $q \equiv 1 \pmod 4$, then
$$\ed_k(PGL_n(\F_q),2) = \begin{cases} 
\ed_k(GL_{n-1}(\F_q),2), &2 \nmid n\\
2^{v_2(n)}(n-2^{v_l(n)})[k(\epsilon):k)], &2 \divides n \text{ and } n \neq 2^t\\
 2^{2t-1}[k(\epsilon):k], &n=2^t, \text{ } [-1] \notin \Gamma\\
 2^{2t-2}[k(\epsilon):k], &n = 2^t,\text{ } [-1] \in \Gamma
 \end{cases}$$

\item (Theorem \ref{PGLn2'}) If $q \equiv 3 \mod 4$, then
$$\ed_k(PGL_n(\F_q),2) = \begin{cases} 
\ed_k(GL_{n-1}(\F_q),2), &2 \nmid n\\
2[k(\zeta_{2^{s-1}}):k], &n=2, \text{ } [-1] \notin \text{Gal}(k(\zeta_{2^{s-1}})/k)\\
[k(\zeta_{2^{s-1}}):k], &n = 2, \text{ } [-1] \in \text{Gal}(k(\zeta_{2^{s-1}})/k)\\
2^{2+v_2(m)}(m-2^{v_2(m)})[k(\epsilon):k], &n = 2m, \text{ } m \neq 2^t,\text{ } [2^{s-1}-1] \notin \Gamma\\
2^{1+v_2(m)}(m-2^{v_2(m)})[k(\epsilon:k], &n=2m, \text{ } m \neq 2^t, \text{ } [2^{s-1}-1] \in \Gamma\\
2^{2t}[k(\epsilon):k], &n = 2m, \text{ } m = 2^t, \text{ } [2^{s-1}-1] \notin \Gamma\\
2^{2t-1}[k(\epsilon):k], &n = 2m, \text{ } m = 2^t, \text{ } [2^{s-1}-1] \in \Gamma
 \end{cases}$$
 
\item (Theorem \ref{SLn2'}) If $q \equiv 1 \mod 4$, then
\begin{align*}
&\ed_k(SL_n(\F_q),2)\\
&= \begin{cases}
\ed_k(GL_{n-1}(\F_q),2), &2 \nmid n\\ 
 2[k(\epsilon):k], &n=2, \text{ } [-1] \notin \Gamma\\
[k(\epsilon):k], &n=2, \text{  } [-1] \in \Gamma, \text{ } x^2 + y^2 = -1 \text{ has a solution in } k(\epsilon+\epsilon^{-1}) \\
 2[k(\epsilon):k], &n=2, \text{  }[-1] \in \Gamma, \text{ } x^2 + y^2 = -1 \text{ has no solutions in } k(\epsilon+\epsilon^{-1})
\end{cases}.
\end{align*}

\item (Theorem \ref{SLn2''}) If $q \equiv 3 \mod 4$, let $\xi = \zeta_{2^{s-1}}$ in $k_\text{sep}$, and let $\Gamma' = \text{Gal}(k(\xi)/k)$.  Then
\begin{align*}
&\ed_k(SL_n(\F_q), 2)\\
&= \begin{cases}
\ed_k(GL_{2m}(\F_q),l), &n = 2m + 1\\
2[k(\xi):k], &n=2,  [-1] \notin \Gamma'\\
[k(\xi):k], &n=2, \text{  } [-1] \in \Gamma', \text{ } x^2 + y^2 = -1 \text{ has a solution in } k(\xi+\xi^{-1}) \\
 2[k(\xi):k], &n=2, \text{  }[-1] \in \Gamma', \text{ } x^2 + y^2 = -1 \text{ has no solutions in } k(\xi+\xi^{-1})
 \end{cases}
\end{align*}

\item (Theorem \ref{SLn2})  
 Suppose that $q \equiv 1 \mod 4$ and $\Gamma$ is trivial. Then
$$\ed_k(SL_n(\F_q),2) = \begin{cases}
 \ed_k(GL_{n-1}(\F_q),2), &2 \nmid n \\
\ed_k(GL_n(\F_q),2) &2 \divides n
\end{cases}$$

\item (Theorem \ref{PSLn2}) Suppose that $q \equiv 1 \pmod 4$ and $\Gamma$ is trivial. Let $G = SL_n(\F_q)/\{x\text{Id} : x \in \F_q, \text{ } x^{n'} = 1\}$. Let $v = \min(v_2(n'),s)$. Then if $2 \nmid n'$, then $\ed_k(G,2) = \ed_k(SL_n(\F_q),2)$. And if $2 \divides n'$, then 
$$\ed_k(G,l) = \begin{cases} 
2, &n' = n = 2\\
2^{2t-2}, &n=2^t, \text{ } t > 2, \text{ } v = 1\\
\ed_k(PGL_n(\F_q),2) = 2^{2t-1}, &n = 2^t, \text{ } t > 2, \text{ } v > 1\\
\ed_k(PGL_n(\F_q),2) = 2^{v_2(n)}(n-2^{v_2(n)}), &n \neq 2^t
\end{cases}.$$
Note that for $n' = n,$ $G = PSL_n(\F_q).$

\end{enumerate}
\end{theorem}

\begin{remark}\label{remark1} Duncan and Reichstein calculated the essential $p$-dimension of the pseudo-reflection groups. These groups overlap with the groups above in a few small cases. See the appendix in \cite{Kni2} for the overlapping cases .
\end{remark}

\noindent \textbf{Acknowledgements:} I would like to thank Zinovy Reichstein and Jesse Wolfson for their kind mentorship and invaluable help. I am also grateful to Hadi Salmasian, Federico Scavia, and Jean-Pierre Serre for very helpful comments on a draft.

\bigskip
 
\section{Essential Dimension and Representation Theory Background}\label{edbackground}

For completeness, we recall the relevant background. Fix a field $k$.  Let $G$ be a finite group, $p$ a prime. 

\begin{definition} Let $T: \text{Fields}/k \to \text{Sets}$ be a functor. Let $F/k$ be a field extension, and $t \in T(F)$. The\emph{ essential dimension of} $\mathit{t}$ is 
$$\ed_k(t) = \min_{F' \subset F \text{ s.t. } t \in Im(T(F') \to T(F))} \trdeg_k(F').$$\end{definition}

\begin{definition} Let $T: \text{Fields}/k \to \text{Sets}$ be a functor. The \emph{essential dimension of} $\mathit{T}$ is
$$\ed_{k}(T) = \sup_{t \in T(F), F/k \in \text{Fields}/k} \ed_k(t).$$\end{definition} 

\begin{definition} For $G$ be a finite group, let $$H^1(-;G):\text{Fields}/k \to \text{Sets}$$ be defined by $$H^1(-;G)(F/k) = \{\text{the isomorphism classes of } G\text{-torsors over }\text{Spec}F \}.$$ 
\end{definition}

\begin{definition} The \emph{essential dimension of} $\mathit{G}$ is
$$\ed_k(G) = \ed_k(H^1(-;G)).$$ \end{definition}

\begin{definition} Let $T: \text{Fields}/k \to \text{Sets}$ be a functor. Let $F/k$ be a field extension, and $t \in T(F)$. The \emph{essential} $\mathit{p}$\emph{-dimension of} $\mathit{t}$ is 
$$\ed_k(t,p) = \min \trdeg_k(F'')$$
where the minimum is taken over all
\begin{align*}
F'' \subset F' \text{ a finite extension}, \text{ with } F \subset F'\\
[F':F] \text{ finite } \text{ s.t. } p \nmid [F':F] \text{ and }\\
\text{the image of } t \text{ in } T(F') \text{ is in } \text{Im}(T(F'') \to T(F'))
\end{align*}
\end{definition}

\begin{note} $\ed_k(t,p) = \min_{F \subset F', p \nmid [F':F]} \ed_k(t|_{F'}).$
\end{note}

\begin{definition} Let $T: \text{Fields}/k \to \text{Sets}$ be a functor. The \emph{essential} $\mathit{p}$\emph{-dimension of} $\mathit{T}$ is
$$\ed_{k}(T,p) = \sup_{t \in T(F), F/k \in \text{Fields}/k} \ed_k(t,p).$$\end{definition}

\begin{definition} The \emph{essential} $\mathit{p}$\emph{-dimension of} $\mathit{G}$ is 
$$\ed_k(G,p) = \ed_k(H^1(-;G),p).$$ \end{definition}

Let $\syl_p(G)$ denote the set of Sylow $p$-subgroups of $G$. 

\begin{lemma}\label{Lempsyl} Let $S \in \syl_p(G)$. Then $\ed_k(G,p) = \ed_k(S,p).$  \end{lemma}

\begin{lemma}[\cite{Kni}, Corollary 2.11]\label{sylp} If $k_1/k$ a finite field extension of degree prime to $p$, $S \in \syl_p(G)$. Then $\ed_k(G,p) = \ed_k(S,p) = \ed_{k_1}(S,p).$ \end{lemma}

\begin{corollary}\label{rootofunity} Let $G$ be a finite group, $k$ a field of characteristic $\neq p$, $S \in \syl_p(G)$, $\zeta$ a primitive $p$-th root of unity, then
$$\ed_k(G,p) = \ed_{k(\zeta)}(S,p).$$ \end{corollary}

\begin{theorem}\label{KM4.1} [Karpenko-Merkurjev \cite{KM}, Theorem 4.1] Let $G$ be a $p$-group, $k$ a field with $\text{char } k \neq p$ containing a primitive $p$th root of unity. Then $\ed_k(G,p) = \ed_k(G)$ and $\ed_k(G,p)$ coincides with the least dimension of a faithful representation of $G$ over $k$.
\end{theorem}

\noindent The Karpenko-Merkurjev Theorem allows us to translate the question formulated in terms of extensions and transcendance degree into a question of representation theory of Sylow $p$-subgroups.  

\begin{definition}\label{Def3.1} Let $H$ be an abelian $p$-group. Define $H[p]$ to be the largest elementary abelian $p$-group contained in $H$, i.e. $H = \{z \in H : z^p = 1\}$. \end{definition}

\begin{definition} For $G$ an abelian group, $k$ a field, let $\widehat{G} = \text{Hom}(G, k_\text{sep}^\times)$, where $k_\text{sep}$ denotes a separable closure of $k$ in $\overline{k}$. \end{definition}

\begin{remark} Note that if $G$ is elementary abelian $p$-group and $k$ contains $p$-th roots of unity, then $\widehat{G}$ is simply the characters of $G$.
\end{remark}

\begin{remark} Note that for $G = (\Z/l^s\Z)^n$, $\widehat{G} = \text{Hom}(G,k(\zeta_{l^s})^\times)$.
\end{remark}

\begin{definition} For an abelian $p$-group $H$, let $\text{rank}(H)$ denote the rank of $H[p]$ as a vector space over $\F_p$.
\end{definition}

The next two lemmas are due to Meyer-Reichstein \cite{MR} and reproduced in \cite{BMS}.

\begin{lemma}[\cite{MR}, Lemma 2.3; \cite{BMS}, Lemma 3.5]\label{BMKS3.5} Let $k$ be a field with $\text{char } k \neq p$ containing $p$-th roots of unity. Let $H$ be a finite $p$-group and let $\rho$ be a faithful representation of $H$ of minimal dimension. Let $C = Z(H)$. Then $\rho$ decomposes as a direct sum of exactly $r = \text{rank}(C)$ irreducible representations
$$\rho = \rho_1 \oplus \ldots \oplus \rho_r.$$ and 
if $\chi_i$ are the central characters of $\rho_i$, then $\{\chi_i|_{C[p])}\}$ is a basis for $\widehat{C[p]}$ over $k$. \end{lemma}

\begin{lemma}[\cite{MR}, Lemma 2.3; \cite{BMS}, Lemma 3.4]\label{BMKS3.4} Let $k$ be a field with $\text{char } k \neq p$ containing $p$-th roots of unity. Let $H$ be a finite $p$-group and let $(\rho_i: H \to GL(V_i))_{1 \leq i \leq n}$ be a family of irreducible representations of $H$ with central characters $\chi_i$. Let $C = Z(H)$. Suppose that $\{\chi_i|_{C[p]} : 1 \leq i \leq n\}$ spans $\widehat{C[p]}$. Then $\bigoplus_i \rho_i$ is a faithful representation of $H$.
\end{lemma}

\noindent Lemmas \ref{BMKS3.5} and \ref{BMKS3.4} allow us to translate a question of analyzing faithful representations into a question of analyzing irreducible representations. We will need a few more lemmas for the proof.

\begin{definition} For $l$ a prime, $n \in \Z$, let $\mu_l(n)$ denote the the largest integer $d$ such that $l^d \leq n$. 
\end{definition} 

We proved the following lemmas in \cite{Kni2}:

\begin{lemma}\label{Pl(Sn)} Let $\sigma_{i}^{j}$ be the permutation which permutes the $i$th set of $l$ blocks of size $l^{j-1}$. Then 
$$\langle \{\sigma_i^j\}_{1 \leq j \leq \mu_l(n), 1 \leq i \leq \lfloor \frac{n}{l^{j}} \rfloor} \rangle \in \syl_l(S_n).$$ Let $P_l(S_n)$ denote this particular Sylow $l$-subgroup of $S_n$.
\end{lemma}

\begin{definition} Write $n$ in base $l$ as $n = \sum_{i=0}^{\mu_l(n)} a_i l^i$, and let $\xi_l(n)$ denote the sum of the nonzero digits of $n$ when written in base $l$, that is $\xi_l(n) = \sum_{i=0}^{\mu_l(n)} a_i$.
\end{definition}

\begin{definition} Let $I_j$ be the orbits of $\{1, \dots, n\}$ under the action of $P_l(S_n)$.  There are $\xi_l(n)$ such orbits (see the section on $SL_n(\F_q)$). Let $i_j$ denote the smallest index in $I_j$. For each $j$, note that $|I_j| = l^{k}$ for some $k$. Let $k_j$ be such that $|I_j| = l^{k_j}$. 
\end{definition}

\begin{lemma}\label{edwreath} Let $H$ be a finite group. For any prime $l$, let $P = H^N \rtimes P_l(S_N)$. Then
$$\ed_k(P,l) = N\ed_k(H,l).$$
\end{lemma}

\begin{lemma}\label{restrZ}  A representation of a finite $p$-group, $H$, is faithful if and only if its restriction to $Z(H)$ is faithful and if and
only if its restriction to $Z(H)[p]$ is faithful.
\end{lemma}

\begin{definition} Let $|G|_l = l^{v_l(|G|)}$; i.e. $|G|_l$ is the order of a Sylow $l$-subgroup of $G$. \end{definition}

\begin{lemma}\label{claim1} For an invertible matrix $A$, there is a rearrangement of the columns such that $a_{i,i} \neq 0$ for all $i.$. 
\end{lemma}

\begin{definition} Let $\mu_{l^s}$ denote the group of $l^s$-th roots of unity. Note that $\mu_{l^s} =  \langle \zeta_{l^s} \rangle$.
\end{definition}

\begin{definition} For $\mbf{a} \in (\Z/l^s\Z)^n$,  define $\psi_{\mbf{a}} \in \widehat{(\mu_{l^s})^n}$ to be $\psi_{\mbf{a}}: (\mu_{l^s})^n \to k(\zeta_{f})^{\times}$ given by 
 $$\psi_{\mbf{a}}(\mbf{x}) = \prod_{i=1}^n (x_i)^{a_i}.$$
Let $f = \frac{l^s}{\text{gcd}(a_i)}$. View $k(\zeta_f)$ as a vector space over $k$. Let $d = [k(\zeta_f):k]$, and let the representation $\Psi_{\mbf{a}}: (\mu_{l^s})^n \to GL_d(k)$ be defined by 
$$\Psi_{\mbf{a}}(\mbf{x}) = \text{ multiplication by } \prod_{i=1}^n (x_i)^{a_i}.$$
\end{definition}

\begin{remark} Note that the map given by $\mbf{a} \mapsto \psi_{\mbf{a}}$ is an isomorphism between $(\Z/l^s\Z)^n$ and $\widehat{\mu_{l^s}}$. 
\end{remark}

\begin{definition}
Let $\Gamma = \text{Gal}(k(\zeta_{l^s})/k)$. For $\phi \in \Gamma$, note that $\phi(\zeta_{l^s}) = (\zeta_{l^s})^{\gamma_\phi}$ for a unique $\gamma_\phi \in (\Z/l^s\Z)^\times$. Define $\gamma_\phi$ to be the element of $(\Z/l^s\Z)^\times$ such that $\phi(\zeta_{l^s}) = (\zeta_{l^s})^{\gamma_\phi}$.
\end{definition}

\begin{remark} Note that the map $\phi \mapsto \gamma_\phi$ gives an injection $\Gamma \hookrightarrow (\Z/l^s\Z)^\times$. \end{remark}

\begin{lemma}\label{auxlemma} For any prime $l$, let  $\Gamma = \text{Gal}(k(\zeta_{l^s})/k)$. Consider the action of $\Gamma$ on $\widehat{(\mu_{l^s})^n}$ given by $\phi(\psi_{\mbf{a}}) = \phi \circ \psi_{\mbf{a}}$. Then the corresponding action of $\gamma_\phi \in (\Z/l^s\Z)^\times$ on $(\Z/l^s\Z)^n \cong \widehat{(\mu_{l^s})^n}$ is given by scalar multiplication by $\gamma_\phi$. \end{lemma}

\begin{lemma}\label{corrlemma} For any prime $l$, let $\Gamma = \text{Gal}(k(\zeta_{l^s})/k) \hookrightarrow (\Z/l^s\Z)^\times$. Then the irreducible representations of $(\mu_{l^s})^n$ over $k$ are in bijection with $\mbf{a} \in (\Z/l^s\Z)^n/\Gamma$, where the action of $\phi \in \Gamma$ is given by scalar multiplication by $\gamma_\phi$. The bijection is given by $\mbf{a} \in (\Z/l^s\Z)^n/\Gamma \mapsto \Psi_{\mbf{a}}: (\mu_{l^s})^n \to GL_d(k)$, where $d = [k(\zeta_f):k]$ for $f = \frac{l^s}{\text{gcd}(a_i)}$. Furthermore, if $\Psi_\mbf{a}$ is non-trivial on $S[l]$, then $l \nmid a_i$ for some $i$ and $\Psi_{\mbf{a}}$ has dimension $[k(\zeta_{l^s}):k]$.
\end{lemma}

\begin{lemma}\label{changepersp} For any prime $l$, let  $\Gamma = \text{Gal}(k(\zeta_{l^s})/k) \hookrightarrow (\Z/l^s\Z)^\times$ and the action of $\phi \in \Gamma$ be given by scalar multiplication by $\gamma_\phi$. Then the orbit of $\Psi_{\mbf{a}}$ under the action of $P_l(S_n)$ on $\text{Irr}((\mu_{l^s})^n)$ will have the same size as the orbit of $\mbf{a}$ under the action of $P_l(S_n)$ on $(\Z/l^s\Z)^n/\Gamma$.  
\end{lemma}

\begin{remark} Let $T = \{\mbf{b} \in (\mu_{l^s})^n : \prod_{i=1}^n b_i = 1\}$. Note that the map given by $\mbf{a} \mapsto \Psi_{\mbf{a}}|_T$ gives an isomorphism between $(\Z/l^s\Z)^n/\{(x,\dots,x)\}$ and $\widehat{T}$. 
\end{remark}

\begin{lemma}\label{auxlemmaPSLn} For any prime $l$, let  $\Gamma = \text{Gal}(k(\zeta_{l^s})/k)$. Let $T = \{\mbf{b} \in (\mu_{l^s})^n : \prod_{i=1}^n b_i = 1\}$ and let  $H = (\Z/l^s\Z)^{n}/\{(x,\dots,x)\}$. Consider the action of $\Gamma$ on $\widehat{T}$ given by $\phi(\lambda) = \phi \circ \lambda$ for $\lambda \in \widehat{T}$. Then the corresponding action of $\gamma_\phi \in (\Z/l^s\Z)^\times$ on $H = (\Z/l^s\Z)^n/\{(x,\dots,x)\} \cong \widehat{T}$ is given by scalar multiplication by $\gamma_\phi$. \end{lemma}

\begin{lemma}\label{corrlemmaPSLn} For any prime $l$, let  $v = \min(v_l(n'),s)$. Let $\Gamma = \text{Gal}(k(\zeta_{l^s})/k) \hookrightarrow (\Z/l^s\Z)^\times$. Let $T = \{\mbf{b} \in (\mu_{l^s})^n : \prod_{i=1}^n b_i = 1\}$ and let $H = (\Z/l^s\Z)^{n}/\{(x,\dots,x)\}$. Then the irreducible representations of $T$ are in bijection with  $\mbf{a} \in H/\Gamma$, where the action of $\phi \in \Gamma$ is given by scalar multiplication by $\gamma_\phi$. The bijection is given by $\mbf{a} \in H/\Gamma \mapsto \Psi_{\mbf{a}}|_T: T \to GL_d(k)$, where $d = [k(\zeta_f):k]$ for $f = \frac{l^s}{\text{gcd}(a_i)}$. Furthermore, if $\Psi_{\mbf{a}}|_T$ is non-trivial on $T[l]$, then $l \nmid a_i$ for some $i$ and $\Psi_{\mbf{a}}$ has dimension $[k(\zeta_{l^s}):k]$.
\end{lemma}

\begin{lemma}\label{changeperspPSLn} For any prime $l$, let  $\Gamma = \text{Gal}(k(\zeta_{l^s})/k) \hookrightarrow (\Z/l^s\Z)^\times$ and the action of $\phi \in \Gamma$ be given by scalar multiplication by $\gamma_\phi$.  Let $T = \{\mbf{a} \in (\mu_{l^s})^n : \prod_{i=1}^n a_i = 1\}$. Let $H = (\Z/l^s\Z)^{n}/\{(x,\dots,x)\}$. Then the orbit of $\Psi_{\mbf{a}}$ under the action of $P_l(S_n)$ on $\text{Irr}(T)$ will have the same size as the orbit of $\mbf{a}$ under the action of $P_l(S_n)$ on $H/\Gamma$.  
\end{lemma}

\begin{lemma}\label{irrH1} For any $l$, let $\mbf{a} = (a_1, \dots, a_{l^{k}})$ with $\sum_{i=1}^{l^{k}} a_i$ invertible. Then 
$$|\text{orbit}(\mbf{a})| \geq l^{k}$$ under the action of $P_l(S_{l^{k}})$ on $(\Z/l^s\Z)^{l^k}$.  
\end{lemma}

\begin{definition} 
The dihedral groups are groups of order $2n$ with the following presentation:
$$D_{2n} = \langle x,y : x^{n} = 1 = y^2, yxy = x^{-1} \rangle.$$
\end{definition}

\begin{definition} 
The semi-dihedral groups are groups of order $2^n$ with the following presentation:

$$SD_{2^n} = \langle x,y : x^{2^{n-1}} = y^2 = 1, yxy = x^{2^{n-2}-1} = -x^{-1}\rangle.$$

\end{definition}

\begin{definition} 
The generalized quaternion groups are groups of order $4n$ with the following presentation:

$$Q_{4n} = \langle w, v : w^{n} = v^2, w^{2n} = 1, vwv^{-1} = w^{-1} \rangle.$$

\begin{proposition}\label{quatbasis} Let $A$ be a $4$-dimensional central simple algebra over $F$. Suppose that there exist $j,k \in A$ satisfying the following conditions: 
$$j^2=k^2=-1, jk=-kj.$$
Then $A = (-1,-1)_F$. \end{proposition} 

\begin{proof}
It suffices to show that $\{1,j,k,jk\}$ is a basis for $A$.  Suppose by way of contradiction that $\{1,j,k,jk\}$ is linearly dependent over $F$.  Then we would have $a,b,c,d \in F$ such that $a + bj + ck + djk = 0$. But then multiplying on the left by $j$ we would get 
\begin{align*}
0 &= aj + bj^2 + cjk + dj^2k\\
&= aj - b + cjk + (dj)(jk)\\
&= (c+dj)jk + aj - b
\end{align*}
So $jk$ lies in the span of $1$ and $j$ over $F$. So $jk$ commutes with $j$. But $j(kj) = -j^2k = -k \neq k = (kj)j$. This is a contradiction. Therefore $\{1,j,k,jk\}$ is linearly dependent over $F$. So since $A$ is a $4$-dimensional over $F$, $\{1,j,k,jk\}$ is a basis for $A$ over $F$. Hence $A = (-1,-1)_F$.
\end{proof}

\begin{lemma}\label{Gamma}
 Let $s  > 2$ be an integer, let $\epsilon = \zeta_{2^{s}}$ in $k_\text{sep}$. Let $\Gamma = \text{Gal}(k(\epsilon)/k)$. Then $5^{2^{s-2}} = 1$ in $(\Z/2^s\Z)^\times$ and
\begin{align*}
 [k(\epsilon):k] = \begin{cases} 2[k(\epsilon+\epsilon^{-1}):k], &\Gamma = \langle 5^{2^{i}}, -1 \rangle \text{ for }i = 0, \dots, s-2\\
&\text{ or equivalently } [-1] \in \Gamma\\
[k(\epsilon+\epsilon^{-1}):k], &\Gamma = \langle -5^{2^i} \rangle \text{ for } i = 1, \dots, s-2\\
&\qquad \text{ or } \langle 5^{2^{i}} \rangle \text{ for } i = 0, \dots, s-3 \\
 &\text{or equivalently } [-1] \notin \Gamma
\end{cases}
\end{align*}

\end{lemma}

\begin{proof}

By Lemma \ref{Gammacond}, $2^{s-1}-1 = -5^{2^{s-3}}$ in $(\Z/2^s\Z)^\times$.
 Hence 
 $$5^{2^{s-2}} = (-5^{2^{s-3}})^2 = (2^{s-1}-1)^2 = 1.$$

Note that 
\begin{align*}
&= (x-\epsilon)(x-\epsilon^{-1})\\
&= x^2 - \epsilon^{-1}x - \epsilon x +1\\
&= x^2 - (\epsilon+\epsilon^{-1})x +1\\
&\in k(\epsilon+\epsilon^{-1})[x]
\end{align*}
So the minimal polynomial for $\epsilon$ over $k(\epsilon+\epsilon^{-1})$ is either $x^2 - (\epsilon+\epsilon^{-1})x + 1$ or $x-\epsilon$. If the minimal polynomial is $x^2 - (\epsilon + \epsilon^{-1})x +1$, then $\text{Gal}(k(\epsilon)/k(\epsilon+\epsilon^{-1}))$ is generated by raising $\epsilon$ to the $-1$-th power. And so since $\text{Gal}(k(\epsilon)/k(\epsilon-\epsilon^{-1})) \subset \Gamma$, we must have $\gamma = [-1] \in  \Gamma.$  So if $[-1] \notin \Gamma$, then we can conclude that $[k(\epsilon):k(\epsilon+\epsilon^{-1})] = 1$. 

On the other hand, note that for $\gamma = -1$, $$\epsilon^\gamma + \epsilon^{-\gamma} = \epsilon^{-1}+\epsilon = \epsilon + \epsilon^{-1}.$$ 
So $[-1] \notin \text{Gal}(k(\epsilon+\epsilon^{-1})/k)$. Thus if $[-1] \in \Gamma$, then we can conclude that 
$$\text{Gal}(k(\epsilon+\epsilon^{-1})/k) \neq \Gamma.$$
So if $[-1] \in \Gamma$, then $k(\epsilon) \neq k(\epsilon + \epsilon^{-1})$, and hence $[k(\epsilon):k(\epsilon + \epsilon^{-1})] = 2$. Thus we have shown that 
$$[k(\epsilon):k(\epsilon+\epsilon^{-1})] = \begin{cases} 1, &[-1] \notin \Gamma\\
2, &[-1] \in \Gamma\end{cases}.$$
The subgroups of $(\Z/2^{s}\Z)^\times \cong \langle 5 \rangle \times \langle -1 \rangle \cong (\Z/2^{s-2}\Z) \times \Z/2\Z$ are the following
\begin{itemize}
\item $\langle 5^{2^{i}}, -1 \rangle$ for $i = 0, \dots, s-2$ (this is $\langle -1 \rangle$ for $i = s-2$)
\item $\langle 5^{2^{i}} \rangle$ for $i = 0, \dots, s-2$ (this is the trivial subgroup for $i = s-2$)
\item $\langle -5^{2^{i}} \rangle$ for $i = 0, \dots, s-3$
\end{itemize}
The subgroups which contain $-1$ are $\langle 5^{2^{i}}, -1 \rangle$ for $i = 0, \dots, s-2$ and the subgroups which do not contain $-1$ are
\begin{itemize}
\item $\langle 5^{2^{i}} \rangle$ for $i = 0, \dots, s-2$
\item $\langle -5^{2^{i}} \rangle$ for $i = 0, \dots, s-3$.
\end{itemize}

\end{proof}

\end{definition}
\subsection{Clifford's Theorem}

Let $N \triangleleft G$, $L = G/N$. 

\begin{definition} For a representation $\rho: G \to GL(V)$, $f:G' \to G$, define $f^*(\rho): G' \to GL(V)$ by $f^*(\rho) = \rho \circ f$. \end{definition}

Then note that for $f_1:G'' \to G'$, $f_2: G' \to G$, we have that 
$$f_1^*(f_2^*(\rho)) = (\rho \circ f_2) \circ f_1 = \rho \circ (f_2 \circ f_1) = (f_2 \circ f_1)^*(\rho).$$

Let $\text{Rep}(G)$ denote the set of isomorphism classes of representations of $G$.  Then $\text{Aut}(G)$ acts on $\text{Rep}(G)$ by $f_\rho = (f^{-1})^*(\rho)$ for $f \in \text{Aut}(G), \rho \in \text{Rep}(G)$. This is an action since 
$$(f \circ g)_\rho = ((f \circ g)^{-1})^*(\rho) = \rho \circ (g^{-1} \circ f^{-1}) = (\rho \circ g^{-1}) \circ f^{-1} = (f^{-1})^*((g^{-1})^*(\rho)).$$

Let $\text{Irr}(G) \subset \text{Rep}(G)$ denote the set of isomorphism classes of irreducible representations of $G$. Then $\text{Irr}(G)$ is invariant under the action of $\text{Aut}(G)$ since if $\rho$ is irreducible, then $f_\rho = (f^{-1})^*(\rho) = \rho \circ f^{-1}$ is also irreducible.

Let $\text{Inn}(G)$ denote the set of inner automorphisms of $G$.  for $g \in G$, let $\varphi_g$ denote the inner automorphism $\varphi_g: G \to G$ defined by $\varphi_g(h) = g^{-1}hg$.  

Note that since $N \triangleleft G$, we can restrict $\varphi_g$ to $N$, and so we have $G \to \text{Aut}(N)$ given by $g \mapsto (\varphi_g|_N \in \text{Aut(N)})$. Then $G$ acts on $\text{Rep}(N)$ by 
$$g(\lambda) := (\varphi_g|_N)(\lambda) = ((\varphi_g|_N)^{-1})^*(\lambda) = \lambda \circ (\varphi_g|_N)^{-1}$$
for $g \in G, \lambda \in \text{Rep}(N)$.

Note that if $g \in N$, we can define $\alpha:V \to V$ by $\varphi(v) = \lambda(g)v$, and then for $h \in G$, $v \in V$

$$(\lambda(h) \circ \alpha)(v) = \lambda(h)(\lambda(g)v) = \lambda(hg)(v)$$
and 
$$(\alpha \circ g(\lambda)(h))(v) = (\alpha \circ \lambda \circ (\varphi_g|_N)^{-1}(h))(v) = g\lambda(g^{-1}hg)(v) = \lambda(hg)(v).$$

Thus $g(\lambda) \cong \lambda$ for $g \in N$. Thus $N$ acts trivially on $\text{Irr}(N) \subset \text{Rep}(N)$. Hence $L = G/N$ acts on $\text{Irr}(N) \subset \text{Rep}(N)$.

\begin{theorem}[Clifford's Theorem]\label{cliff}  Let $N \triangleleft G$, $L = G/N$, $\rho \in \text{Irr}(G)$. Then there exist pairwise non-isomorphic $\lambda_1, \dots, \lambda_c \in \text{Irr}(N)$ such that 
$$\rho|_{N} \cong  \left( \oplus_{i=1}^c  \lambda_i \right)^{\oplus d}, \text{ for some } c, d,$$ 
Let $S \subset \text{Irr}(N)$ be the set $\{ \overline{\lambda_1}, \dots, \overline{\lambda_c}\}$.  Then 
\begin{enumerate}
\item $S$ is an $L$-invariant subset of $\text{Irr}(N)$.
\item $L$ acts on $S$ transitively.
\end{enumerate}
\end{theorem}

\begin{proof} 
For the proof, see \cite{Kni2}.
\end{proof}

\section{\texorpdfstring{The General Linear Groups - $q \equiv 1 \mod 4$}{The General Linear Groups - q equiv 1 mod 4}}

In this section, we will prove the following theorem:

\begin{theorem}\label{GLn2} Let $p \neq 2$ be a prime and $q = p^r$. Let $k$ be a field with $\text{char } k \neq 2$.   Assume that $q \equiv 1 \pmod 4$, and let $s = v_2(q-1)$. Then
$$\ed_k(GL_n(\F_q),2) = n[k(\zeta_{2^s}):k].$$
\end{theorem}

\noindent By (\cite{Sta}, Theorem 3.7),
$$|GL_n(\F_q)|_2 = (2^s)^n \cdot 2^{v_2(n!)} = 2^{sn} \cdot |S_n|_2.$$

\begin{proposition}\label{GLsyl2} For $q \equiv 1 \mod 4$, let $s = v_2(q-1)$. Then for $P \in \syl_2(GL_n(\F_q))$, 
$$P \cong (\mu_{2^s})^{n} \rtimes P_2(S_{n}),$$
where the action of $P_2(S_n)$ on $\mbf{b} \in (\mu_{2^s})^n$ is given by permuting the $b_i$.
\end{proposition}

Granting this proposition, we can prove Theorem \ref{GLn2}:
\begin{proof}[Proof of Theorem \ref{GLn2}]

 By Lemma \ref{edwreath} and Proposition \ref{GLsyl2}, 
$$\ed_k(GL_n(\F_q),2) = n\ed_k(\mu_{2^s}, 2).$$
Since $\zeta_2 \in k$ for any $k$, by Theorem \ref{KM4.1}, $\ed_{k}(\mu_2^s,2) = \ed_{k}(\mu_{2^s})$. And by \cite{KM} (Corollary 5.2), $\ed_{k}(\mu_{2^s}) = [k(\zeta_{2^s}):k]$, where $\zeta_{2^s}$ denotes a primitive $2^s$-th root of unity. Thus 
\begin{align*}
\ed_k(GL_n(\F_q),l) &= n[k(\zeta_{2^s}):k].
\qedhere
\end{align*}
 \end{proof}
\begin{proof}[Proof of Lemma \ref{GLsyl2}]
\noindent \footnote{This construction follows \cite{Sta}.}Let $\zeta_{2^s}$ be a primitive $2^s$-th root of unity in $\F_{q}$, and let 
$$E_1 = \begin{pmatrix}
\zeta_{2^s} &   &        & \\
  & 1 &        & \\
  &   & \ddots & \\
  &   &        & 1
\end{pmatrix}, \ldots, E_{n} = \begin{pmatrix}
1 &        &   &   \\
  & \ddots &   &   \\
  &        & 1 &   \\
  &        &   & \zeta_{2^s}
\end{pmatrix}$$
The symmetric group on $n$ letters acts on $\langle E_1, \ldots, E_{n} \rangle$ by permuting the $E_i$, and it can be embedded into $GL_n(\F_q)$. Let
\begin{align*} 
P &= \langle E_1, \ldots, E_{n} \rangle \rtimes P_2(S_{n})\\
&\cong (\mu_{2^s})^{n} \rtimes P_2(S_{n})
\end{align*}
Then
\begin{align*}
|P|& = |(\mu_{2^s})^{n}| \cdot |P_2(S_{n})|\\
&= |GL_n(\F_q)|_2
\end{align*}
Therefore, $P \in \syl_2(GL_n(\F_q))$.
\qedhere
\end{proof}

\section{\texorpdfstring{The General Linear Groups - $q \equiv 3 \mod 4$}{The general Linear Groups - q equiv 3 mod 4}}

In this section, we will prove the following theorem:

\begin{theorem}\label{GLn2'} Let $p \neq 2$ be a prime and $q = p^r$. Let $k$ be a field with $\text{char } k \neq 2$. Assume that $q \equiv 3 \mod 4$, and let $s = v_2(q+1) + 1$. Let $\epsilon = \zeta_{2^s}$ in $k_\text{sep}$. Let $\Gamma = \text{Gal}(k(\epsilon)/k)$.
 Then
\begin{align*}\ed_k(GL_n(\F_q), 2) &= \begin{cases} 2m[k(\epsilon - \epsilon^{-1}):k], &n = 2m\\ed_k(GL_{2m}(\F_q),2) + 1, &n = 2m+1
\end{cases}\\
&=\begin{cases} m[k(\epsilon):k], &n = 2m, \text{ } \Gamma = \langle 5^{2^{i}}, -1 \rangle \text{ or } \langle -5^{2^{i}} \rangle \text{ for } i = 0, \dots, s-3\\
&\text{ or equivalently }[2^{s-1}-1] \in \Gamma\\
2m[k(\epsilon):k], &n = 2m, \text{ } \Gamma =\langle -1 \rangle \text{ or } \langle 5^{2^i} \rangle \text{ for } i = 0,\dots, s-2\\
&\text{ or equivalently }[2^{s-1}-1] \notin \Gamma\\
\ed_k(GL_{2m}(\F_q),2) + 1, &n = 2m+1
\end{cases}.
\end{align*}
\end{theorem}
\noindent By Stather (\cite{Sta}, Theorem 3.7), for $q \equiv 3 \mod 4$,
$$|GL_n(\F_q)|_2 = \begin{cases} 2^{v_2(m!)}\cdot (4 \cdot 2^{s-1})^m, &n = 2m\\
2\cdot 2^{v_2(m!)} \cdot (4 \cdot 2^{s-1})^m, &n = 2m+1\end{cases}.$$

Note that since $q \equiv 3 \mod 4$, we can write $q = 4k+3$, and so $q - 1 = 4k + 2 = 2(2k+1)$. Thus $v_2(q-1) = 1$. Hence $v_2(q^2 - 1) = v_2(q+1) + 1$. 

For the proof, we will use the following propositions, which we will prove afterwards.

\begin{proposition}\label{GLsyl2'} For $q \equiv 3 \mod 4$, let $s = v_2(q+1) + 1$. Then for $P \in \syl_2(GL_n(\F_q))$,  
$$P \cong \begin{cases}
(SD_{2^{s+1}})^m \rtimes P_2(S_m), &n = 2m\\
\left((SD_{2^{s+1}})^m \rtimes P_2(S_m)\right) \times \Z/2\Z, &n = 2m+1\end{cases},$$
\end{proposition}

\begin{proposition}\label{edSemiDihedral} Let $\epsilon = \zeta_{2^s}$ in $k_\text{sep}$. Let $\Gamma = \text{Gal}(k(\epsilon)/k)$. Then
\begin{align*}
\ed_k(SD_{2^{s+1}}) &= 2[k(\epsilon - \epsilon^{-1}):k]\\
&= \begin{cases} [k(\epsilon):k], &\Gamma = \langle 5^{2^{i}}, -1 \rangle \text{ or } \langle -5^{2^{i}} \rangle \text{ for } i = 0, \dots, s-3\\
&\text{ or equivalently }[2^{s-1}-1] \in \Gamma\\
2[k(\epsilon):k], &\Gamma =\langle -1 \rangle \text{ or } \langle 5^{2^i} \rangle \text{ for } i = 0,\dots, s-2\\
&\text{ or equivalently }[2^{s-1}-1] \notin \Gamma
\end{cases}.
\end{align*}
\end{proposition}

Granting these propositions, we can prove Theorem \ref{GLn2'}:
\begin{proof}[Proof of Theorem \ref{GLn2'}]
Note that for $P_1 \in \syl_2(GL_{2m}(\F_q)), P_2 \in \syl_2(GL_{2m+1}(\F_q))$, 
$$P_2 \cong P_1 \times \Z/2\Z.$$ Thus it suffices to prove the theorem for the case $n = 2m$ and then add $1$ for the essential $2$-dimension of $GL_{2m+1}(\F_q)$. By Lemma \ref{edwreath} and Proposition \ref{GLsyl2'}, 
$$\ed_k(GL_{2m}(\F_q),2) = m\ed_k(SD_{2^{s+1}}, 2) = m\ed_k(SD_{2^{s+1}}).$$
By Proposition \ref{edSemiDihedral}, 
\begin{align*}
\ed_k(SD_{2^{s+1}},2) &= 2[k(\epsilon - \epsilon^{-1}):k]\\
&= \begin{cases} [k(\epsilon):k], &\Gamma = \langle 5^{2^{i}}, -1 \rangle \text{ or } \langle -5^{2^{i}} \rangle \text{ for } i = 0, \dots, s-3\\
&\text{ or equivalently }[2^{s-1}-1] \in \Gamma\\
2[k(\epsilon):k], &\Gamma =\langle -1 \rangle \text{ or } \langle 5^{2^i} \rangle \text{ for } i = 0,\dots, s-2\\
&\text{ or equivalently }[2^{s-1}-1] \notin \Gamma
\end{cases}.
\end{align*} Thus 
\begin{align*} \ed_k(GL_{2m}(\F_q),2) &= 2[k(\epsilon - \epsilon^{-1}):k]\\ 
&= \begin{cases} m[k(\epsilon):k], &\Gamma = \langle 5^{2^{i}}, -1 \rangle \text{ or } \langle -5^{2^{i}} \rangle \text{ for } i = 0, \dots, s-3\\
&\text{ or equivalently }[2^{s-1}-1] \in \Gamma\\
2m[k(\epsilon):k], &\Gamma =\langle -1 \rangle \text{ or } \langle 5^{2^i} \rangle \text{ for } i = 0,\dots, s-2\\
&\text{ or equivalently }[2^{s-1}-1] \notin \Gamma
\end{cases} \qedhere
\end{align*}

 \end{proof}
 
 \begin{proof}[Proof of Proposition \ref{GLsyl2'}] \footnote{This construction follows \cite{CF}.}Let $\epsilon$ be $2^{s}$-th root of unity in $\F_{q^2}$, and let 
$$X = \begin{pmatrix} 0 & 1\\
1 & \epsilon + \epsilon^q \end{pmatrix}.$$ 

Since $(\epsilon + \epsilon^q)^q = \epsilon^q + \epsilon$, we can conclude that $\epsilon + \epsilon^q \in \F_q$ and so $X \in GL_2(\F_q)$. Note that since $s = v_2(q+1)+1$, $2^s \divides 2(q+1)$. Thus $(\epsilon^{q+1})^2 = \epsilon^{2(q+1)} = 1$, and so $\epsilon^{q+1} = -1$. Thus for $A = \begin{pmatrix} 1 & 1\\
\epsilon & \epsilon^q \end{pmatrix}$, $$A^{-1}XA = \frac{1}{\epsilon^q - \epsilon} \begin{pmatrix} -1 - \epsilon^2 & -1 - \epsilon^{q+1}\\
1 + \epsilon^{q+1} & 1 + \epsilon^{2q} \end{pmatrix} = \begin{pmatrix} \epsilon & 0\\
0 & \epsilon^q \end{pmatrix}.$$
 Therefore, $|X| = 2^{s}$. Let $$Y = \begin{pmatrix} 1 & 0\\
 \epsilon + \epsilon^q & -1\end{pmatrix}.$$  Then $$(Y)^2 = \begin{pmatrix} 1 & 0\\
 0 & 1 \end{pmatrix}.$$
 Note that
$$X^{2^{s-1}} = A(A^{-1}XA)^{2^{s-1}}A^{-1} = A \begin{pmatrix} -1 & 0 \\
0 & -1 \end{pmatrix} A^{-1} = -\Id.$$
And
 $$YXY = -X^{-1} = X^{2^{s-1}-1}.$$
 Thus $\langle X, Y \rangle$ is isomorphic to the semi-dihedral group $SD_{2^{s+1}}$.
Let
$$X_1 = \begin{pmatrix}
X &   &        & \\
  & 1 &        & \\
  &   & \ddots & \\
  &   &        & 1
\end{pmatrix}, \ldots, X_{m} = \begin{cases} 
 \begin{pmatrix}
1 &        &   &  \\
  & \ddots &   &  \\
  &        & 1 &  \\
  &        &   & X
\end{pmatrix}, &n = 2m\\
\begin{pmatrix}
1 &        &   &   &\\
  & \ddots &   &   &\\
  &        & 1 &   &\\
  &        &   & X &\\
  &        &   &   & 1
\end{pmatrix}, &n = 2m+1\end{cases}$$
and let $$Y_1 = \begin{pmatrix}
Y &   &        & \\
  & 1 &        & \\
  &   & \ddots & \\
  &   &        & 1
\end{pmatrix}, \ldots, Y_{m} = \begin{cases} 
 \begin{pmatrix}
1 &        &   &  \\
  & \ddots &   &  \\
  &        & 1 &  \\
  &        &   & Y
\end{pmatrix}, &n = 2m\\
\begin{pmatrix}
1 &        &   &   &\\
  & \ddots &   &   &\\
  &        & 1 &   &\\
  &        &   & Y &\\
  &        &   &   & 1
\end{pmatrix}, &n = 2m+1\end{cases}$$
The symmetric group on $m$ letters acts on $\langle X_1, \ldots, X_m, Y_1, \ldots, Y_m \rangle$ by permuting the $X_i$ and $Y_i$, and it can be embedded into $GL_n(\F_q)$. For $n = 2m$, let 
\begin{align*}
P_1 &= \langle X_1, \ldots, X_m, Y_1, \ldots, Y_m \rangle \rtimes P_2(S_m)\\
&\cong (SD_{2^{s+1}})^m \rtimes P_2(S_m)
\end{align*}
Then 
\begin{align*}
|P_1|& = |(SD_{2^{s+1}})|^m \cdot |P_2(S_m)|\\
&= (2^{s+1})^m \cdot 2^{v_2(m!)}\\
&= |GL_{2m}(\F_q)|_2
\end{align*}
Therefore, for $n = 2m$, $P_1 \in \syl_{2}(GL_n(\F_q))$.

For $n = 2m+1$, let 
$$Z = \begin{pmatrix} 
1 &        &   & \\
  & \ddots &   & \\
  &        & 1 & \\
  &        &   & -1
\end{pmatrix}.$$  
Let 
\begin{align*}
P_2 &= (\langle X_1, \ldots, X_m, Y_1, \ldots, Y_m\rangle \rtimes P_2(S_m)) \times \langle Z \rangle\\
&\cong \left((SD_{2^{s+1}})^m \rtimes P_2(S_m) \right) \times \Z/2\Z
\end{align*}
Then 
\begin{align*}
|P_2|& = |(SD_{2^{s+1}})^m| \cdot |P_2(S_m)| \cdot |\Z/2\Z|\\
&= (2^{s+1})^m \cdot v_2(m!) \cdot 2\\
&= |GL_{2m+1}(\F_q)|_2
\end{align*}
Therefore, for $n = 2m+1$, $P_2 \in \syl_{2}(GL_n(\F_q))$. \qedhere

\end{proof}

\subsection{\texorpdfstring{Character table of $SD_{2^{s+1}}$}{Character table of SD2s+1}}

We will first find the character table of $SD_{2^{s+1}}$.  Since $SD_{2^{s+1}} = \langle x \rangle \rtimes \langle y \rangle \cong \mu_{2^s} \rtimes \mu_{2}$, we can find the irreducible representations over $k_\text{sep}$ using Wigner-Mackey theory (see \cite{GV}). The distinct irreducible representations of $\mu_{2^s}$ are given by $\Psi_i$ for $i \in \Z/2^s\Z$ and they extend to the whole group if and only if $y(\Psi_i) = \Psi_i$. And 
$y(\Psi_i)(x) = \Psi_i(yxy) = \Psi_i(x^{2^{s-1}-1}) = \text{multiplication by } x^{(2^{s-1}-1)i}$.  So $\Psi_i$ extends to the whole group if and only if 
\begin{align*}
&i = (2^{s-1}-1)i &\mod 2^s \\
&\Leftrightarrow 0 = (2^{s-1}-2)i  &\mod 2^s\\
&\Leftrightarrow 2^{s-1} \divides i
\end{align*}

The $1$-dimensional irreducible representations of $SD_{2^{s+1}}$ are given by 
\begin{itemize}
\item the trivial representation,
\item $\Psi_{2^{s-1}}$ (extended to $SD_{2^{s+1}}$),
\item $\Psi_1$ (acting on $\langle y \rangle \cong \mu_2$ and extended to $SD_{2^{s+1}}$), 
\item $\Psi_{2^{s-1}} \otimes \Psi_1$.
\end{itemize}
The characters of these representations are given by
$$\begin{array}{c | c }
        & x^ay^b \\
 \hline 
\text{triv} & 1  \\
 \hline 
 \psi_{2^{s-1}} (\text{ acting on }\langle x \rangle)& (-1)^a \\
 \hline  
 \psi_1 (\text{ acting on }\langle y \rangle) & (-1)^b\\
 \hline 
\psi_{2^{s-1}} \otimes \psi_1 & (-1)^a(-1)^b
\end{array}$$

The $2$-dimensional irreducible representations of $SD_{2^{s+1}}$ are given by $\text{Ind}_{\mu_{2^s}}^{SD_{2^{s+1}}} \psi_i$ for $i \in \Z/2^s\Z$ with $2^{s-1} \nmid i$ and $\psi_i$ in distinct orbits under the action of $\mu_2$ on $\widehat{\mu_{2^s}}$. The faithful irreducible representations are those for which $2 \nmid i$. Let $\epsilon = \zeta_{2^s}$. The characters of these representations are given by

\begin{align*}
\chi_{i}(x^a) &= \frac{1}{2^s}\sum_{g \in SD_{2^{s+1}}, \text{ }g^{-1}x^ag \in \langle x \rangle}\psi_i(g^{-1}x^ag)\\
&=\frac{1}{2^s}(2^s(\psi_i(x^a) + 2^s\psi_i(x^{(2^{s-1}-1)a}))\\
&= \psi_i(x^a) + \psi_i(x^{(2^{s-1}-1)a})\\
&= (\epsilon)^{ai} + (\epsilon)^{(2^{s-1}-1)ai}\\
&= (\epsilon)^{ai} + (-1)^{ai}(\epsilon)^{-ai}
\end{align*}
and
\begin{align*}
\chi_{i}(x^ay) &= \frac{1}{2^s}\sum_{g \in SD_{2^{s+1}}, \text{ } g^{-1}x^ayg \in \langle x \rangle}\psi_i(g^{-1}x^ayg)\\
&= 0 \text{ since } g^{-1} x^ay g \notin \langle x \rangle \text{ for all } g \in SD_{2^{s+1}}\\
\end{align*}
So we get the following $2$-dimensional characters:
$$\begin{array}{c | c | c }
        & x^a & x^ay\\
 \hline 
\chi_i & \epsilon^{ai} + (-1)^{ai}\epsilon^{-ai} & 0\\
\end{array}$$

$(\text{Ind}_{\mu_{2^s}}^{SD_{2^{s+1}}} \Psi_i)(x)$ sends $x$ to $\epsilon^{i}$ in the first copy of $k$. And $xy = yx^{2^{s-1}-1}$, so $x$ sends $yx$ to $\epsilon^{(2^{s-1}-1)i} = (-\epsilon)^{-i}$ in the second copy of $k$.  So the matrix corresponding to $(\text{Ind}_{\mu_{2^s}}^{SD_{2^{s+1}}} \psi_i)(x)$ is given by $$\begin{pmatrix} \epsilon^{i} & 0 \\
0 & (-\epsilon)^{-i} \end{pmatrix}.$$
$(\text{Ind}_{\mu_{2^s}}^{SD_{2^{s+1}}} \Psi_i)(y)$ send $x$ to $x$ in the second copy of $k$. And $y^2 = 1$, so it sends $yx$ to $x$ in the first copy of $k$. So the matrix corresponding to $(\text{Ind}_{\mu_{2^s}}^{SD_{2^{s+1}}} \Psi_i)(y)$ is given by $$\begin{pmatrix} 0 & 1 \\
1 & 0 \end{pmatrix}.$$

Let
$$X = \begin{pmatrix} 0 & 1 \\
1 & \epsilon-\epsilon^{-1} \end{pmatrix}$$ 
and let 
$$Y = \begin{pmatrix} 1 & 0 \\
\epsilon - \epsilon^{-1} & -1 \end{pmatrix}.$$ 
Note that for $A = \begin{pmatrix} 1 & 1\\
\epsilon & -\epsilon^{-1} \end{pmatrix}$, 
\begin{align*}
A^{-1}XA &= \frac{1}{-\epsilon^{-1} - \epsilon}\begin{pmatrix} -\epsilon^{-1} & -1\\
-\epsilon & 1 \end{pmatrix} \begin{pmatrix} 0 & 1\\
1 & \epsilon - \epsilon^{-1} \end{pmatrix} \begin{pmatrix} 1 & 1\\
\epsilon & -\epsilon^{-1} \end{pmatrix}\\
&= \frac{1}{-\epsilon^{-1} - \epsilon}\begin{pmatrix} -1 & -\epsilon\\
1 & -\epsilon^{-1} \end{pmatrix} \begin{pmatrix} 1 & 1\\
\epsilon & -\epsilon^{-1} \end{pmatrix}\\
&= \frac{1}{-\epsilon^{-1} - \epsilon}\begin{pmatrix} -1 - \epsilon^2 & 0\\
0 & 1 + \epsilon^{-2}\end{pmatrix}\\
&= \begin{pmatrix} \epsilon & 0\\
0 & -\epsilon^{-1} \end{pmatrix}
\end{align*}
And
\begin{align*}
A^{-1}YA &= \frac{1}{-\epsilon^{-1} - \epsilon}\begin{pmatrix} -\epsilon^{-1} & -1\\
-\epsilon & 1 \end{pmatrix} \begin{pmatrix} 1 & 0\\
\epsilon - \epsilon^{-1} & -1 \end{pmatrix} \begin{pmatrix} 1 & 1\\
\epsilon & -\epsilon^{-1} \end{pmatrix}\\
&= \frac{1}{-\epsilon^{-1} - \epsilon}\begin{pmatrix} -\epsilon & 1\\
-\epsilon^{-1} & -1 \end{pmatrix} \begin{pmatrix} 1 & 1\\
\epsilon & -\epsilon^{-1} \end{pmatrix}\\
&= \frac{1}{-\epsilon^{-1} - \epsilon}\begin{pmatrix} 0 & -\epsilon-\epsilon^{-1}\\
-\epsilon^{-1}-\epsilon & 0\end{pmatrix}\\
&= \begin{pmatrix} 0 & 1\\
1 & 0 \end{pmatrix}
\end{align*}

So $\text{Ind}_{\mu_{2^s}}^{SD_{2^{s+1}}} \psi_i$ is isomorphic to $\lambda_i: SD_{2^{s+1}} \to GL_2(k(\epsilon - \epsilon^{-1}))$ defined by $\lambda_i(x) = X^i$ and $\lambda_i(y) = Y$. Note that these representations are defined over $k(\epsilon - \epsilon^{-1})$ and $\lambda_i$ is faithful if and only if $2 \nmid i$. So the faithful irreducible representations of $SD_{2^{s+1}}$ over $k_\text{sep}$ are given by $\lambda_i$ for $2 \nmid i$ (not all of these are distinct).

\subsubsection{Proof of Proposition \ref{edSemiDihedral}}
For the proof, we will need the following lemma.

\begin{lemma}\label{Gammacond} Let $s  > 2$ be an integer, let $\epsilon = \zeta_{2^{s}}$ in $k_\text{sep}$. Let $\Gamma = \text{Gal}(k(\epsilon)/k)$. Then 
$2^{s-1}-1 = -5^{2^{s-3}}$ and 
$$[k(\epsilon):k] = \begin{cases} 2[k(\epsilon - \epsilon^{-1}):k], &\Gamma = \langle 5^{2^{i}}, -1 \rangle \text{ or } \langle -5^{2^{i}} \rangle \text{ for } i = 0, \dots, s-3\\
&\text{ or equivalently }[2^{s-1}-1] \in \Gamma\\
[k(\epsilon - \epsilon^{-1}):k], &\Gamma =\langle -1 \rangle \text{ or } \langle 5^{2^i} \rangle \text{ for } i = 0,\dots, s-2\\
&\text{ or equivalently }[2^{s-1}-1] \notin \Gamma
\end{cases}.$$
\end{lemma}

\begin{proof}

Note that 
\begin{align*}
&(x-\epsilon)(x-\epsilon^{2^{s-1}-1})\\
&= (x-\epsilon)(x+\epsilon^{-1})\\
&= x^2 + \epsilon^{-1}x - \epsilon x -1\\
&= x^2 - (\epsilon-\epsilon^{-1})x -1\\
&\in k(\epsilon-\epsilon^{-1})[x]
\end{align*}
So the minimal polynomial for $\epsilon$ over $k(\epsilon-\epsilon^{-1})$ is either $x^2 - (\epsilon-\epsilon^{-1})x -1$ or $x-\epsilon$. 

If the minimal polynomial is $x^2 - (\epsilon-\epsilon^{-1})x -1$, then $\text{Gal}(k(\epsilon)/k(\epsilon-\epsilon^{-1}))$ is generated by raising $\epsilon$ to the $(2^{s-1}-1)$-th power. And so since $\text{Gal}(k(\epsilon)/k(\epsilon-\epsilon^{-1})) \subset \Gamma$, we must have $\gamma = [2^{s-1}-1] \in  \Gamma.$  So if $[2^{s-1}-1] \notin \Gamma$, then we can conclude that $[k(\epsilon):k(\epsilon-\epsilon^{-1})] = 1$. 

On the other hand, note that for $\gamma = 2^{s-1}-1$, $$\epsilon^\gamma - \epsilon^{-\gamma} = -\epsilon^{-1}+\epsilon = \epsilon - \epsilon^{-1}.$$ 
So $[2^{s-1}-1] \notin \text{Gal}(k(\epsilon-\epsilon^{-1})/k)$. Thus if $[2^{s-1}-1] \in \Gamma$, then we can conclude that 
$$\text{Gal}(k(\epsilon-\epsilon^{-1})/k) \neq \Gamma.$$
So if $[2^{s-1}-1] \in \Gamma$, then $k(\epsilon) \neq k(\epsilon-\epsilon^{-1})$, and hence $[k(\epsilon):k(\epsilon-\epsilon^{-1})] = 2$.

Note that for $s > 2$, 
$$(\Z/2^s\Z)^\times \cong \langle 5 \rangle \times \langle -1 \rangle \cong (\Z/2^{s-2}\Z) \times \Z/2\Z.$$
The subgroups of $(\Z/2^s\Z)^\times$ are the following
\begin{itemize}
\item $\langle 5^{2^{i}}, -1 \rangle$ for $i = 0, \dots, s-2$ (this is $\langle -1 \rangle$ for $i = s-2$)
\item $\langle 5^{2^{i}} \rangle$ for $i = 0, \dots, s-2$ (this is the trivial subgroup for $i = s-2$)
\item $\langle -5^{2^{i}} \rangle$ for $i = 0, \dots, s-3$.
\end{itemize}

Note that for $s > 2$, $2^{s-1}-1 \neq 1$ and $(2^{s-1}-1)^2 = 1$; so  $|2^{s-1}-1| = 2$. So $2^{s-1}-1$ is equal to one of the following
\begin{itemize}
\item $-1$
\item $5^{2^{s-3}}$
\item $-5^{2^{s-3}}$.
\end{itemize} 
We cannot have $2^{s-1} - 1 = -1 \mod 2^s$ since then we would have $2^{s-1} = 0 \mod 2^{s}$, a contradiction. 
So $2^{s-1}-1$ is equal to either $5^{2^{s-3}}$ or $-5^{2^{s-3}}$.

If $2^{s-1} - 1 = 5^{2^{s-3}} \mod 2^{s}$, then since $s > 2$, we must have $2^{s-1}-1 = 5^{2^{s-3}} = 1 \mod 4$. But for $s > 2$, $2^{s-1} - 1 = -1 \neq 1 \mod 4$. So we cannot have $2^{s-1}-1 = 5^{2^{s-3}}$. Thus we must have $2^{s-1}-1 = -5^{2^{s-3}}$. 

The subgroups which contain this element are the following
\begin{itemize}
\item $\langle 5^{2^{i}}, -1 \rangle$ for $i = 0, \dots, s-3$
\item $\langle -5^{2^{i}} \rangle$ for $i = 0, \dots, s-3$.
\end{itemize}
The subgroups which do not contain $2^{s-1}-1 = -5^{2^{s-3}}$ are $\langle -1 \rangle$ and $\langle 5^{2^i} \rangle$ for $i = 0,\dots, s-2$.
\end{proof}

\begin{proof}[Proof of Proposition \ref{edSemiDihedral}]

Let $G = SD_{2^{s+1}}$. By Mashke's theorem, since $\text{char } k \nmid |G|$,  $k[G]$ is semi-simple. Then by the Artin-Wedderburn theorem, we can write
$$k[G] = M_{n_1}(D_1) \times \dots \times M_{n_m}(D_m),$$
for division rings $D_1, \dots, D_m$ over $k$.

The centers $Z_i = Z(M_{n_i}(D_i))$ are given by the scalar matrices with entries in $Z(D_i)$. Since $Z(D_i)$ is an abelian division ring, it is a field.  Let $t_i = [Z_i:k]$.

Note that $D_i \otimes_{Z_i} \overline{Z_i}$ is a central simple $\overline{Z_i}$ algebra. And the only division algebra over $\overline{Z_i}$ is $\overline{Z_i}$. So by the Artin-Wedderburn theorem $D_i \otimes_{Z_i} \overline{Z_i} \cong M_{d_i}(\overline{Z_i})$. So $\dim_{\overline{Z_i}}(D_i \otimes_{Z_i} \overline{Z_i}) = d_i^2$. So 
$$\dim_{Z_i}(D_i) = \dim_{\overline{Z_i}}(D_i \otimes_{Z_i} \overline{Z_i}) = d_i^2.$$

Note that there is a simple module corresponding to  $M_{n_i}(D_i)$ given by 
$V_i = \{ \begin{pmatrix} v_1 & 0 & \dots & 0\end{pmatrix} : v_1 \in D_i\} \oplus \dots \oplus \{\begin{pmatrix} 0 & \dots & 0 & v_n\end{pmatrix} : v_n \in D_i\}$. 
The dimension of $V_i$ over $k$ is given by
$$\dim_k(V_i) = n_it_id_i^2.$$ 

Consider one of the $M_{n_i}(D_i)$ and let $n = n_i$, $D = D_i$, $d = d_i$, $Z = Z(D)$, $t = t_i = [Z:k]$. 
%Note that $D/Z = k(\{\chi(g) : g \in G\})/k$ 
Note that 
\begin{align*}
D \otimes_k Z &= D \otimes_Z (Z \otimes_k Z)\\
&= D \otimes_Z Z^{t}\\
&= (D \otimes_Z Z)^t\\
&= D^t 
\end{align*}
And so
\begin{align*}
M_n(D) \otimes_k Z &= M_n(D \otimes_k Z)\\
&= M_n(D^t)\\
&= M_n(D)^{t}
\end{align*}
So for $V$ a simple $M_n(D)$-module over $k$, we have 
$$V \otimes_k Z = U_1 \oplus \dots \oplus U_t,$$
for $U_j$ irreducible over $Z$, where $U_j$ is the simple module corresponding to the $i$th copy of $M_n(D)$. Note that
\begin{align*}
M_n(D) \otimes_Z k_\text{sep} &= M_n(D \otimes_Z k_{\text{sep}})\\
&= M_n(M_d(k_\text{sep}))\\
&= M_{nd}(k_\text{sep})
\end{align*}

So over $k_\text{sep}$, we have $(U_j)_{k_\text{sep}} = W_i^{\oplus d}$ for $W_i$ irreducible over $k_\text{sep}$. So since $\dim(U_j) = nd^2$, we must have $\dim(W_i) = nd$. If $V$ corresponds to a faithful representation, then one of the $W_i$ must be faithful and so will have dimension $2$. So we have $nd = 2$. So $$\dim(V) = 2dt = 2d[Z:k].$$
Note that $U_j = W_j^{\oplus d}$ is defined over $Z$, but $W_j$ is not necessarily defined over $Z$. 

Let $\epsilon = \zeta_{2^s} \in k_\text{sep}$. Recall that the faithful $2$-dimensional irreducible representations over $k_\text{sep}$ are isomorphic to $\lambda_i$ with $2 \nmid i$ given by $\lambda_i(x) = X^i$, $\lambda_i(y) = Y$ where 
$$X = \begin{pmatrix} 0 & 1 \\
1 & \epsilon-\epsilon^{-1} \end{pmatrix}, \text{ } Y = \begin{pmatrix} 1 & 0 \\
\epsilon - \epsilon^{-1} & -1 \end{pmatrix}$$
These irreducible representations are defined over $k(\epsilon - \epsilon^{-1})$. Also, since the character on $x$ is given by $\epsilon^i + (-1)^i\epsilon^{-i}$, we must have $\epsilon^i + (-1)^i\epsilon^{-i} \in Z$. And since $2 \nmid i$, $(-1)^i = -1$. So 
$$k(\epsilon^i -\epsilon^{-i}) \subset Z \subset k(\epsilon - \epsilon^{-1}).$$

Note that for $2 \nmid i$, $\epsilon^i$ is also a primitive $2^s$-th root of unity. Let $\zeta = \epsilon^i$. Then we can repeat the construction of the irreducible representations with $\zeta$ in the place of $\epsilon$ and call those irreducible representations $\varphi_j$. These representations are defined over $k(\zeta-\zeta^{-1}) = k(\epsilon^i - \epsilon^{-i})$ and their characters on $x$ will be $\zeta^j + (-1)^j\zeta^{-j} = \epsilon^{ij} + (-1)^j\epsilon^{-ij}$. Note that since both $\epsilon$ and $\epsilon^i$ are primitive $2^s$-th roots of unity, there exists $t$ (with $2 \nmid t$) such that $\epsilon = (\epsilon^i)^t = \epsilon^{it}$. Then consider $\varphi_t$. The character of $\varphi_t$ on $x$ is given by $\zeta^t + (-1)^{t}\zeta^{-t} = \epsilon^{it} - \epsilon^{-it} = \epsilon - \epsilon^{-1}$.  So since $\varphi_t$ is defined over $k(\zeta-\zeta^{-1}) = k(\epsilon^i - \epsilon^{-i})$, we must have
$$k(\epsilon-\epsilon^{-1}) \subset k(\epsilon^i - \epsilon^{-i}).$$
Therefore,
$$k(\epsilon^i -\epsilon^{-i}) = Z = k(\epsilon-\epsilon^{-1}).$$

So $W_j$ is defined over $Z = k(\epsilon-\epsilon^{-1})$. That is, there exists $S_j$ such that $(S_j)_{k_\text{sep}} = S_j \otimes_Z k_\text{sep} = W_j$.  Note that we can write
$$Z[G] = A_1 \times \dots \times A_m$$
for $A_i$ simple. The corresponding simple $A_i$-module is a direct sum of simple $(A_i)_{k_\text{sep}}$-modules. Since simple $(A_i)_{k_\text{sep}}$-modules and $(A_j)_{k_\text{sep}}$-modules for distinct $i$ and $j$ are pairwise non-isomorphic, the simple $A_i$-module and the simple $A_j$-module do not have common irreducible components over a separable closure.
So since $U_j$ and $S_j$ have a common irreducible component, $W_j$, over $k_\text{sep}$, they must be isomorphic. Therefore $W_j^{\oplus d} \cong U_j \cong S_j \cong W_j$ and hence $d = 1$. Thus 
$$\dim(V) = 2[Z:k] = 2[k(\epsilon-\epsilon^{-1}):k].$$

Thus $\ed_k(SD_{2^{s+1}},2) \geq 2[k(\epsilon - \epsilon^{-1}):k]$. And the map $\lambda_i: SD_{2^{s+1}} \to GL(k(\epsilon-\epsilon^{-1}))$ gives a faithful representation of $SD_{2^{s+1}}$ of dimension $2[k(\epsilon-\epsilon^{-1}):k]$. Therefore,
$$\ed_k(SD_{2^{s+1}},2) = \dim(V) = 2[Z:k] = 2[k(\epsilon - \epsilon^{-1}):k].$$

So by Lemma \ref{Gammacond},
\begin{align*} \ed_k(SD_{2^{s+1}},2) &= \begin{cases} [k(\epsilon):k], &\Gamma = \langle 5^{2^{i}}, -1 \rangle \text{ or } \langle -5^{2^{i}} \rangle \text{ for } i = 0, \dots, s-3\\
&\text{ or equivalently }[2^{s-1}-1] \in \Gamma\\
2[k(\epsilon):k], &\Gamma =\langle -1 \rangle \text{ or } \langle 5^{2^i} \rangle \text{ for } i = 0,\dots, s-2\\
&\text{ or equivalently }[2^{s-1}-1] \notin \Gamma
\end{cases}. \qedhere \end{align*}

\end{proof}

\section{\texorpdfstring{The Projective General Linear Groups - $q \equiv 1 \mod 4$}{The Projective General Linear Groups - q equiv 1 mod 4}}

\begin{theorem}\label{PGLn2} 
Let $p \neq 2$ be a prime and $q = p^r$. Let $k$ be a field with $\text{char } k \neq 2$.     Assume that $q \equiv 1 \pmod 4$, and let $s = v_2(q-1)$. Let $\epsilon = \zeta_{2^s}$ in $k_\text{sep}$ and let $\Gamma = \text{Gal}(k(\epsilon)/k)$. Then 
$$\ed_k(PGL_n(\F_q),2) = \begin{cases} 
\ed_k(GL_{n-1}(\F_q),2), &2 \nmid n\\
2^{v_2(n)}(n-2^{v_l(n)})[k(\epsilon):k)], &2 \divides n \text{ and } n \neq 2^t\\
 2^{2t-1}[k(\epsilon):k], &n=2^t, \text{ } [-1] \notin \Gamma\\
 2^{2t-2}[k(\epsilon):k], &n = 2^t,\text{ } [-1] \in \Gamma
 \end{cases}$$
\end{theorem}

By (\cite{Gr}, Proposition 1.1), 
$$|PGL_n(\F_q)| = \frac{|GL_n(\F_q)|}{q-1}.$$
So
\begin{align*}
|PGL_n(\F_q)|_2 &= \frac{|GL_n(\F_q)|_2}{2^{v_2(q-1)}} = 2^{s(n-1)} \cdot |S_{n}|_2\\
\end{align*}

\begin{lemma} For $q \equiv 1 \mod 4$, $P \in \syl_2(PGL_n(\F_q))$
$$P \cong (\mu_{2^s})^n/\{(x,x,\dots,x)\} \rtimes P_2(S_n).$$ 
where the action of $P_2(S_n)$ on $\mbf{a}$ is given by permuting the $a_i$.
\end{lemma}

\begin{proof}
$PGL_n(\F_q)$ is defined to be 
$$PGL_n(\F_q) = GL_n(\F_q)/Z(GL_n(\F_q)).$$
By Proposition \ref{GLsyl2}, the Sylow $2$-subgroups of $GL_n(\F_q)$ are isomorphic to $(\mu_{2^s})^n \rtimes P_2(S_n)$. The center of $GL_n(\F_q)$ is given by 
$$Z(GL_n(\F_q)) = \{x\text{Id}_n : x \in \F_q, x \neq 0\}.$$ 
So we see that a Sylow $2$-subgroup of $PGL_n(\F_q)$ will be isomorphic to
$$P = (\mu_{2^s})^n/\{(x,x,\dots,x)\} \rtimes P_2(S_n).$$
\end{proof}

The proofs in the cases $2 \nmid n$ and $2 \divides n$, $n \neq 2^t$ are identical to that in \cite{Kni2} for $l \neq 2$. So for the remainder of this section, we will assume that $n = 2^t$.

 \begin{definition} For $j = 1,2$, let $I_j$ denote the $j$th sub-block of $2^{k-1}$ entries in $\{1, \dots, 2^k\}$. Let $A_j = \sum_{i \in I_j} a_i$. \end{definition}
 
  \begin{lemma}\label{ZPGLn2} For $P \in \syl_n(PGL_2(\F_q))$ in the case $n = 2^t$ 
$$Z(P) \cong \langle (1, \dots, 1, \zeta_2, \dots, \zeta_2) \rangle \cong \mu_2.$$
\end{lemma}

\begin{proof}
The proof is identical to that in \cite{Kni2} for $l \neq 2$.
\end{proof}

Note that since $q \equiv 1 \mod 4$, we know that $s = v_2(q-1) > 1$.

\subsection{\texorpdfstring{The case $n = 2^t$, $[-1] \notin \Gamma$}{The case n = 2t, -1 notin Gamma}} 

For the proof of Theorem \ref{PGLn2} in the case $n = 2^t$, $[-1] \in \Gamma$, we will need the following lemmas.

\begin{lemma}\label{gammaaction} Suppose that $[-1] \notin \Gamma = \text{Gal}(k(\zeta_{2^s})/k)$.  $n = 2^t$, and $\mbf{a} \in (\Z/2^s\Z)^n$ with 
$$A_1 = -A_2 \text{ invertible}.$$
Then the orbit of $\mbf{a}$ under the action of $P_2(S_n)$ on $(\Z/2^s\Z)^n/\Gamma$ has the same size as the orbit of $\mbf{a}$ under the action of $P_2(S_n)$ on $(\Z/2^s\Z)^n$.
\end{lemma}

\begin{proof}

We will show that the orbits have the same size by showing that the stabilizers have the same size.  Let $\tau \in P_2(S_n)$ be in the stabilizer of $\mbf{a}$ in $(\Z/2^s\Z)^n/\Gamma$. Then there exists $\phi \in \Gamma$ such that $\tau(\mbf{a}) = \gamma_\phi \mbf{a}$. Let $\gamma = \gamma_\phi$. We want to show that we must then have $\gamma = 1$ (because this would mean that $\tau$ is in the stabilizer under the action of $P_2(S_n)$ on $(\Z/2^s\Z)^n$.

Note that $\tau$ permutes the $I_j$, and the permutation is either trivial or $I_1 \mapsto I_2 \mapsto I_1$. Suppose that $\tau(\mbf{a}) = \gamma \mbf{a}$. Then $A_1 = \gamma^2 A_1$, so $\gamma^2 = 1$. Thus $\gamma = \pm 1$. And so since $[-1] \notin \Gamma$, we must have $\gamma = 1$.

\end{proof}

\begin{lemma}\label{irrH2} Let $n = 2^t$, $s > 1$, and $\mbf{a} \in (\Z/2^s\Z)^n$ with 
$$A_1 = -A_2 \text{ invertible}.$$
Then
$$|\text{orbit}(\mbf{a})| \geq 2^{2t-1} 
$$ under the action of $P_2(S_n)$ on $(\Z/2^s\Z)^n$. 
\end{lemma}

\begin{proof}[Proof of Lemma \ref{irrH2}]

Since $A_1$ and $A_2$ are invertible, by Lemma \ref{irrH1}  we can conclude that that the orbit of $\mbf{a}$ under the action of $P_2(S_{2^{t-1}}) \times P_2(S_{2^{t-1}}) \subset P_2(S_n)$ has size at least $2^{t-1} \cdot 2^{t-1} = 2^{2t-2}$. 

Since $s > 1$, then $A_1 = -A_2 \neq A_2$. Then for $\tau$ the permutation $i \mapsto i+2^{k-1} \mod 2^{k}$, we get 
$$\tau(\mbf{a}) = (a_{2^{k-1}+1}, \dots, a_{2^{k}}, a_1, \dots, a_{2^{k-1}}).$$
And since $A_1 \neq A_2$, this is not equal to any of the $\sigma(\mbf{a})$ for $\sigma \in P_2(S_{2^{t-1}}) \times P_2(S_{2^{t-1}})$.  Thus the size of the orbit is at least $2^{2t-2} + 1$, and so it must be at least $2^{2t-1}$ since it must divide $|P_2(S_{2^{t}})|$ which is a power of $2$. 
\qedhere
\end{proof}

\begin{corollary}\label{irrH2cor}  Suppose that $[-1] \notin \Gamma = \text{Gal}(k(\zeta_{2^s})/k)$. Then for $n = 2^t$, $s > 1$, $\mbf{a} \in (\Z/2^s\Z)^n$ with 
$$A_1 = -A_2 \text{ invertible},$$ we can conclude that
$$|\text{orbit}(\mbf{a})| \geq 2^{2t-1}$$
under the action of $P_2(S_n)$ on $\text{Irr}((\mu_{2^s})^n)$.
\end{corollary}

\begin{proof}
By Lemma \ref{changepersp}, the orbit of $\Psi_{\mbf{a}}$ under the action of $P_2(S_n)$ on $\text{Irr}((\mu_{2^s})^n)$ has the same size as the orbit of $\mbf{a}$ under the action of $P_2(S_n)$ on $(\Z/2^s\Z)^n/\Gamma$. And if $[-1] \notin \Gamma$, then by Lemma \ref{gammaaction} this is the same as the orbit of $\mbf{a}$ under the action of $P_2(S_n)$ on $(\Z/2^s\Z)^n$. And by Lemma \ref{irrH2}, the orbit of $\mbf{a}$ under the action of $P_2(S_n)$ on $(\Z/2^s\Z)^n$ will have size at least $2^{2t-1}$. Therefore the orbit of $\Psi_\mbf{a}$ has size at least $2^{2t-1}$. 
\end{proof}

\subsubsection{\texorpdfstring{Proof for $n = 2^t$, $[-1] \notin \Gamma$}{Proof for n = 2t, -1 notin Gamma}}

\begin{proof}[Proof of Theorem \ref{PGLn2} for the case $n = 2^t$, $-1 \notin \Gamma$.]

Recall that $P \cong (\mu_{2^s})^n/\{(x,\dots,x)\} \rtimes P_2(S_n).$  Let $\rho$ be a faithful representation of $P$ of minimum dimension (and so it is also irreducible since the center has rank $1$.)   Let $S' = (\mu_{2^s})^n/\{(x,\dots,x)\}$.  By Clifford's Theorem (Theorem \ref{cliff}), $\rho|_{S'}$ decomposes into a direct sum of irreducibles in the following manner:
$$\rho|_{S'} \cong  \left( \oplus_{i=1}^c  \lambda_i \right)^{\oplus d}, \text{ for some } c, d,$$ 
with the $\lambda_i$ non-isomorphic, and $P_2(S_n)$ acts transitively on the $\lambda_i$, so the $\lambda_i$ have the same dimension and the number of $\lambda_i$, $c$, divides $|P_2(S_n)|$ (which is a power of $2$), so $c$ is a power of $2$. Also, since $\rho$ is faithful, it is non-trivial on $Z(P)$, thus one of the $\lambda_i$ must be non-trivial on $Z(P) \subset (\mu_{2^s})^n[2]$.   Without loss of generality assume the $\lambda_1$ is non-trivial on $Z(P)$.

Note that the irreducible representations of $S'$ are in bijection with irreducible representations of $(\mu_{2^s})^n$ which are trivial on $\{(x,\dots,x)\}$.  
By Lemma \ref{corrlemma}, the irreducible reprsentations of $(\mu_{2^s})^n$ are given by $\Psi_{\mbf{a}}$ with $\mbf{a} \in (\Z/2^s\Z)^n/\Gamma$, for $\Gamma = \text{Gal}(k(\zeta_{2^s})/k)$, and if $\Psi_{\mbf{a}}$ is non-trivial on $(\mu_{2^s})^n[2]$, then $\Psi_\mbf{a}$ has dimension $[k(\zeta_{2^s}):k]$. 
Since $\lambda_1$ is non-trivial on $Z(P) \subset (\mu_{2^s})^n[2]$, we must have $\dim(\lambda_1) = [k(\zeta_{2^s}):k]$, and so $\dim(\lambda_i) = [k(\zeta_{2^s}):k]$ for all $i$. And $\Psi_\mbf{a}$ will be trivial on $\{(x,\dots,x)\}$ if and only if $\sum_{i=1}^n a_i = 0$. So $\lambda_1 \cong \Psi_{\mbf{a}}$ for some $\mbf{a} \in (\Z/l^s\Z)^n/\Gamma$ with $\sum_{i=1}^n a_i$. 

Recall that $I_j$ denotes the $j$th sub-block of $2^{t-1}$ entries in $\{1, \dots, 2^t\}$. Since $\lambda_1$ is non-trivial on 
$$Z(P) = \langle (1, \dots, 1, \zeta_2, \dots, \zeta_2)\rangle,$$ we must have that 
$$0 \neq 2^{s-1}A_2.$$
Thus $2 \nmid A_2$ and so $A_2$ is invertible. And so since $0 = \sum_{i=1}^n a_i = A_1 + A_2$, we must have $A_1 = - A_2$ invertible. So by Corollary \ref{irrH2cor}, the orbit of $\lambda_1$ under the action of $P_2(S_{n})$ will have size at least $2^{2t-1}$.  So $c \geq 2^{2t-1}$. Thus 
$$\dim(\rho) \geq 2^{2t-1}[k(\zeta_{2^s}):k].$$

We can construct a faithful representation of this dimension in the following manner.  Let $\mbf{a} = (1, 0, \dots, 0, -1, 0, \dots, 0)$ where $-1$ is in the $2^{t-1}+1$-th index. And consider
$$\Psi_{\mbf{a}}:  S' \to GL(k(\zeta_{2^s})) = GL_d(k),$$  
where $d = [k(\zeta_{2^s}):k]$. The orbit of $\mbf{a}$ under the action of $P_2(S_n)$ has size $2^{2t-1}$. So the orbit of $\Psi_{\mbf{a}}$ under the action of $P_2(S_{n})$ on the irreducible representations (not isomorphism classes) of $S'$ has size $2^{2t-1}$. Let $\text{Stab}_{\mbf{a}}$ be the stabilizer of $\Psi_{\mbf{a}}$ in $P_2(S_{n})$. We can extend $\Psi_{\mbf{a}}$ to $S' \rtimes \text{Stab}_{\mbf{a}}$ by defining $\Psi_{\mbf{a}}(\mbf{b},\tau) = \tau_{\Psi_{\mbf{a}}}(\mbf{b}) = \Psi_{\mbf{a}}(\mbf{b})$ (since $\tau \in \text{Stab}_{\mbf{a}}$). Let $\rho = \text{Ind}_{S' \rtimes \text{Stab}_{\mbf{a}}}^P \Psi_{\mbf{a}}$. Then $\rho$ has dimension 
$$[P_2(S_{n}): \text{Stab}_{\mbf{a}}]\dim(\Psi_{\mbf{a}}) = 2^{2t-1}[k(\zeta_{2^s}):k]$$ and $\rho$ is non-trivial (and hence faithful) on $Z(P)$. So this is a faithful representation of $P$ of dimension  $2^{2t-1}[k(\zeta_{2^s}):k]$.
Thus we have shown that for $n = 2^t$, if $[-1] \notin \Gamma$, then
\begin{align*}
\ed_k(PGL_{n}(\F_q),2)) &= 2^{2t-1}[k(\zeta_{2^s}):k].
\qedhere
\end{align*}
\end{proof}

\subsection{\texorpdfstring{The case $n = 2^t$, $[-1] \in \Gamma$}{The case n = 2t, -1 in Gamma}} 

For $n = 2^t$, we have 
$$P = (\mu_{2^s})^{2^t}/\{(x,\dots,x)\} \rtimes P_2(S_{2^t}),$$
where $P_2(S_{2^t})$ acts on $(\mu_{2^s})^{2^t}/\{(x,\dots,x)\}$ by permuting the indices. 

% Since $P \cong S' \rtimes P_2(S_{2^t})$ and $S'$ is abelian, we can find the irreducible representation over $k_\text{sep}$ using Wigner-Mackey theory (see \cite{GV}). The irreducible representations of $S'$ are in bijection with irreducible representations of $(\mu_{2^s})^{2^t}$ which are trivial on $\{(x,\dots,x)\}$. So the irreducible representations of $S'$ are given by $\Psi_{\mbf{a}}$ with $\sum_{i=1}^{2^t} a_i = 0$.

 \begin{lemma}\label{irrH2'} Let $\Gamma = \text{Gal}(k(\zeta_{2^s})/k),$ $n = 2^t$, and $\mbf{a} \in (\Z/2^s\Z)^n$ with 
$$A_1, \text{ }A_2 \text{ invertible}.$$
Then
$$|\text{orbit}(\mbf{a})| \geq 2^{2t-2},$$
under the action of $P_l(S_n)$ on $(\Z/2^s\Z)^n/\Gamma$. 
\end{lemma}

\begin{proof}
Let $\tau \in P_2(S_n)$ be in the stabilizer of $\mbf{a}$ in $(\Z/2^s\Z)^n/\Gamma$. Then there exists $\phi \in \Gamma$ such that $\tau(\mbf{a}) = \gamma_\phi \mbf{a}$. Let $\gamma = \gamma_\phi$. If $\tau$ stabilizes the $A_j$, then we have $A_j = \gamma A_j$ and since the $A_j$ are invertible, we can conclude that $\gamma = 1$. So the orbit under the action of $P_2(S_{2^{t-1}})^2$ on $(\Z/2^s\Z)^n/\Gamma$ is the same as the orbit under the action on $(\Z/2^s\Z)^n$, which is equal to the product of the orbit of $\mbf{a}$ under the action of $P_2(S_{2^{t-1}})$ on $I_1$ and the orbit of $\mbf{a}$ under the action of $P_2(S_{2^{t-1}})$ on $I_2$. So by Lemma \ref{irrH1}, we can conclude that the orbit has size at least $2^{2t-2}$ under the action of $P_2(S_{2^{t-1}})^2$. So the orbit under the action of $P_2(S_n)$ on $(\Z/2^s\Z)^n/\Gamma$ has size at least $2^{2t-2}$. 

\end{proof}

\begin{corollary}\label{irrH2cor'}  For $n = 2^t$, $\mbf{a} \in (\Z/2^s\Z)^n$ with 
$$A_1, \text{ }A_2 \text{ invertible},$$ we can conclude that
$$|\text{orbit}(\Psi_\mbf{a})| \geq 2^{2t-2}$$
under the action of $P_2(S_n)$ on $\text{Irr}((\mu_{2^s})^n)$.
\end{corollary}

\begin{proof}
By Lemma \ref{changepersp}, the orbit of $\Psi_{\mbf{a}}$ under the action of $P_2(S_n)$ on $\text{Irr}((\mu_{2^s})^n)$ has the same size as the orbit of $\mbf{a}$ under the action of $P_2(S_n)$ on $(\Z/2^s\Z)^n/\Gamma$. And by Lemma \ref{irrH2'}, the orbit of $\mbf{a}$ under the action of $P_2(S_n)$ on $(\Z/2^s\Z)^n/\Gamma$ will have size at least $2^{2t-2}$. Therefore the orbit of $\Psi_\mbf{a}$ has size at least $2^{2t-2}$. 
\end{proof}

%Let $S' = (\mu_{2^s})^{2^t}/\{(x,\dots,x)\}.$ Over $k_\text{sep}$, the irreducible representations of $S'$ are given by $\psi_{\mbf{a}}$ with $\sum_{i=1}^{2^t} a_i = 0$.  Let $L_\mbf{a} = \{\sigma \in P_2(S_n) : \sigma(\mbf{a}) = \mbf{a}\}$. Then we can extend $\psi_{\mbf{a}}$ to $S' \rtimes L_\mbf{a}$ by defining $\psi_{\mbf{a}}(\mbf{x},\sigma) = \psi_{\mbf{a}}(\mbf{x})$ for $\sigma \in L_\mbf{a}$.  For $\lambda$ an irreducible representation of $L_{\mbf{a}}$, we can extend $\lambda$ to $S' \rtimes L_{\mbf{a}}$ by defining $\lambda(x,\sigma) = \lambda(\sigma)$. By Wigner-Mackey theory, the irreducible representations of $P$ over $k_\text{sep}$ are given by 
%$$\theta_{\mbf{a},\lambda} = \text{Ind}_{S' \rtimes L_{\mbf{a}}}^P (\psi_{\mbf{a}} \otimes \lambda).$$
% The dimension is given by 
% $$\dim(\theta_{\mbf{a},\lambda}) = \frac{|P_2(S_{2^t})|}{|L_\mbf{a}|}\dim(\lambda).$$

\subsubsection{\texorpdfstring{Proof for $n = 2^t$, $[-1] \in \Gamma$}{Proof for n = 2t, -1 in Gamma}}

\begin{lemma}\label{faithrep} Let $\epsilon = \zeta_{2^s} \in k_\text{sep}$, $\Gamma = \text{Gal}(k(\epsilon)/k)$, and $n = 2^t$. Assume that $[-1] \in \Gamma$. Then there exists a faithful representation of $P = (\mu_{2^s})^n/\{(x,\dots,x)\} \rtimes P_2(S_n)$ of dimension $2^{2t-2}[k(\epsilon):k]$.
\end{lemma}

\begin{proof}

Let $S' = (\mu_{2^s})^n/\{(x,\dots,x)\}$. Let $\mbf{a} = (1,0,\dots,0,-1,0,\dots,0)$. And consider
$$\mbf{a}:  S' \to k(\epsilon)^\times$$
defined by $\mbf{a}(\mbf{x}) = \prod_{i=1}^n (x_i)^{a_i} = x_1(x_{2^{t-1}+1})^{-1}$.  
 Let $L_\mbf{a} = \{\sigma \in P_2(S_n) : \sigma(\mbf{a}) = \mbf{a}\}$. Then we can extend the character $\mbf{a}$ to $S' \rtimes L_\mbf{a}$ by defining $\mbf{a}(\mbf{x},\sigma) = \mbf{a}(\mbf{x})$ for $\sigma \in L_\mbf{a}$.  Note that $L_\mbf{a} = \{ \sigma \in P_2(S_n) : \sigma(1) = 1, \sigma(2^{t-1}+1) = 2^{t-1}+1\}$, and $[P_2(S_n):L_\mbf{a}] = 2^{2t-1}$. Let 
$$\rho = \text{Ind}_{S' \rtimes L_\mbf{a}}^P \mbf{a}: P \to GL_{2^{2t-1}}(k(\epsilon)).$$
Let $V = k(\epsilon)$ be the ($1$-dimensional) $k(\epsilon)$-module corresponding to $\mbf{a}$, and let $H = S' \rtimes L_\mbf{a}$. Then the induced module corresponding to $\rho$ is
$$V' = \text{Ind}_{H}^P V = k(\epsilon)[P] \otimes_{k(\epsilon)[H]} V.$$

I want to show that $\rho$ can be defined over $k(\epsilon+\epsilon^{-1})$, i.e. there is a $P$-invariant sub-module of $V'$ with coefficients in $k(\epsilon+\epsilon^{-1})$..

Let $U \subset P_2(S_n)$ be a set of representatives of $P/H$. Then a basis for $V'$ is given by $\{\sigma \otimes 1 : \sigma \in U\}$. The action of $h \in H$ on $\sigma \otimes 1$ is given by 
\begin{align*}
h \cdot (\sigma \otimes 1) &= h\sigma \otimes 1\\
&= \sigma(\sigma^{-1}h\sigma) \otimes 1\\
&= \sigma \otimes \mbf{a}(\sigma^{-1}h\sigma)\\
&= \sigma \otimes \mbf{a}(\sigma^{-1}(h))
\end{align*} 
Every element of $P$ can be written uniquely as $\sigma h$ for some $\sigma \in U, \text{ } h \in H$.  For $\tau \in P_2(S_n)$, write $\tau \sigma = \sigma' h$. Since $\tau \sigma \in P_2(S_n)$ and $\sigma' \in P_2(S_n)$, we can conclude that $h \in P_2(S_n) \cap H = L_\mbf{a}$ and so $\mbf{a}(h) = 1$.  Then the action of $\tau$ on $\sigma \otimes 1$ is given by 
\begin{align*} 
\tau \cdot (\sigma \otimes 1) &= \tau \sigma \otimes 1\\
&= \sigma' h \otimes 1\\
&= \sigma' \otimes \mbf{a}(h)\\
&= \sigma' \otimes 1
\end{align*}

 Note that the representatives of $P/H$ are in bijection with $\mbf{b}$ in the orbit of $\mbf{a}$ under the action of $P_2(S_n)$ via $\sigma H \mapsto \sigma(\mbf{a})$. For $\mbf{b}$ in the orbit, let $u_\mbf{b} = \sigma \otimes 1$ where $\sigma(\mbf{a}) = \mbf{b}$.  Then a basis for $V'$ is given by 
$$\{u_\mbf{b} : \mbf{b} \in \text{orb}(a)\},$$
The action of $x \in S'$ is given by 
%$$x \cdot u_\mbf{b} = \mbf{b}(x) u_\mbf{b},$$ 
\begin{align*}
x \cdot u_\mbf{b} &= x \cdot (\sigma \otimes 1)\\
&= \sigma \otimes \mbf{a}(\sigma^{-1}(x))\\
&= \sigma \otimes \mbf{b}(x)\\
&= \mbf{b}(x) (\sigma \otimes 1)\\
&= \mbf{b}(x) u_\mbf{b}
\end{align*}
and the action of $\tau \in P_2(S_n)$ is given by 
\begin{align*}
\tau \cdot u_\mbf{b}  &= \tau \cdot (\sigma \otimes 1), &\text{where } \sigma(\mbf{a}) = \mbf{b}\\
&= \sigma' \otimes 1, &\text{where } \tau \sigma = \sigma' h\\
&= u_{\sigma'(\mbf{a})}\\
&= u_{\sigma'h(\mbf{a})}, &\text{ since } h \in L_\mbf{a}\\
&= u_{\tau \sigma(\mbf{a})}\\
&= u_{\tau(\mbf{b})}.
\end{align*} 
So the action of $x\tau \in P$ is given by 
$$(x\tau) \cdot (u_\mbf{b}) = x \cdot u_{\tau(\mbf{b})} = \tau (\mbf{b})(x) \cdot u_{\tau \mbf{b}} = \mbf{b}(\tau^{-1} x) \cdot u_{\tau (\mbf{b})}.$$

Note that the orbit of $\mbf{a} = (1,0,\dots,0,-1,0,\dots,0)$ is given by $\mbf{b} \in (\mu_{2^s})^n$ such that $b_{i_1} = \pm 1 \text{ for some } i_1 \in I_1 = \{1,\dots,2^{t-1}\} \text{ and } b_{i_2} = -b_{i_1} \text{ for some } i_2 \in I_2 = \{2^{t-1}+1,\dots, 2^t\} \text{ and } b_i = 0$ for all other indices. So the elements in the orbit come in pairs $\{\mbf{b},-\mbf{b}\}$ with $b_{i_1} = 1$ for $i_1 \in I_1$ and $b_{i_2} = -1$ for $i_2 \in I_2$. Thus a basis for $V'$ is given by 
\begin{align*}
\{u_\mbf{b} : \mbf{b} \in \text{orb}(a)\} = \bigcup_{\mbf{b} \in \text{orb}(a) \text{ with } b_{i} = 1 \text{ for some } i \in I_1} \{u_\mbf{b}, u_{-\mbf{b}}\} 
\end{align*}
For $\mbf{b} \in \text{orb}(a) \text{ with } b_{i} = 1 \text{ for some } i \in I_1$, let
\begin{align*}
v_\mbf{b} &= \epsilon^{-1} u_\mbf{b} - \epsilon u_{-\mbf{b}}\\
v_{-\mbf{b}} &= -u_{\mbf{b}} + u_{-\mbf{b}}
\end{align*}
I claim that 
$$W = \bigoplus_{\mbf{b} \in \text{orb}(a) \text{ with } b_{i} = 1 \text{ for some } i \in I_1} \left(k(\epsilon+\epsilon^{-1})v_{\mbf{b}} \oplus k(\epsilon+\epsilon^{-1})v_{-\mbf{b}}\right)$$
is a $P$-invariant sub-module of $V'$. The action of $\tau \in P_2(S_n)$  on $v_\mbf{b}$ is given by
\begin{align*}
\tau \cdot v_\mbf{b} &= \tau \cdot (\epsilon^{-1} u_\mbf{b} - \epsilon u_{-\mbf{b}})\\
&= \epsilon^{-1} u_{\tau (\mbf{b})} - \epsilon u_{-\tau(\mbf{b})}\\
&= \begin{cases} v_{\tau(\mbf{b})}, &\tau(\mbf{b}) \text{ has } b_i = 1 \text{ for some } i \in I_1\\
(\epsilon+\epsilon^{-1})v_{\tau(\mbf{b})} + v_{-\tau(\mbf{b})}, &\tau(\mbf{b}) \text{ has } b_i = -1 \text{ for some } i \in I_1\end{cases}\\
&\in W,
\end{align*}
and the action $\tau \in P_2(S_n)$  on $v_{-\mbf{b}}$ is given by
\begin{align*}
\tau \cdot v_{-\mbf{b}} &= \tau \cdot (-u_\mbf{b} + u_{-\mbf{b}})\\
&= -u_{\tau(\mbf{b})} + u_{-\tau(\mbf{b})}\\
&= \begin{cases} v_{-\tau(\mbf{b})}, &\tau(\mbf{b}) \text{ has } b_i = 1 \text{ for some } i \in I_1\\
-v_{\tau(\mbf{b})}, &\tau(\mbf{b}) \text{ has } b_i = -1 \text{ for some } i \in I_1\end{cases}\\
&\in W.
\end{align*}
%For $x \in S'$, note that $\mbf{b}(x) = \epsilon^j$ for some $j$.  
%If we considered $W = k v_\mbf{b} \oplus k v_{-\mbf{b}}$, then for $x$ such that $\mbf{b}(x) = \epsilon$, $x \cdot v_\mbf{b} = u_\mbf{b} - u_{-\mbf{b}} \in W$, but for $x$ such that $\mbf{b}(x) = \epsilon^{-1}$, $x \cdot v_{\mbf{b}} = \epsilon^{-2}u_{\mbf{b}} - \epsilon^{-2}u_{-\mbf{b}} = (\epsilon+\epsilon^{-1})v_\mbf{b} + v_{-\mbf{b}} \notin W$.  So it would not suffice to just consider $x$ with $\mbf{x} = \epsilon$? 
%
%
Note that 
$$G = \{g \in S' : \text{ one entry of } g \text{ is } \epsilon \text{ and all other entries are }1\}$$ is a generating set of $S'$.  So it suffices to consider the action of $g \in G$.  Note that for $g \in G$, $\mbf{b}(g) \in \{\epsilon,\epsilon^{-1},1\}$. The action of $g$ on $v_\mbf{b}$ is given by 
\begin{align*}
g  \cdot v_\mbf{b} &= g \cdot (\epsilon^{-1} u_\mbf{b} - \epsilon u_{-\mbf{b}})\\
&= \epsilon^{-1} \mbf{b}(g) u_{\mbf{b}} - \epsilon (-\mbf{b})(g)  u_{-\mbf{b}}\\
&= \begin{cases}
u_{\mbf{b}} - u_{-\mbf{b}}, &\mbf{b}(g) = \epsilon\\
\epsilon^{-2} u_{\mbf{b}} - \epsilon^{2}u_{-\mbf{b}}, &\mbf{b}(g) = \epsilon^{-1}\\
\epsilon^{-1}u_\mbf{b} - \epsilon u_{-\mbf{b}}, &\mbf{b}(g) = 1 \end{cases}\\
&= \begin{cases}
-v_{-\mbf{b}}, &\mbf{b}(g)=\epsilon\\
(\epsilon+\epsilon^{-1})v_\mbf{b} + v_{-\mbf{b}}, &\mbf{b}(g) = \epsilon^{-1}\\
v_{\mbf{b}}, &\mbf{b}(g) = 1
\end{cases}\\
&\in W,
\end{align*}
and the action of $g$ on $v_{-\mbf{b}}$ is given by 
\begin{align*}
g \cdot v_{-\mbf{b}} &= g \cdot (-u_\mbf{b} + u_{-\mbf{b}})\\
&= -(\mbf{b}(g))  u_{\mbf{b}} + (-\mbf{b})(g)  u_{-\mbf{b}}\\
&= \begin{cases}
-\epsilon u_{\mbf{b}} + \epsilon^{-1} u_{-\mbf{b}}, &\mbf{b}(g) = \epsilon\\
-\epsilon^{-1} u_{\mbf{b}} + \epsilon u_{-\mbf{b}}, &\mbf{b}(g) = \epsilon^{-1}\\
-u_{\mbf{b}} + u_{-\mbf{b}}, &\mbf{b}(g) = 1 \end{cases}\\
&= \begin{cases}
(\epsilon+\epsilon^{-1})v_{-\mbf{b}} + v_{\mbf{b}}, &\mbf{b}(g) = \epsilon\\
-v_{\mbf{b}}, &\mbf{b}(g) = \epsilon^{-1}\\
v_{-\mbf{b}}, &\mbf{b}(g) = 1
\end{cases}\\
&\in W
\end{align*}

Therefore, $W = \bigoplus_{\mbf{b} \in \text{orb}(a) \text{ with } b_{i} = 1 \text{ for some } i \in I_1} \left(k(\epsilon+\epsilon^{-1}) v_{\mbf{b}}\oplus k(\epsilon+\epsilon^{-1})v_{-\mbf{b}}\right)$ is a $P$-invariant sub-module of $V'$. Hence $W$ corresponds to a representation $\rho': P \to GL_{2^{2t-1}}(k(\epsilon+\epsilon^{-1}))$ of $P$ of dimension $2^{2t-1}$ over $k(\epsilon+\epsilon^{-1})$.

Let $d = [k(\epsilon+\epsilon^{-1}):k]$. Then we can embed $k(\epsilon+\epsilon^{-1})$ in $GL_d(k)$ And thus we have an embedding $GL_{2^{2t-1}}(k(\epsilon+\epsilon^{-1})) \hookrightarrow GL_{2^{2t-1}d}(k)$. Thus we have 
$$\rho': P \to GL_{2^{2t-1}}(k(\epsilon+\epsilon^{-1})) \hookrightarrow GL_{2^{2t-1}d}(k).$$
Therefore, $\rho'$ is a faithful representation over $k$ of dimension 
$$2^{2t-1}d = 2^{2t-1}[k(\epsilon+\epsilon^{-1}):k] = 2^{2t-2}[k(\epsilon):k]$$ 
for $[-1] \in \Gamma$.

\end{proof}

\begin{proof}[Proof of Theorem \ref{PGLn2} for the case $n = 2^t$, $-1 \in \Gamma$.]

Recall that $P \cong (\mu_{2^s})^n/\{(x,\dots,x)\} \rtimes P_2(S_n).$  Let $\rho$ be a faithful representation of $P$ of minimum dimension (and so it is also irreducible since the center has rank $1$.)   Let $S' = (\mu_{2^s})^n/\{(x,\dots,x)\}$.  By Clifford's Theorem (Theorem \ref{cliff}), $\rho|_{S'}$ decomposes into a direct sum of irreducibles in the following manner:
$$\rho|_{S'} \cong  \left( \oplus_{i=1}^c  \lambda_i \right)^{\oplus d}, \text{ for some } c, d,$$ 
with the $\lambda_i$ non-isomorphic, and $P_2(S_n)$ acts transitively on the $\lambda_i$, so the $\lambda_i$ have the same dimension and the number of $\lambda_i$, $c$, divides $|P_2(S_n)|$ (which is a power of $2$), so $c$ is a power of $2$. Also, since $\rho$ is faithful, it is non-trivial on $Z(P)$, thus one of the $\lambda_i$ must be non-trivial on $Z(P) \subset (\mu_{2^s})^n[2]$.   Without loss of generality assume the $\lambda_1$ is non-trivial on $Z(P)$.

Note that the irreducible representations of $S'$ are in bijection with irreducible representations of $(\mu_{2^s})^n$ which are trivial on $\{(x,\dots,x)\}$.  
By Lemma \ref{corrlemma}, the irreducible reprsentations of $(\mu_{2^s})^n$ are given by $\Psi_{\mbf{a}}$ with $\mbf{a} \in (\Z/2^s\Z)^n/\Gamma$, for $\Gamma = \text{Gal}(k(\zeta_{2^s})/k)$, and if $\Psi_{\mbf{a}}$ is non-trivial on $(\mu_{2^s})^n[2]$, then $\Psi_\mbf{a}$ has dimension $[k(\zeta_{2^s}):k]$. 
Since $\lambda_1$ is non-trivial on $Z(P) \subset (\mu_{2^s})^n[2]$, we must have $\dim(\lambda_1) = [k(\zeta_{2^s}):k]$, and so $\dim(\lambda_i) = [k(\zeta_{2^s}):k]$ for all $i$. And $\Psi_\mbf{a}$ will be trivial on $\{(x,\dots,x)\}$ if and only if $\sum_{i=1}^n a_i = 0$. So $\lambda_1 \cong \Psi_{\mbf{a}}$ for some $\mbf{a} \in (\Z/2^s\Z)^n/\Gamma$ with $\sum_{i=1}^n a_i$. 

Recall that $I_j$ denotes the $j$th sub-block of $2^{t-1}$ entries in $\{1, \dots, 2^t\}$. Since $\lambda_i$ is non-trivial on 
$$Z(P) = \langle (1, \dots, 1, \zeta_2, \dots, \zeta_2)\rangle,$$ we must have that 
$$0 \neq 2^{s-1}A_2.$$
Thus $2 \nmid A_2$ and so $A_2$ is invertible. And so since $0 = \sum_{i=1}^n a_i = A_1 + A_2$, we must have $A_1 = - A_2$ invertible.  So by Corollary \ref{irrH2cor'}, the orbit of $\lambda_1$ under the action of $P_2(S_{2^{t}})$ will have size at least $2^{2t-2}$.  So $c \geq 2^{2t-2}$. Thus 
$$\dim(\rho) \geq 2^{2t-2}[k(\zeta_{2^s}):k].$$

And since we are assuming that $[-1] \in \Gamma$, by Lemma \ref{faithrep}, there exists a faithful representation of $P$ of dimension $2^{2t-2}[k(\epsilon):k]$, Therefore in the case $[-1] \in \Gamma$, 
$$\ed_k(PGL_{2^t}(\F_q),2) = 2^{2t-2}[k(\zeta_{2^s}):k].$$

\end{proof}
\section{\texorpdfstring{The Projective General Linear Groups - $q \equiv 3 \mod 4$}{The Projective General Linear Groups - q equiv 3 mod 4}}

\begin{theorem}\label{PGLn2'} 
Let $p \neq 2$ be a prime and $q = p^r$. Let $k$ be a field with $\text{char } k \neq 2$.     Assume that $q \equiv 3 \pmod 4$, and let $s = v_2(q+1)+1$. Let $\epsilon = \zeta_{2^s}$ in $k_\text{sep}$ and let $\Gamma = \text{Gal}(k(\epsilon)/k)$. Then 
$$\ed_k(PGL_n(\F_q),2) = \begin{cases} 
\ed_k(GL_{n-1}(\F_q),2), &2 \nmid n\\
2[k(\zeta_{2^{s-1}}):k], &n=2, \text{ } [-1] \notin \text{Gal}(k(\zeta_{2^{s-1}})/k)\\
[k(\zeta_{2^{s-1}}):k], &n = 2, \text{ } [-1] \in \text{Gal}(k(\zeta_{2^{s-1}})/k)\\
2^{2+v_2(m)}(m-2^{v_2(m)})[k(\epsilon):k], &n = 2m, \text{ } m \neq 2^t,\text{ } [2^{s-1}-1] \notin \Gamma\\
2^{1+v_2(m)}(m-2^{v_2(m)})[k(\epsilon:k], &n=2m, \text{ } m \neq 2^t, \text{ } [2^{s-1}-1] \in \Gamma\\
2^{2t}[k(\epsilon):k], &n = 2m, \text{ } m = 2^t, \text{ } [2^{s-1}-1] \notin \Gamma\\
2^{2t-1}[k(\epsilon):k], &n = 2m, \text{ } m = 2^t, \text{ } [2^{s-1}-1] \in \Gamma
 \end{cases}$$
\end{theorem}

Since $q \equiv 3 \mod 4$, we can write $q = 3 + 4a = 1 + 2(1+4a)$ for some integer $a$, so $v_2(q-1) = 1$ and $v_2(q+1) \geq 2$. Hence $s = v_2(q+1)+1 > 2$. Since $v_2(q-1) = 1$, we have
$$|PGL_n(\F_q)|_2 = \frac{|GL_n(\F_q)|_2}{2}  = \begin{cases} 2^{v_2(m!)-1}\cdot (2^{s+1})^m, &n = 2m\\
2^{v_2(m!)} \cdot (2^{s+1})^m, &n = 2m+1
\end{cases}.$$

\subsection{\texorpdfstring{A Sylow $2$-subgroup}{A Sylow 2-subgroup}}

\begin{lemma} For $q \equiv 3 \mod 4$, $P \in \syl_2(PGL_n(\F_q))$
$$P \cong \begin{cases}
(SD_{2^{s+1}})^m/\langle(-1,\dots,-1)\rangle \rtimes P_2(S_m), &n = 2m\\
\left((SD_{2^{s+1}})^m \times \Z/2\Z\right)/\langle(-1,\dots,-1)\rangle\rtimes P_2(S_m), &n = 2m+1\end{cases}.$$ 
where the action of $P_2(S_n)$ on $\mbf{a}$ is given by permuting the $a_i$.
\end{lemma}

\begin{proof}
$PGL_n(\F_q)$ is defined to be 
$$PGL_n(\F_q) = GL_n(\F_q)/Z(GL_n(\F_q)).$$
By the proof of Proposition \ref{GLsyl2'} in the section on $GL_n(\F_q)$, the Sylow $2$-subgroups of $GL_n(\F_q)$ are isomorphic to $$\begin{cases}
(SD_{2^{s+1}})^m \rtimes P_2(S_m), &n = 2m\\
\left((SD_{2^{s+1}})^m \rtimes P_2(S_m)\right) \times \langle Z \rangle, &n = 2m+1\end{cases},$$
where $SD_{2^{s+1}} = \langle X, Y \rangle$ for $X = \begin{pmatrix} 0 & 1\\
1 & \epsilon + \epsilon^{q} \end{pmatrix}$ and $Y = \begin{pmatrix} 1 & 0\\
\epsilon + \epsilon^q & -1 \end{pmatrix}$ with $\epsilon = \zeta_{2^s} \in k_\text{sep}$ and $Z = \text{diag}(1,\dots,1,-1)$ . The center of $GL_n(\F_q)$ is given by 
$$Z(GL_n(\F_q)) = \{x\text{Id}_n : x \in \F_q, x \neq 0\}.$$ 
Note that the only diagonal matrices in $SD_{2^{s+1}}$ are given by $-\text{Id} = X^{2^{s-1}}$.
So we see that a Sylow $2$-subgroup of $PGL_n(\F_q)$ will be isomorphic to
$$P = \begin{cases}
(SD_{2^{s+1}})^m/\langle(-1,\dots,-1)\rangle \rtimes P_2(S_m), &n = 2m\\
\left((SD_{2^{s+1}})^m \times \Z/2\Z\right)/\langle(-1,\dots,-1)\rangle\rtimes P_2(S_m), &n = 2m+1\end{cases},$$
where 
$SD_{2^{s+1}} = \langle x,y : x^{2^s} = y^2 = 1, yxy = x^{2^{s-1}-1} = -x^{-1}\rangle$ and $-1 = x^{2^{s-1}}$.
\end{proof}

\begin{corollary}For $q \equiv 3 \mod 4$, $P \in \syl_2(PGL_n(\F_q))$, $n = 2m$,
$$P \cong (\mu_{2^s})^m/\langle(-1,\dots,-1)\rangle \rtimes \left((\Z/2\Z)^m \rtimes P_2(S_m)\right),$$ 
where the action of $\sigma \in P_2(S_m)$ on $\mbf{b} \in (\Z/2\Z)^m$ is given by permuting the indices and the action of $(\mbf{b},\sigma) \in (\Z/2\Z)^m \rtimes P_2(S_m)$ on $\mbf{c} \in (\mu_{2^s})^m/\langle (-1,\dots,-1)\rangle$ is given $$(\mbf{b},\sigma)(\mbf{c}) = \mbf{x} : x_i =  (c_{\sigma(i)})^{(2^{s-1}-1)^{b_i}}.$$
\end{corollary}

\begin{proof}
Note that 
$$(SD_{2^{s+1}})^m = (\mu_{2^s} \rtimes \mu_2)^m = (\mu_{2^s})^m \rtimes (\Z/2\Z)^m,$$ where the action of $(\Z/2\Z)^m$ on $(\mu_{2^s})^m$ is the product of the action of $\Z/2\Z$ on $\mu_{2^s}$ in $SD_{2^{s+1}}$, which is given by $y(x) = x^{((2^{s-1}-1)^y)}$.

So
\begin{align*} 
P &= (SD_{2^{s+1}})^m/\langle (-1,\dots,-1) \rangle \rtimes P_2(S_m)\\
&\cong \left((\mu_{2^s})^m/\langle (-1,\dots,-1)\rangle \rtimes (\Z/2\Z)^m\right) \rtimes P_2(S_m)
\end{align*}
Note that 
\begin{align*}
((\mbf{a},\mbf{b}),\sigma)\cdot ((\mbf{c},\mbf{d}),\tau) &=  ((\mbf{a},\mbf{b}) \cdot (\sigma(\mbf{c},\mbf{d})),\sigma\tau)\\
&= ((\mbf{a},\mbf{b}) \cdot (\sigma(\mbf{c}),\sigma(\mbf{d})),\sigma\tau)\\
&= (\mbf{a} \cdot \mbf{b}(\sigma(\mbf{c})),\mbf{b} \cdot \sigma(\mbf{d}),\sigma\tau)
\end{align*} 
So $((\mbf{a},\mbf{b}),\sigma) \mapsto (\mbf{a},(\mbf{b},\sigma))$ is an ismorphism
\begin{align*}
\left((\mu_{2^s})^m/\langle (-1,\dots,-1)\rangle \rtimes (\Z/2\Z)^m\right) \rtimes P_2(S_m) &\cong (\mu_{2^s})^m/\langle (-1,\dots-1)\rangle \rtimes \left((\Z/2\Z)^m \rtimes P_2(S_m)\right)
\end{align*}
where the action of $\sigma \in P_2(S_m)$ on $\mbf{b} \in (\Z/2\Z)^m$ is given by permuting the indices and the action of $(\mbf{b},\sigma) \in (\Z/2\Z)^m \rtimes P_2(S_m)$ on $\mbf{c} \in (\mu_{2^s})^m/\langle (-1,\dots,-1)\rangle$ is given by $(\mbf{b},\sigma)(\mbf{c}) = \mbf{b}(\sigma(\mbf{c}))$. In other words,
 $$(\mbf{b},\sigma)(\mbf{c}) = \mbf{x} : x_i =  (c_{\sigma(i)})^{(2^{s-1}-1)^{b_i}}.$$
\end{proof}

The proof when $2 \nmid n$ (i.e. $n = 2m+1$) is simple:

\begin{proof}[Proof of Theorem \ref{PGLn2'} for the case $2 \nmid n$]
For $n = 2m+1$, 
$$P = \left((SD_{2^{s+1}})^m \times \Z/2\Z\right)/\langle(-1,\dots,-1)\rangle\rtimes P_2(S_m)$$
Let $P' = (SD_{2^{s+1}})^m \rtimes P_2(S_m)$. We can construct an isomorphism from $P'$ to $P$ by sending $(\mbf{b},\sigma)$ to $(\mbf{b},1,\sigma)$. Therefore, the Sylow $2$-subgroups of $PGL_{2m+1}(\F_q)$ are isomorphic to Sylow $2$-subgroups of $GL_{2m}(\F_q)$. Thus 
\begin{align*} \ed_k(PGL_{2m+1}(\F_q),2) &= \ed_k(GL_{2m}(\F_q),2) = \ed_k(GL_{n-1}(\F_q),2). \qedhere
\end{align*}
\end{proof}

%\subsection{\texorpdfstring{The case $n = 2$, $\Gamma$ trivial}{The case n = 2, Gamma trivial}}
%
%\begin{proof}[Proof of Theorem \ref{PGLn2'} in the case $n=2$, $\Gamma$ trivial]
%For $n = 2$, we have 
%\begin{align*} 
%P &= SD_{2^{s+1}}/\langle -1 \rangle
%\end{align*}
%The center is given by 
%$$\langle x^{2^{s-2}} \rangle/\langle -1 \rangle.$$
%So since the center has rank $1$, minimal dimensional faithful representations of $P$ are also irreducible. The irreducible representations of $P = SD_{2^{s+1}}/\langle -1 \rangle$ are in bijection with the irreducible representations of $SD_{2^{s+1}}$ that are trivial on $\langle -1 \rangle$.  
%
%If $\Gamma$ is trivial, then $k(\zeta_{2^s}) = k$.
%Then by Wigner-Mackey theory, the irreducible representations of $SD_{2^{s+1}}$ over $k$ are given by four $1$-dimensional irreducible representations which are not faithful and $2$-dimensional irreducible representations. Therefore, if $\Gamma$ is trivial then 
%$$\ed_k(PGL_2(\F_q),2) = 2.$$
%
%\end{proof}

\subsection{\texorpdfstring{The case $n = 2$}{The case n = 2}}

For $n = 2$, we have 
\begin{align*} 
P &= SD_{2^{s+1}}/\langle -1 \rangle\\
&= \langle x,y : x^{2^{s}} = y^2 = 1, yxy = x^{2^{s-1}-1} = -x^{-1}\rangle/ \langle x^{2^{s-1}} \rangle\\
&= \langle w = x\langle x^{2^{s-1}} \rangle, z = y\langle x^{2^{s-1}} \rangle : w^{2^{s-1}} = z^2 = 1, zwz = x^{2^{s-1}-1}\langle x^{2^{s-1}}\rangle = w^{-1}\rangle\\
&\cong D_{2^s}.
\end{align*}

So it suffices to calculate the essential dimension of the dihedral groups of order a power of $2$: $D_{2^s}$. Since $s > 2$, it suffices to consider $D_{2^{s+1}}$ for $s > 1$. That is, it suffices to prove the following proposition:

\begin{proposition}\label{edDihedral} Let $s > 1$, $\epsilon = \zeta_{2^s}$ in $k_\text{sep}$. Let $\Gamma = \text{Gal}(k(\epsilon)/k)$. Then
\begin{align*}
\ed_k(D_{2^{s+1}}) &= 2[k(\epsilon + \epsilon^{-1}):k]\\
&= \begin{cases} [k(\epsilon):k], &\Gamma = \langle 5^{2^{i}}, -1 \rangle \text{ for }i = 0, \dots, s-2\\
&\text{ or equivalently } [-1] \in \Gamma\\
2[k(\epsilon):k], &\Gamma = \langle -5^{2^i} \rangle \text{ for } i = 1, \dots, s-2\\
&\qquad \text{ or } \langle 5^{2^{i}} \rangle \text{ for } i = 0, \dots, s-3 \\
 &\text{or equivalently } [-1] \notin \Gamma
\end{cases}
\end{align*}
\end{proposition}
The second equality in Proposition \ref{edDihedral} comes from the following Lemma \ref{Gamma}.

\subsubsection{\texorpdfstring{Character table of $D_{2^{s+1}}$}{Character table of D2s+1}}

We will first find the character table of $D_{2^{s+1}}$.  Since $D_{2^{s+1}} = \langle x \rangle \rtimes \langle y \rangle \cong \mu_{2^s} \rtimes \mu_{2}$, we can find the irreducible representations over $k_\text{sep}$ using Wigner-Mackey theory (see \cite{GV}). The distinct irreducible representations of $\mu_{2^s}$ are given by $\Psi_i$ for $i \in \Z/2^s\Z$ and they extend to the whole group if and only if $y(\Psi_i) = \Psi_i$. And 
$y(\Psi_i)(x) = \Psi_i(yxy) = \Psi_i(x^{-1}) = \text{multiplication by } x^{-i}$.  So $\Psi_i$ extends to the whole group if and only if 
\begin{align*}
&i = -i &\mod 2^s \\
&\Leftrightarrow 2^{s-1} \divides i
\end{align*}

The $1$-dimensional irreducible representations of $SD_{2^{s+1}}$ are given by 
\begin{itemize}
\item the trivial representation,
\item $\Psi_{2^{s-1}}$ (extended to $D_{2^{s+1}}$),
\item $\Psi_1$ (acting on $\langle y \rangle \cong \mu_2$ and extended to $D_{2^{s+1}}$), 
\item $\Psi_{2^{s-1}} \otimes \Psi_1$.
\end{itemize}
The characters of these representations are given by
$$\begin{array}{c | c }
        & x^ay^b \\
 \hline 
\text{triv} & 1  \\
 \hline 
 \psi_{2^{s-1}} (\text{ acting on }\langle x \rangle)& (-1)^a \\
 \hline  
 \psi_1 (\text{ acting on }\langle y \rangle) & (-1)^b\\
 \hline 
\psi_{2^{s-1}} \otimes \psi_1 & (-1)^a(-1)^b
\end{array}$$

The $2$-dimensional irreducible representations of $D_{2^{s+1}}$ are given by $\text{Ind}_{\mu_{2^s}}^{D_{2^{s+1}}} \psi_i$ for $i \in \Z/2^s\Z$ with $2^{s-1} \nmid i$ and $\psi_i$ in distinct orbits under the action of $\mu_2$ on $\widehat{\mu_2^s}$. The faithful irreducible representations are those for which $2 \nmid i$. Let $\epsilon = \zeta_{2^s}$. The characters of these representations are given by

\begin{align*}
\chi_{i}(x^a) &= \frac{1}{2^s}\sum_{g \in D_{2^{s+1}}, \text{ }g^{-1}x^ag \in \langle x \rangle}\psi_i(g^{-1}x^ag)\\
&=\frac{1}{2^s}(2^s(\psi_i(x^a) + 2^s\psi_i(x^{-a}))\\
&= \psi_i(x^a) + \psi_i(x^{-a})\\
&= \epsilon^{ai} + \epsilon^{-ai}
\end{align*}
and
\begin{align*}
\chi_{i}(x^ay) &= \frac{1}{2^s}\sum_{g \in D_{2^{s+1}}, \text{ } g^{-1}x^ayg \in \langle x \rangle}\psi_i(g^{-1}x^ayg)\\
&= 0 \text{ since } g^{-1} x^ay g \notin \langle x \rangle \text{ for all } g \in D_{2^{s+1}}\\
\end{align*}
So we get the following $2$-dimensional characters:
$$\begin{array}{c | c | c }
        & x^a & x^ay\\
 \hline 
\chi_i & \epsilon^{ai} + \epsilon^{-ai} & 0\\
\end{array}$$

$(\text{Ind}_{\mu_{2^s}}^{D_{2^{s+1}}} \Psi_i)(x)$ sends $x$ to $\epsilon^{i}$ in the first copy of $k$. And $xy = yx^{-1}$, so $x$ sends $yx$ to $\epsilon^{-i}$ in the second copy of $k$.  So the matrix corresponding to $(\text{Ind}_{\mu_{2^s}}^{D_{2^{s+1}}} \psi_i)(x)$ is given by $$\begin{pmatrix} \epsilon^{i} & 0 \\
0 & \epsilon^{-i} \end{pmatrix}.$$
$(\text{Ind}_{\mu_{2^s}}^{D_{2^{s+1}}} \Psi_i)(y)$ send $x$ to $x$ in the second copy of $k$. And $y^2 = 1$, so it sends $yx$ to $x$ in the first copy of $k$. So the matrix corresponding to $(\text{Ind}_{\mu_{2^s}}^{D_{2^{s+1}}} \Psi_i)(y)$ is given by $$\begin{pmatrix} 0 & 1 \\
1 & 0 \end{pmatrix}.$$

Let
$$X = \begin{pmatrix} 0 & 1 \\
-1 & \epsilon+\epsilon^{-1} \end{pmatrix}$$ 
and let 
$$Y = \begin{pmatrix} 1 & 0 \\
\epsilon+\epsilon^{-1} & -1 \end{pmatrix}.$$ 
Note that for $A = \begin{pmatrix} 1 & 1\\
\epsilon & \epsilon^{-1} \end{pmatrix}$, 
\begin{align*}
A^{-1}XA &= \frac{1}{\epsilon^{-1} - \epsilon}\begin{pmatrix} \epsilon^{-1} & -1\\
-\epsilon & 1 \end{pmatrix} \begin{pmatrix} 0 & 1\\
-1 & \epsilon + \epsilon^{-1} \end{pmatrix} \begin{pmatrix} 1 & 1\\
\epsilon & \epsilon^{-1} \end{pmatrix}\\
&= \frac{1}{\epsilon^{-1} - \epsilon}\begin{pmatrix} 1 & -\epsilon\\
-1 & \epsilon^{-1} \end{pmatrix} \begin{pmatrix} 1 & 1\\
\epsilon & \epsilon^{-1} \end{pmatrix}\\
&= \frac{1}{\epsilon^{-1} - \epsilon}\begin{pmatrix} 1 - \epsilon^2 & 0\\
0 & -1 + \epsilon^{-2}\end{pmatrix}\\
&= \begin{pmatrix} \epsilon & 0\\
0 & \epsilon^{-1} \end{pmatrix}
\end{align*}
And
\begin{align*}
A^{-1}YA &= \frac{1}{\epsilon^{-1} - \epsilon}\begin{pmatrix} \epsilon^{-1} & -1\\
-\epsilon & 1 \end{pmatrix} \begin{pmatrix} 1 & 0\\
\epsilon + \epsilon^{-1} & -1 \end{pmatrix} \begin{pmatrix} 1 & 1\\
\epsilon & \epsilon^{-1} \end{pmatrix}\\
&= \frac{1}{\epsilon^{-1} - \epsilon}\begin{pmatrix} -\epsilon & 1\\
\epsilon^{-1} & -1 \end{pmatrix} \begin{pmatrix} 1 & 1\\
\epsilon & \epsilon^{-1} \end{pmatrix}\\
&= \frac{1}{\epsilon^{-1} - \epsilon}\begin{pmatrix} 0 & -\epsilon+\epsilon^{-1}\\
\epsilon^{-1}-\epsilon & 0\end{pmatrix}\\
&= \begin{pmatrix} 0 & 1\\
1 & 0 \end{pmatrix}
\end{align*}

So $\text{Ind}_{\mu_{2^s}}^{D_{2^{s+1}}} \psi_i$ is isomorphic to $\lambda_i: D_{2^{s+1}} \to GL_2(k(\epsilon + \epsilon^{-1}))$ defined by $\lambda_i(x) = X^i$ and $\lambda_i(y) = Y$. Note that these representations are defined over $k(\epsilon + \epsilon^{-1})$ and $\lambda_i$ is faithful if and only if $2 \nmid i$. So the faithful irreducible representations of $D_{2^{s+1}}$ over $k_\text{sep}$ are given by $\lambda_i$ for $2 \nmid i$ (not all of these are distinct).

\subsubsection{Proof of Proposition \ref{edDihedral}}

For the proof, we will need the following lemma.

\begin{lemma}\label{Zform} For $\epsilon = \zeta_{2^s}$, $2 \nmid i$, 
$$k(\{\epsilon^{ai} + \epsilon^{-ai} : a \in \Z/2^{s}\Z\}) = k(\epsilon+\epsilon^{-1}).$$
\end{lemma}

\begin{proof}

For $2 \nmid i$, $\epsilon^i$ is also a primitive $2^{s}$-th root of unity. So there exists $a$ such that $\epsilon = (\epsilon^{i})^a = \epsilon^{ai}$. So $\epsilon+\epsilon^{-1} \in \{\epsilon^{ai}+\epsilon^{-ai} : a \in \Z/2^{s-1}\Z\}$.  

I claim that for any $n \in \Z$, $\epsilon^{n} + \epsilon^{-n} \in k(\epsilon+\epsilon^{-1})$. I will prove this by induction.

Base case: $\epsilon+\epsilon^{-1} \in k(\epsilon+\epsilon^{-1})$. (More interesting: $\epsilon^2 + \epsilon^{-2} = (\epsilon+\epsilon^{-1})^2 - 2$.)

Induction step: Assume that $\epsilon^m + \epsilon^{-m} \in k(\epsilon+\epsilon^{-1})$ for all $m < n$. Note that by the binomial theorem, 
\begin{align*}
(\epsilon+\epsilon^{-1})^{n} &= \sum_{j=0}^n \begin{pmatrix} n \\
j \end{pmatrix} \epsilon^j(\epsilon^{-1})^{n-j}\\
&= \epsilon^n + \epsilon^{-n} +\sum_{j=1}^{n-1} \begin{pmatrix} n \\
j \end{pmatrix} \epsilon^j(\epsilon^{-1})^{n-j}\\
&= \epsilon^n + \epsilon^{-n} + \sum_{j=1}^{\lfloor \frac{n-1}{2} \rfloor} \begin{pmatrix} n \\
j \end{pmatrix} (\epsilon^{2j-n} + \epsilon^{n-2j}) + \begin{cases} 0, &n \text{ odd}\\
\begin{pmatrix} n \\
\frac{n}{2} \end{pmatrix}, &n \text{ even} \end{cases}
\end{align*}
And 
\begin{align*}
&\sum_{j=1}^{\lfloor \frac{n-1}{2} \rfloor} \begin{pmatrix} n \\
j \end{pmatrix} (\epsilon^{2j-n} + \epsilon^{n-2j}) \in k(\epsilon+\epsilon^{-1}),\\
&\text{ by the induction hypothesis since } j \leq \lfloor \frac{n-1}{2} \rfloor < n
\end{align*}
Therefore
\begin{align*}
\epsilon^n+\epsilon^{-n} &= (\epsilon+\epsilon^{-1})^n - \sum_{j=1}^{\lfloor \frac{n-1}{2} \rfloor} \begin{pmatrix} n \\
j \end{pmatrix} (\epsilon^{2j-n} + \epsilon^{n-2j}) - \begin{cases} 0, &n \text{ odd}\\
\begin{pmatrix} n \\
\frac{n}{2} \end{pmatrix}, &n \text{ even} \end{cases}\\
&\in k(\epsilon+\epsilon^{-1})
\end{align*}
Thus for $2 \nmid i$, $\{\epsilon^{ai}+\epsilon^{-ai} : a \in \Z/2^{s-1}\Z\} = k(\epsilon+\epsilon^{-1})$.  

\end{proof}

\begin{proof}[Proof of Proposition \ref{edDihedral}]

Let $G = D_{2^{s+1}}$. By Mashke's theorem, since $\text{char } k \nmid |G|$,  $k[G]$ is semi-simple. Then by the Artin-Wedderburn theorem, we can write
$$k[G] = M_{n_1}(D_1) \times \dots \times M_{n_m}(D_m),$$
for division rings $D_1, \dots, D_m$ over $k$.

The centers $Z_i = Z(M_{n_i}(D_i))$ are given by the scalar matrices with entries in $Z(D_i)$. Since $Z(D_i)$ is an abelian division ring, it is a field.  Let $t_i = [Z_i:k]$.

Note that $D_i \otimes_{Z_i} \overline{Z_i}$ is a central simple $\overline{Z_i}$ algebra. And the only division algebra over $\overline{Z_i}$ is $\overline{Z_i}$. So by the Artin-Wedderburn theorem $D_i \otimes_{Z_i} \overline{Z_i} \cong M_{d_i}(\overline{Z_i})$. So $\dim_{\overline{Z_i}}(D_i \otimes_{Z_i} \overline{Z_i}) = d_i^2$. So 
$$\dim_{Z_i}(D_i) = \dim_{\overline{Z_i}}(D_i \otimes_{Z_i} \overline{Z_i}) = d_i^2.$$

Note that there is a simple module corresponding to  $M_{n_i}(D_i)$ given by 
$V_i = \{ \begin{pmatrix} v_1 & 0 & \dots & 0\end{pmatrix} : v_1 \in D_i\} \oplus \dots \oplus \{\begin{pmatrix} 0 & \dots & 0 & v_n\end{pmatrix} : v_n \in D_i\}$. 
The dimension of $V_i$ over $k$ is given by
$$\dim_k(V_i) = n_it_id_i^2.$$

Consider one of the $M_{n_i}(D_i)$ and let $n = n_i$, $D = D_i$, $d = d_i$, $Z = Z(D)$, $t = t_i = [Z:k]$. Note that 
\begin{align*}
D \otimes_k Z &= D \otimes_Z (Z \otimes_k Z)\\
&= D \otimes_Z Z^{t}\\
&= (D \otimes_Z Z)^t\\
&= D^t 
\end{align*}
And so
\begin{align*}
M_n(D) \otimes_k Z &= M_n(D \otimes_k Z)\\
&= M_n(D^t)\\
&= M_n(D)^{t}
\end{align*}
So for $V$ a simple $M_n(D)$-module over $k$, we have 
$$V \otimes_k Z = U_1 \oplus \dots \oplus U_t,$$
for $U_j$ irreducible over $Z$, where $U_j$ is the simple module corresponding to the $i$th copy of $M_n(D)$. Note that
\begin{align*}
M_n(D) \otimes_k k_\text{sep} &= M_n(D \otimes_k k_{\text{sep}})\\
&= M_n(M_d(k_\text{sep}))\\
&= M_{nd}(k_\text{sep})
\end{align*}

So over $k_\text{sep}$, we have $(U_j)_{k_\text{sep}} = W_i^{\oplus d}$ for $W_i$ irreducible over $k_\text{sep}$. So since $\dim(U_i) = nd^2$, we must have $\dim(W_i) = nd$. If $V$ corresponds to a faithful representation, then one of the $W_i$ must be faithful and so will have dimension $2$. So we have $nd = 2$. So $$\dim(V) = 2dt = 2d[Z:k].$$
Note that $U_j = W_j^{\oplus d}$ is defined over $Z$, but $W_j$ is not necessarily defined over $Z$. 

Let $\epsilon = \zeta_{2^s} \in k_\text{sep}$. Recall that the faithful $2$-dimensional irreducible representations over $k_\text{sep}$ are isomorphic to $\lambda_i$ with $2 \nmid i$ given by $\lambda_i(x) = X^i$, $\lambda_i(y) = Y$ where 
$$X = \begin{pmatrix} 0 & 1 \\
-1 & \epsilon+\epsilon^{-1} \end{pmatrix}, \text{ } Y = \begin{pmatrix} 1 & 0 \\
\epsilon + \epsilon^{-1} & -1 \end{pmatrix}$$
These irreducible representations are defined over $k(\epsilon + \epsilon^{-1})$. Also, since the character on $x^a$ is given by $\epsilon^{ai} + \epsilon^{-ai}$, we must have $\epsilon^{ai} + \epsilon^{-ai} \in Z$ for all $a$. So
$$k(\{\epsilon^{ai} + \epsilon^{-ai} : a \in \Z/2^x\Z\}) \subset Z \subset k(\epsilon+\epsilon^{-1}).$$
And since $2 \nmid i$, by Lemma \ref{Zform} $k(\{\epsilon^{ai} + \epsilon^{-ai} : a \in \Z/2^s\Z\}) = k(\epsilon+\epsilon^{-1}).$ Therefore,
$$Z = k(\epsilon+\epsilon^{-1}).$$

So $W_j$ is defined over $Z = k(\epsilon+\epsilon^{-1})$. That is, there exists $S_j$ such that $(S_j)_{k_\text{sep}} = S_j \otimes_Z k_\text{sep} = W_j$.  Note that we can write
$$Z[G] = A_1 \times \dots \times A_m$$
for $A_i$ simple. The corresponding simple $A_i$-module is a direct sum of simple $(A_i)_{k_\text{sep}}$-modules. Since simple $(A_i)_{k_\text{sep}}$-modules and $(A_j)_{k_\text{sep}}$-modules for distinct $i$ and $j$ are pairwise non-isomorphic, the simple $A_i$-module and the simple $A_j$-module do not have common irreducible components over a separable closure.
So since $U_j$ and $S_j$ have a common irreducible component, $W_j$, over $k_\text{sep}$, they must be isomorphic. Therefore $W_j^{\oplus d} \cong U_j \cong S_j \cong W_j$ and hence $d = 1$. Thus 
$$\dim(V) = 2[Z:k] = 2[k(\epsilon+\epsilon^{-1}):k].$$

Thus $\ed_k(SD_{2^{s+1}},2) \geq 2[k(\epsilon + \epsilon^{-1}):k]$. And the map $\lambda_i: D_{2^{s+1}} \to GL(k(\epsilon+\epsilon^{-1}))$ gives a faithful representation of $D_{2^{s+1}}$ of dimension $2[k(\epsilon+\epsilon^{-1}):k]$. Therefore,
$$\ed_k(D_{2^{s+1}},2) = \dim(V) = 2[Z:k] = 2[k(\epsilon + \epsilon^{-1}):k].$$
So by Lemma \ref{Gamma},
\begin{align*} \ed_k(D_{2^{s+1}},2) &= \begin{cases} [k(\epsilon):k], &\Gamma = \langle 5^{2^{i}}, -1 \rangle \text{ for }i = 0, \dots, s-2\\
&\text{ or equivalently } [-1] \in \Gamma\\
2[k(\epsilon):k], &\Gamma = \langle -5^{2^i} \rangle \text{ for } i = 1, \dots, s-2\\
&\qquad \text{ or } \langle 5^{2^{i}} \rangle \text{ for } i = 0, \dots, s-3 \\
 &\text{or equivalently } [-1] \notin \Gamma
\end{cases}. \qedhere \end{align*}

\end{proof}

For the remainder of this section, we will assume that $n > 2$.

\subsection{\texorpdfstring{The centers in the case $n > 2$}{The centers in the case n > 2}}

\begin{definition} For $m \neq 2^t$, $I_j$ be the orbits of $\{1, \dots, m\}$ under the action of $P_2(S_m)$.  There are $\xi_2(m)$ such orbits. Let $i_j$ denote the smallest index in $I_j$. For each $j$, note that $|I_j| = 2^{k}$ for some $k$. Let $k_j$ be such that $|I_j| = 2^{k_j}$. Let $A_j = \sum_{i \in I_j} a_i.$ 
\end{definition}

  \begin{definition}  For $m = 2^t$, then for $j = 1,2$, let $I_j$ denote the $j$th sub-block of $2^{t-1}$ entries in $\{1, \dots, 2^t\}$, let $k_j = 2^{t-1}$, and let $A_j = \sum_{i \in I_j} a_i$. \end{definition}

\begin{lemma}\label{ZPGLn2'} For $q \equiv 3 \mod 4$, $n = 2m$, $m \neq 2^t$, $P \in \syl_2(PGL_n(\F_q))$, let $\mbf{b}^j$ be given by 
$$(\mbf{b}^j)_i = \begin{cases} -1, &i \in I_j\\
1, &i \notin I_j \end{cases}.$$ 
Then 
$$Z(P)[2] = \langle \mbf{b}^j \rangle_{j=1}^{\xi_2(m)}/\{(-1,\dots,-1)\} \cong \langle \mbf{b}^j \rangle_{j=1}^{\xi_2(m)-1} \cong (\mu_{2})^{\xi_2(m)-1}.$$
\end{lemma}

\begin{proof}
Let $P =  (SD_{2^{s+1}})^m/\langle(-1,\dots,-1)\rangle \rtimes P_2(S_m)$. Fix $(\mbf{b}, \tau) \in P$. Then for $(\mbf{b}', \tau') \in P$,
$$(\mbf{b},\tau)(\mbf{b}',\tau') = (\mbf{b} \tau(\mbf{b}'), \tau\tau') \text{ and } (\mbf{b}', \tau')(\mbf{b},\tau) = (\mbf{b}' \tau'(\mbf{b}), \tau'\tau).$$
Thus $(\mbf{b},\tau)$ is in the center if and only if $\tau \in Z(P_2(S_{m}))$ and 
$$\mbf{b}\tau(\mbf{b}') = \mbf{b}'\tau'(\mbf{b}) \mod \langle (-1,\dots,-1)\rangle$$ for all $\mbf{b}',\tau'$. Choosing $\tau' = \Id$, we see we must have $\mbf{b}\tau(\mbf{b}') = \mbf{b}'\mbf{b} \mod \langle (-1,\dots,-1)\rangle$.  Thus we must have $\tau(\mbf{b}') = \mbf{b}' \mod \langle(-1,\dots,-1)\rangle$ for all $\mbf{b}'$ in $(SD_{2^s})^m/\langle (-1,\dots,-1)\rangle$. For any $\tau \neq \text{Id}$, we can choose a $\mbf{b}'$ for which this is not satisfied, so we can conclude that we must have $\tau = \Id$. We also need $\tau(\mbf{b}) = \mbf{b} \mod \langle(-1,\dots,-1)\rangle$ for all $\tau \in P_2(S_m)$. 

Since $n \neq 2^t$, for each $i, i'$ in the same $I_j$, there exists $\tau \in P_2(S_n)$ that sends $i$ to $i'$ and fixes some other index. Since there is an index that is fixed by $\tau$, in order for $\tau(\mbf{b})$ to equal $\mbf{b}\mbf{x}$ for $\mbf{x} = (x,\dots,x)$, we must have $x = 1$ and so $\tau(\mbf{b}) = \mbf{b}$. So $b_i = b_{i'}$ for $i,i'$ in the same $I_j$. Let $\mbf{b}^j$ be given by 
$$(\mbf{b}^j)_i = \begin{cases} -1, &i \in I_j\\
1, &i \notin I_j \end{cases}.$$ 
Then 
\begin{align*} Z(P)[2] = \langle \mbf{b}^j \rangle_{j=1}^{\xi_l(n)}/\langle(-1,\dots,-1)\rangle \cong \langle \mbf{b}^j \rangle_{j=1}^{\xi_2(m)-1} &\cong (\mu_{2})^{\xi_2(m)-1}. \qedhere
\end{align*}
\end{proof}

 \begin{lemma}\label{ZPGLn2''} For $P \in \syl_n(PGL_l(\F_q))$ in the case $l \divides q - 1$, $n = 2m$, $m = 2^t$,
$$Z(P) \cong \langle (1,\dots,1,-1,\dots,-1) \rangle \cong \mu_2.$$
\end{lemma}

\begin{proof}
Let $P =  (SD_{2^{s+1}})^m/\langle(-1,\dots,-1)\rangle \rtimes P_2(S_m)$. Just as in the case $n = 2m$, $m \neq 2^t$, in order for $(\mbf{b}, \tau)$ to be in the center we must have $\tau = \text{Id}$ and $\tau(\mbf{b}) = \mbf{b} \mod \{(-1,\dots,-1)\}$ for all $\tau \in P_2(S_m)$.  

Note that for each $i, i'$ in the same $I_j$, there exists $\tau \in P_2(S_m)$ that sends $i$ to $i'$ and fixes some other index. Since there is an index that is fixed by $\tau$, in order for $\tau(\mbf{b})$ to equal $\mbf{b}\mbf{x}$ for $\mbf{x} = (x,x,\dots,x)$, we must have $x = 1$ and so $\tau(\mbf{b}) = \mbf{b}$. So $b_1 = \dots = b_{2^{t-1}}$, $b_{2^{t-1}+1} = \dots = b_{2^t}$.  If we consider the last generator, $\sigma_1^{t}$, we see that we must have $b_{i+2^{t-1}} = b_{i}x$ for some fixed $x = (-1)^{a}$.  Thus $\mbf{b}$ must be of the form
$$\mbf{b} = (b, \dots, b, b(-1)^a, \dots b(-1)^a).$$
In $PGL_{2^t}(\F_q)$, the set of all elements of this form is a cyclic group of order $2$ generated by
$$\mbf{b} = (1,\dots,1,-1,\dots,-1).$$
So we have 
\begin{align*} Z(P) &= \langle (1,\dots,1,-1,\dots,-1) \rangle \cong \mu_2.
\qedhere
\end{align*}
\end{proof}

\subsection{\texorpdfstring{The case $n = 2m$, $n > 2$}{The case n = 2m, n > 2}}

For the proof of Theorem \ref{PGLn2'} in the cases with $n = 2m$, $n > 2$, we will need the following lemmas.

\begin{lemma}\label{changepersp'} Let  $\Gamma = \text{Gal}(k(\zeta_{2^s})/k) \hookrightarrow (\Z/2^s\Z)^\times$ and the action of $\phi \in \Gamma$ be given by scalar multiplication by $\gamma_\phi$. Then the orbit of $\Psi_{\mbf{a}}$ under the action of $(\Z/2\Z)^m \rtimes P_2(S_m)$ on $\text{Irr}((\mu_{2^s})^m)$ will have the same size as the orbit of $\mbf{a}$ under the action of $(\Z/2\Z)^m \rtimes P_2(S_m)$ on $(\Z/2^s\Z)^m/\Gamma$ given by $\mbf{a} \mapsto \sigma(\mbf{b}(\mbf{a})),$ where $\mbf{b}(\mbf{a})_i = (a_i(2^{s-1}-1)^{b_i})$.
\end{lemma}

\begin{proof}
Let $S = (\mu_{2^s})^m$. By Lemma \ref{corrlemma}, the irreducible representations of $S$ are in bijection with $\mbf{a} \in (\Z/2^s\Z)^n/\Gamma$, where the action of $\phi \in \Gamma$ is given by scalar multiplication by $\gamma_\phi$. The bijection is given by $\mbf{a} \in (\Z/2^s\Z)^n/\Gamma \mapsto \Psi_{\mbf{a}}.$ 

Let $G = (\mu_{2^s})^m \rtimes P_2(S_m)$. The action of $G$ on $\text{Irr}(S)$ is given by 
$$g(\lambda)(x) = \lambda(g(x)),$$ 
for $g \in G$, $\lambda \in \text{Irr}(S)$, $x \in S$. And for $\Psi_{\mbf{a}} \in \text{Irr}(S),$ the orbit of $\Psi_{\mbf{a}}$ in $\text{Irr}(S)$ under the action of $G$ corresponds to the orbit of $\psi_{\mbf{a}}$ in $\widehat{S}$ under the action of $G$.

Under the isomorphism $\widehat{S} \cong (\Z/2^s\Z)^n$, we have that the action of $g = (\mbf{b},\sigma)$ on $\widehat{S}$, which is given by 
\begin{align*}
g(\psi_\mbf{a})(x) &= \psi_{\mbf{a}}(g(x))\\
&= \psi_{\mbf{a}}((\mbf{b},\sigma)(x))\\
&= \psi_\mbf{a}(\mbf{c}), \text{ where } c_i = (x_{\sigma(i)})^{(2^{s-1}-1)^{b_i}}\\
&= \prod_{i=1}^m (x_{\sigma(i)})^{a_i(2^{s-1}-1)^{b_i}}\\
&= \prod_{i=1}^m (x_{i})^{a_{\sigma^{-1}(i)}(2^{s-1}-1)^{b_{\sigma^{-1}(i)}}}\\
&= \psi_{\sigma^{-1}(\mbf{b}(\mbf{a}))}(x),
\end{align*}
corresponds to the action of $g = (\mbf{b},\sigma)$ on $(\Z/2^s\Z)^n$ is given by $\mbf{a} \mapsto \sigma^{-1}(\mbf{b}(\mbf{a})),$ where $\mbf{b}(\mbf{a})_i = (a_i(2^{s-1}-1)^{b_i})$.

Note that the action of $G$ commutes with the action of $\Gamma$, so we get a corresponding action of $G$ on $(\Z/2^s\Z)^m/\Gamma$ under the bijection $\text{Irr}(S) \leftrightarrow (\Z/2^s\Z)^m/\Gamma$, which is also given by $\mbf{a} \mapsto \sigma^{-1}(\mbf{b}(\mbf{a})),$ where $\mbf{b}(\mbf{a})_i = (a_i(2^{s-1}-1)^b_i)$. The orbit of $\mbf{a}$ under this action will have the same size as the orbit of $\mbf{a}$ under the action $\mbf{a} \mapsto \sigma(\mbf{b}(\mbf{a})),$ where $\mbf{b}(\mbf{a})_i = (a_i(2^{s-1}-1)^{b_i})$.

Therefore, the orbit of $\Psi_{\mbf{a}}$ under the action of $G = (\Z/2\Z)^m \rtimes P_2(S_m)$ on $\text{Irr}(S)$ has the same size as the orbit of $\mbf{a}$ in $(\Z/2^s\Z)^n/\Gamma$ under the action of $G$ given by $\mbf{a} \mapsto \sigma(\mbf{b}(\mbf{a})),$ where $\mbf{b}(\mbf{a})_i = (a_i(2^{s-1}-1)^{b_i})$.
\end{proof}

\begin{lemma}\label{irr} Let $\Gamma = \text{Gal}(k(\zeta_{2^s})/k)$ and $\mbf{a} \in (\Z/2^s\Z)^m$ with $A_{j_1},\text{ } A_{j_2}$ invertible.
Then
$$|\text{orbit}(\mbf{a})| \geq \begin{cases} 2^{1 + k_{j_1}+k_{j_2}}, &[2^{s-1}-1] \in \Gamma\\
2^{2+k_{j_1}+k_{j_2}}, &[2^{s-1}-1] \notin \Gamma \end{cases},$$
under the action of $(\Z/2\Z)^m \rtimes P_2(S_m)$ on $(\Z/2^s\Z)^m/\Gamma$. 
\end{lemma}

\begin{proof}
Let $(\mbf{b},\tau) \in (\Z/2\Z)^m \rtimes P_2(S_m)$ be in the stabilizer of $\mbf{a}$ in $(\Z/2^s\Z)^m/\Gamma$. Then there exists $\phi \in \Gamma$ such that $(\mbf{b},\tau)(\mbf{a}) = \gamma_\phi \mbf{a}$. Let $\gamma = \gamma_\phi$. Let 
$$U = P_2(S_{2^{k_{j_1}}}) \times P_2(S_{2^{k_{j_2}}}) \subset P_2(S_m) \subset (\Z/2\Z)^m \rtimes P_{(S_m)}.$$

For $\tau \in U$, $\tau$ stabilizes the $A_j$, so we have $A_j = \gamma A_j$ and since the $A_{j_1}$ is invertible, we can conclude that $\gamma = 1$. So the orbit under the action of $U$ on $(\Z/2^s\Z)^m/\Gamma$ is the same as the orbit under the action on $(\Z/2^s\Z)^m$, which is equal to the product of the orbit of $\mbf{a}$ under the action of $P_2(S_{2^{k_{j_1}}})$ on  $I_{j_1}$ and the orbit of $\mbf{a}$ under the action of $P_2(S_{2^{k_{j_2}}})$ on $I_{j_2}$. So by Lemma \ref{irrH1}, we can conclude that the orbit has size at least $2^{k_{j_1}+k_{j_2}}$ under the action of $U$.

Now let $\mbf{c} \in (\Z/2\Z)^m$ be such that $c_i = 1$ for $i \in I_{j_1}$ and $0$ for all other entries. Let $\tau \in U$, and suppose that $(\mbf{c},\tau)(\mbf{a}) = \gamma \sigma(\mbf{a})$ for some $\sigma \in U$. Then we must have $(2^{s-1}-1)A_{j_1} = \gamma A_{j_1}$ and $A_{j_2} = \gamma A_{j_2}$, which would imply that $\gamma = 2^{s-1}-1$ and $\gamma = 1$, a contradiction. So the orbit has size at least $2^{1 + k_{j_1}+k_{j_2}}$.

Similarly, let $\mbf{d} \in (\mu_{2^s})^m$ be such that $d_i = 1$ for $i \in I_{j_2}$ and $0$ for all other entries. Let $\tau \in U$, and suppose that $(\mbf{d},\tau)(\mbf{a}) = \gamma \sigma(\mbf{a})$ for some $\sigma \in U$. Then we must have $A_{j_1} = \gamma A_{j_1}$ and $(2^{s-1}-1)A_{j_2} = \gamma_{j_2}$, which would imply that $\gamma = 1$ and $\gamma = 2^{s-1}-1$, a contradiction. So $(\mbf{d},\tau)(\mbf{a})$ is not equal to $\gamma\sigma(\mbf{a})$ for any $\sigma \in U$. Then suppose that that $(\mbf{d},\tau)(\mbf{a}) = \gamma (\mbf{c}, \sigma)(\mbf{a})$ for some $\sigma \in U$. Then we must have $A_{j_1} = (2^{s-1}-1)\gamma A_{j_1}$ and $(2^{s-1}-1)A_{j_2} = \gamma A_{j_2}$, which would imply that $\gamma = 2^{s-1}-1$, which is not possible if $[2^{s-1}-1] \notin \Gamma$. So for $[2^{s-1}-1] \notin \Gamma$ $(\mbf{d},\tau)(\mbf{a})$ is not equal to $\gamma\sigma(\mbf{a})$ or $\gamma(\mbf{c},\sigma)(\mbf{a})$ for any $\sigma \in U$. Thus the orbit has size at least $2^{1 + k_{j_1}+k_{j_2}} + 1$ and so since it must be a power of $2$, it has size at least $2^{2+k_{j_1}+k_{j_2}}$.

Thus we have shown that 
$$|\text{orbit}(\mbf{a})| \geq \begin{cases} 2^{1 + k_{j_1}+k_{j_2}}, &[2^{s-1}-1] \in \Gamma\\
2^{2+k_{j_1}+k_{j_2}}, &[2^{s-1}-1] \notin \Gamma \end{cases},$$
under the action of $(\Z/2\Z)^m \rtimes P_2(S_m)$ on $(\Z/2^s\Z)^m/\Gamma$. 

\end{proof}

\begin{corollary}\label{irrCor} For  $\mbf{a} \in (\Z/2^s\Z)^m$ with $A_{j_1},\text{ } A_{j_2}$ invertible, we can conclude that $$|\text{orbit}(\Psi_{\mbf{a}})| \geq \begin{cases} 2^{1 + k_{j_1}+k_{j_2}}, &[2^{s-1}-1] \in \Gamma\\
2^{2+k_{j_1}+k_{j_2}}, &[2^{s-1}-1] \notin \Gamma \end{cases}$$
under the action of $(\Z/2\Z)^m \rtimes P_2(S_m)$ on $\text{Irr}((\mu_{2^s})^m)$.
\end{corollary}

\begin{proof}
By Lemma \ref{changepersp'}, the orbit of $\Psi_{\mbf{a}}$ under the action of $(\Z/2\Z)^m \rtimes P_2(S_m)$ on $\text{Irr}((\mu_{2^s})^m)$ has the same size as the orbit of $\mbf{a}$ under the action of $(\Z/2\Z)^m \rtimes P_2(S_m)$ on $(\Z/2^s\Z)^m/\Gamma$ given by $\mbf{a} \mapsto \sigma(\mbf{b}(\mbf{a})),$ where $\mbf{b}(\mbf{a})_i = (a_i(2^{s-1}-1)^{b_i})$. And by Lemma \ref{irr}, this orbit has size at least $\begin{cases} 2^{1 + k_{j_1}+k_{j_2}}, &[2^{s-1}-1] \in \Gamma\\
2^{2+k_{j_1}+k_{j_2}}, &[2^{s-1}-1] \notin \Gamma \end{cases}.$
\end{proof}

  \begin{lemma}\label{faithrep1} Let $\epsilon = \zeta_{2^s} \in k_\text{sep}$, $\Gamma = \text{Gal}(k(\epsilon)/k)$, and $n = 2m$, $m = 2^t$. Assume that $[2^{s-1}-1] \notin \Gamma$. Then there exists a faithful representation of $P = (\mu_{2^s})^m/\{(x,\dots,x)\} \rtimes (\Z/2\Z)^m \rtimes P_2(S_m))$ of dimension 
  $$\begin{cases} 2^{2t}[k(\epsilon):k], &m = 2^t\\
2^{2+v_2(m)}(m-2^{v_2(m)})[k(\epsilon):k], &m \neq 2^t \end{cases}.$$
  \end{lemma}
  
  \begin{proof}
Let $T = (\mu_{2^s})^m/\langle (-1,\dots,-1)\rangle$. For $m = 2^t$, let $\mbf{a} = (1, 0, \dots,0,1,0,\dots,0)$. For $m \neq 2^t$, fix $j \leq \xi_2(m)-1$ and let $\mbf{a}$ be such that $a_j=1$, $a_n = 1$, and all other entries are $0$. 

Then consider
$$\Psi_{\mbf{a}}: T \to GL(k(\zeta_{2^s})) = GL_d(k),$$  
where $d = [k(\zeta_{2^s}):k]$. 

For $m = 2^t$, the elements of the orbit of $\mbf{a} \in (\Z/2^s\Z)^m$ under the action of $(\Z/2\Z)^m \rtimes P_2(S_m)$ are given by $\mbf{b}$ with $b_i = \in \{1, 2^{s-1}-1\}$ for some $i \in I_1$ as well as for some $i' \in I_2$ and all other entries $0$.  
This orbit has size $2^{2 + 2t-2} = 2^{2t}$. For $m \neq 2^t$, the elements of the orbit of $\mbf{a} \in (\Z/2^s\Z)^m$ under the action of $(\Z/2\Z)^m \rtimes P_2(S_m)$ are given by $\mbf{b}$ with $b_i \in \{1, 2^{s-1}-1\}$ for some $i \in I_{j}$ as well as for some $i' \in I_{\xi_2(m)}$ and all other entries $1$.  This orbit has  size $2^{2+k_j+v_2(m)}$.

 So the orbit of $\Psi_{\mbf{a}}$ under the action of $(\Z/2\Z)^m \rtimes P_2(S_{m})$ on the irreducible representations (not isomorphism classes) of $T$ has size 
 $$\begin{cases} 2^{2t}, &m = 2^t\\
2^{2+ k_{1} + v_2(m)}, &m \neq 2^t \end{cases}.$$ 

Let $\text{Stab}_{\mbf{a}}$ be the stabilizer of $\Psi_{\mbf{a}}$ in $P_2(S_{m})$. We can extend $\Psi_{\mbf{a}}$ to $T \rtimes \text{Stab}_{\mbf{a}}$ by defining $\Psi_{\mbf{a}}(\mbf{b},\tau) = \tau_{\Psi_{\mbf{a}}}(\mbf{b}) = \Psi_{\mbf{a}}(\mbf{b})$ (since $\tau \in \text{Stab}_{\mbf{a}}$). 

Let $\rho = \text{Ind}_{T \rtimes \text{Stab}_{\mbf{a}}}^P \Psi_{\mbf{a}}$. Then $\rho$ has dimension 
$$[(\Z/2\Z)^m \rtimes P_2(S_{m}): \text{Stab}_{\mbf{a}}]\dim(\Psi_{\mbf{a}}) = \begin{cases} 2^{2t}[k(\epsilon):k], &m = 2^t\\
2^{2+k_j+v_2(m)}[k(\epsilon):k], &m \neq 2^t \end{cases}.$$ 

For $m = 2^t$, $\rho$ is non-trivial (and hence faithful on $Z(P))$. So $\rho$ is a faithful representation over $k$ of dimension  $2^{2t}[k(\epsilon):k]$.

For $m \neq 2^t$, recall that we fixed $j \leq \xi_2(m)-1$. Let $\rho_j$ denote the $\rho$ obtained above for $j$. Then let $\rho = \oplus_{j=1}^{\xi_2(m)-1} \rho_j$.  Then $\rho$ has dimension
\begin{align*}
\dim(\rho) &= \sum_{j=1}^{\xi_2(m)-1} \dim(\rho_j)\\
&\geq \sum_{j=1}^{\xi_2(m)-1} 2^{2+k_{j}+v_2(m)}[k(\epsilon):k]\\
&= 2^{2+v_2(m)}\sum_{j=1}^{\xi_2(m)-1} 2^{k_{j}}[k(\epsilon):k]\\
&= 2^{2+v_2(m)}(m-2^{v_2(m)})[k(\epsilon):k].
\end{align*}
And by Lemma \ref{BMKS3.4}, $\rho$ is faithful.

\end{proof}

  \begin{lemma}\label{faithrep2} Let $\epsilon = \zeta_{2^s} \in k_\text{sep}$, $\Gamma = \text{Gal}(k(\epsilon)/k)$, $n = 2m$. Assume that $[2^{s-1}-1] \in \Gamma$. Then there exists a faithful representation of $P = (\mu_{2^s})^m/\{(x,\dots,x)\} \rtimes (\Z/2\Z)^m \rtimes P_2(S_m))$ of dimension 
  $$\begin{cases} 
  2^{2t-1}[k(\epsilon):k], &m = 2^t\\
  2^{1+v_2(m)}(m-2^{v_2(m)})[k(\epsilon):k], &m \neq 2^t \end{cases}.$$
\end{lemma}
\begin{proof}

Let $T = (\mu_{2^s})^m/\langle (-1,\dots,-1)\rangle$.

For $m = 2^t$, let $\mbf{a} = (1,0,\dots,0,1,0,\dots,0)$. For $m \neq 2^t$, fix $j \leq \xi_2(m) - 1$, and let $\mbf{a}$ be such that $\mbf{a}_j = 1$, $a_n = 1$, and all other entries are $0$.

Then consider
$$\mbf{a}: T \to k(\epsilon)^\times$$
defined by $\mbf{a}(\mbf{x}) = \prod_{i=1}^n (x_i)^{a_i} = x_1(x_{2^{t-1}+1}).$ Let $L_\mbf{a} = \{(\mbf{b},\sigma) \in (\Z/2\Z)^m \rtimes P_2(S_m) : (\mbf{b},\sigma)(\mbf{a}) = \mbf{a}\}$. Then we can extend the character $\mbf{a}$ to $T \rtimes L_\mbf{a}$ by defining $\mbf{a}(\mbf{x},\mbf{b},\sigma) = \mbf{a}(\mbf{x})$ for $(\mbf{b},\sigma) \in L_\mbf{a}$. Note that the orbit of $\mbf{a}$ under the action of $(\Z/2\Z)^m \rtimes P_2(S_m)$ on $(\Z/2^s\Z)^m$ has size $\begin{cases} 2^{2t}, &m = 2^t\\
2^{2+ k_{1} + v_2(m)}, &m \neq 2^t \end{cases}$. So 
$$[(\Z/2\Z)^m \rtimes P_2(S_m): L_\mbf{a}] = \begin{cases} 2^{2t}, &m = 2^t\\
2^{2+ k_{1} + v_2(m)}, &m \neq 2^t \end{cases}.$$ 
Let $N = [(\Z/2\Z)^m \rtimes P_2(S_m): L_\mbf{a}]$. Let
$$\rho = \text{Ind}_{T \rtimes L_\mbf{a}} : P \to GL_{N}(k(\epsilon)).$$
Let $V = k(\epsilon)$ be the ($1$-dimensional) $k(\epsilon)$-module corresponding to $\mbf{a}$, and let $H = T \rtimes L_{\mbf{a}}$. Then the induced module corresponding to $\rho$ is
$$V' = \text{Ind}_{H}^P V = k(\epsilon)[P] \otimes_{k(\epsilon)[H]} V.$$

I want to show that $\rho$ can be defined over $k(\epsilon-\epsilon^{-1})$, i.e. there is a $P$-invariant sub-module of $V'$ with coefficients in $k(\epsilon-\epsilon^{-1})$.

Let $U \subset (\Z/2\Z)^m \rtimes P_2(S_m)$ be a set of representatives of $P/H$. Then a basis for $V'$ is given by $\{u \otimes 1 : u \in U\}$. The action of $h \in H$ on $u\otimes 1$ is given by 
\begin{align*}
h \cdot (\sigma \otimes 1) &= h\sigma \otimes 1\\
&= \sigma(\sigma^{-1}h\sigma) \otimes 1\\
&= \sigma \otimes \mbf{a}(\sigma^{-1}h\sigma)\\
&= \sigma \otimes \mbf{a}(\sigma^{-1}(h))
\end{align*} 
Every element of $P$ can be written uniquely as $u h$ for some $u \in U, \text{ } h \in H$.  For $\eta = (\mbf{b},\sigma) \in (\Z/2\Z)^m \rtimes P_2(S_m)$, write $\eta u = u' h$. Since $\eta u \in (\Z/2\Z)^m \rtimes P_2(S_m)$ and $u' \in (\Z/2\Z)^m \rtimes P_2(S_m)$, we can conclude that $h \in ((\Z/2\Z)^m \rtimes P_2(S_m)) \cap H = L_\mbf{a}$ and so $\mbf{a}(h) = 1$.  Then the action of $\eta$ on $u \otimes 1$ is given by
\begin{align*} 
\eta \cdot (u \otimes 1) &= \eta u \otimes 1\\
&= u' h \otimes 1\\
&= u' \otimes \mbf{a}(h)\\
&= u' \otimes 1
\end{align*}

 Note that the representatives of $P/H$ are in bijection with $\mbf{b}$ in the orbit of $\mbf{a}$ under the action of $(\Z/2\Z)^m \rtimes P_2(S_m)$ via $u H \mapsto u(\mbf{a})$. For $\mbf{b}$ in the orbit, let $u_\mbf{b} = u \otimes 1$ where $u(\mbf{a}) = \mbf{b}$.  Then a basis for $V'$ is given by 
$$\{u_\mbf{b} : \mbf{b} \in \text{orb}(a)\},$$
The action of $x \in T$ is given by 
%$$x \cdot u_\mbf{b} = \mbf{b}(x) \cdot u_\mbf{b},$$
\begin{align*}
x \cdot u_\mbf{b} &= x \cdot (\sigma \otimes 1)\\
&= \sigma \otimes \mbf{a}(\sigma^{-1}(x))\\
&= \sigma \otimes \mbf{b}(x)\\
&= \mbf{b}(x) (\sigma \otimes 1)\\
&= \mbf{b}(x) u_\mbf{b}
\end{align*}

and the action of $\eta \in (\Z/2\Z)^m \rtimes P_2(S_m)$ is given by

\begin{align*}
\eta \cdot u_\mbf{b}  &= \eta \cdot (u \otimes 1), &\text{where } u(\mbf{a}) = \mbf{b}\\
&= u' \otimes 1, &\text{where } \eta u = u' h\\
&= u_{u'(\mbf{a})}\\
&= u_{u'h(\mbf{a})}, &\text{ since } h \in L_\mbf{a}\\
&= u_{\eta u(\mbf{a})}\\
&= u_{\eta (\mbf{b})}.
\end{align*} 
So the action of $x\eta \in P$ is given by 
$$(x\eta) \cdot (u_\mbf{b}) = x \cdot u_{\eta (\mbf{b})} = \eta (\mbf{b})(x) \cdot u_{\eta (\mbf{b})} = \mbf{b}(\eta^{-1} x) \cdot u_{\eta (\mbf{b})}.$$

Note that the orbit of $\mbf{a} = (1,0,\dots,0,1,0,\dots,0)$ is given by $\mbf{b} \in (\mu_{2^s})^n$ such that $b_{i_1} \in \{1, 2^{s-1}-1\} \text{ for some } i_1 \in I_1 = \{1,\dots,2^{t-1}\} \text{ and } b_{i_2} \in \{1, 2^{s-1}-1\} \text{ for some } i_2 \in I_2 = \{2^{t-1}+1,\dots, 2^t\} \text{ and } b_i = 0$ for all other indices. So the elements in the orbit come in pairs $\{\mbf{b},(2^{s-1}-1)\mbf{b}\}$ with $b_{i_1} = 1$ for $i_1 \in I_1$ and $b_{i_2} \in \{1, 2^{s-1}-1\}$ for $i_2 \in I_2$. Thus a basis for $V'$ is given by 
\begin{align*}
\{u_\mbf{b} : \mbf{b} \in \text{orb}(a)\} = \bigcup_{\mbf{b} \in \text{orb}(a) \text{ with } b_{i} = 1 \text{ for some } i \in I_1} \{u_\mbf{b}, u_{(2^{s-1}-1)\mbf{b}}\} 
\end{align*}
For $\mbf{b} \in \text{orb}(a) \text{ with } b_{i} = 1 \text{ for some } i \in I_1$, let
\begin{align*}
v_\mbf{b} &= -\epsilon^{-1} u_\mbf{b} - \epsilon u_{(2^{s-1}-1)\mbf{b}}\\
v_{(2^{s-1}-1)\mbf{b}} &= -u_{\mbf{b}} + u_{(2^{s-1}-1)\mbf{b}}
\end{align*}
I claim that 
$$W = \bigoplus_{\mbf{b} \in \text{orb}(a) \text{ with } b_{i} = 1 \text{ for some } i \in I_1} \left(k(\epsilon-\epsilon^{-1})v_{\mbf{b}} \oplus k(\epsilon-\epsilon^{-1})v_{(2^{s-1}-1)\mbf{b}}\right)$$
is a $P$-invariant sub-module of $V'$. The action of $\eta \in (\Z/2\Z)^m \rtimes P_2(S_m)$  on $v_\mbf{b}$ is given by
\begin{align*}
\eta \cdot v_\mbf{b} &= \eta \cdot (-\epsilon^{-1} u_\mbf{b} - \epsilon u_{(2^{s-1}-1)\mbf{b}})\\
&= -\epsilon^{-1} u_{\eta(\mbf{b})} - \epsilon u_{(2^{s-1}-1)\eta(\mbf{b})}\\
&= \begin{cases} v_{\eta(\mbf{b})}, &\eta(\mbf{b}) \text{ has } b_i = 1 \text{ for some } i \in I_1\\
(\epsilon^{-1}-\epsilon)v_{\eta(\mbf{b})} - v_{(2^{s-1}-1)\eta(\mbf{b})}, &\eta(\mbf{b}) \text{ has } b_i = 2^{s-1}-1 \text{ for some } i \in I_1\end{cases}\\
&\in W,
\end{align*}
and the action $\eta \in (\Z/2\Z)^m \rtimes P_2(S_m)$  on $v_{(2^{s-1}-1)\mbf{b}}$ is given by
\begin{align*}
\eta \cdot v_{(2^{s-1}-1)\mbf{b}} &= \eta \cdot (-u_\mbf{b} + u_{(2^{s-1}-1)\mbf{b}})\\
&= -u_{\eta (\mbf{b})} + u_{(2^{s-1}-1)\eta(\mbf{b})}\\
&= \begin{cases} v_{(2^{s-1}-1)\eta(\mbf{b})}, &\eta(\mbf{b}) \text{ has } b_i = 1 \text{ for some } i \in I_1\\
-v_{\eta(\mbf{b})}, &\eta(\mbf{b}) \text{ has } b_i = 2^{s-1}-1 \text{ for some } i \in I_1\end{cases}\\
&\in W.
\end{align*}

Note that 
$$G = \{g \in T : \text{ one entry of } g \text{ is } \epsilon \text{ and all other entries are }1\}$$ is a generating set of $T$.  So it suffices to consider the action of $g \in G$.  Note that for $g \in G$, $\mbf{b}(g) \in \{\epsilon,-\epsilon^{-1},1\}$. The action of $g$ on $v_\mbf{b}$ is given by 
\begin{align*}
g  \cdot v_\mbf{b} &= g \cdot (-\epsilon^{-1} u_\mbf{b} - \epsilon u_{(2^{s-1}-1)\mbf{b}})\\
&= -\epsilon^{-1} \mbf{b}(g) u_{\mbf{b}} - \epsilon ((2^{s-1}-1)\mbf{b})(g)  u_{(2^{s-1}-1)\mbf{b}}\\
&= \begin{cases}
-u_{\mbf{b}} + u_{(2^{s-1}-1)\mbf{b}}, &\mbf{b}(g) = \epsilon\\
\epsilon^{-2} u_{\mbf{b}} - \epsilon^{2}u_{(2^{s-1}-1)\mbf{b}}, &\mbf{b}(g) = -\epsilon^{-1}\\
-\epsilon^{-1}u_\mbf{b} - \epsilon u_{(2^{s-1}-1)\mbf{b}}, &\mbf{b}(g) = 1 \end{cases}\\
&= \begin{cases}
v_{(2^{s-1}-1)\mbf{b}}, &\mbf{b}(g)=\epsilon\\
(\epsilon-\epsilon^{-1})v_\mbf{b} - v_{(2^{s-1}-1)\mbf{b}}, &\mbf{b}(g) = \epsilon^{-1}\\
v_{\mbf{b}}, &\mbf{b}(g) = 1
\end{cases}\\
&\in W,
\end{align*}
and the action of $g$ on $v_{(2^{s-1}-1)\mbf{b}}$ is given by 
\begin{align*}
g \cdot v_{(2^{s-1}-1)\mbf{b}} &= g \cdot (-u_\mbf{b} + u_{(2^{s-1}-1)\mbf{b}})\\
&= -(\mbf{b}(g))  u_{\mbf{b}} + ((2^{s-1}-1)\mbf{b})(g)  u_{(2^{s-1}-1)\mbf{b}}\\
&= \begin{cases}
-\epsilon u_{\mbf{b}} - \epsilon^{-1} u_{(2^{s-1}-1)\mbf{b}}, &\mbf{b}(g) = \epsilon\\
\epsilon^{-1} u_{\mbf{b}} + \epsilon u_{(2^{s-1}-1)\mbf{b}}, &\mbf{b}(g) = -\epsilon^{-1}\\
-u_{\mbf{b}} + u_{(2^{s-1}-1)\mbf{b}}, &\mbf{b}(g) = 1 \end{cases}\\
&= \begin{cases}
(\epsilon-\epsilon^{-1})v_{(2^{s-1}-1)\mbf{b}} - v_{\mbf{b}}, &\mbf{b}(g) = \epsilon\\
-v_{\mbf{b}}, &\mbf{b}(g) = \epsilon^{-1}\\
v_{(2^{s-1}-1)\mbf{b}}, &\mbf{b}(g) = 1
\end{cases}\\
&\in W
\end{align*}

Therefore, $W = \bigoplus_{\mbf{b} \in \text{orb}(a) \text{ with } b_{i} = 1 \text{ for some } i \in I_1} \left(k(\epsilon-\epsilon^{-1}) v_{\mbf{b}}\oplus k(\epsilon-\epsilon^{-1})v_{(2^{s-1}-1)\mbf{b}}\right)$ is a $P$-invariant sub-module of $V'$. Hence $W$ corresponds to a representation $\rho': P \to GL_{N}(k(\epsilon-\epsilon^{-1}))$ of $P$ of dimension $N$ over $k(\epsilon-\epsilon^{-1})$.

Let $d = [k(\epsilon-\epsilon^{-1}):k]$. Then we can embed $k(\epsilon-\epsilon^{-1})$ in $GL_d(k)$ And thus we have an embedding $GL_{N}(k(\epsilon-\epsilon^{-1})) \hookrightarrow GL_{Nd}(k)$. Thus we have 
$$\rho': P \to GL_{N}(k(\epsilon-\epsilon^{-1})) \hookrightarrow GL_{Nd}(k).$$
Thus $\rho'$ is a representation of $P$ of dimension 
\begin{align*}
Nd &= \begin{cases} 2^{2t}[k(\epsilon-\epsilon^{-1}):k], &m = 2^t\\
2^{2+ k_{1} + v_2(m)}[k(\epsilon-\epsilon^{-1}):k], &m \neq 2^t \end{cases}\\
&= \begin{cases} 2^{2t-1}[k(\epsilon):k], &m = 2^t\\
2^{1+ k_{1} + v_2(m)}[k(\epsilon):k], &m \neq 2^t \end{cases}
\end{align*}
for $[2^{s-1}-1] \in \Gamma$.

For $m = 2^t$, $\rho'$ is non-trivial (and hence faithful) on $Z(P)$. So $\rho'$ is a faithful representation over $k$ of dimension  $2^{2t-1}[k(\epsilon):k]$.

For $m \neq 2^t$, recall that we fixed $j \leq \xi_2(m)-1$. Let $\rho_j$ denote the $\rho'$ obtained above for $j$. Then let $\rho = \oplus_{j=1}^{\xi_2(m)-1} \rho_j$.  Then $\rho$ has dimension
\begin{align*}
\dim(\rho) &= \sum_{j=1}^{\xi_2(m)-1} \dim(\rho_j)\\
&\geq \sum_{j=1}^{\xi_2(m)-1} 2^{1+k_{j}+v_2(m)}[k(\epsilon):k]\\
&= 2^{1+v_2(m)}\sum_{j=1}^{\xi_2(m)-1} 2^{k_{j}}[k(\epsilon):k]\\
&= 2^{1+v_2(m)}(m-2^{v_2(m)})[k(\zeta_{2^s}):k].
\end{align*}
And by Lemma \ref{BMKS3.4}, $\rho$ is faithful.

\end{proof}

 \begin{proof}[Proof of Theorem \ref{PGLn2'} for the case $n =2m$]

Recall that for $n = 2m$, 
$$P \cong (\mu_{2^s})^m/\{(-1,\dots,-1)\} \rtimes ((\Z/2\Z)^m \rtimes P_2(S_n)).$$  Let $\rho$ be a faithful representation of $P$ of minimum dimension.  

For $m =2^t$, the center has rank $1$ and so $\rho$ is also irreducible. For $m \neq 2^t$, let $\rho = \oplus_{j=1}^{\xi_2(m)-1} \varphi_j$ be the decomposition into irreducibles. Let $C = Z(P)$. By Lemma \ref{BMKS3.5}, if $\chi_j$ are the central characters of $\varphi_j$, then $\{\chi_j|_{C[2]}\}$ form a basis for $\widehat{C[2]}$. Let $\mbf{b}^j$  be the dual basis for $C[2]$ so that $\varphi_j(\mbf{b}^i)$ is trivial for $i \neq j$.  

For $m \neq 2^t$, let $T = (\mu_{2^s})^m/\langle (-1,\dots,-1)\rangle$. For $j \leq \xi_2(m)$, let
$$T_j = \{\mbf{b} \in T: b_i = 1 \text{ for } i \notin I_j\}.$$
For $j \leq \xi_2(m)$, define $\mbf{e}^j \in T$ by 
$$(\mbf{e}^j)_i = \begin{cases} -1, &i \in I_j\\
1, &i \notin I_j \end{cases}.$$ Then $\{\mbf{e}^j\}$ is a basis for $C[2]$.  Write $\mbf{b}^j = \oplus_{i} a_{i,j}\mbf{e}^i$. Then $\varphi_j$ will be non-trivial on $T_j \cap C[2]$ if and only if $a_{j,j} \neq 0$. Note that $(a_{i,j})$ is the change of basis matrix from $\{\mbf{e}^j\}$ to $\{\mbf{b}^j\}$. Since it is a change of basis matrix, it must be invertible. By Lemma \ref{claim1}, we can rearrange the $\mbf{b}^j$ such that for all $i$ $a_{i,i} \neq 0$ in the change of basis matrix from $\{\mbf{e}^j\}$ to $\{\mbf{b}^j\}$. And so we can rearrange the $\varphi_j$ such that $\chi_j|_{C[2]}$ is non-trivial on $T_j \cap C[2]$ and thus $\varphi_j$ is non-trivial on $T_j \cap C[2]$.

For $m = 2^t$, let $\varphi = \rho$. For $m \neq 2^t$, fix $j \leq \xi_2(m)-1$ and let $\varphi = \varphi_j$. 

By Clifford's Theorem (Theorem \ref{cliff}), $\varphi|_{T}$ decomposes into a direct sum of irreducibles in the following manner:
$$\varphi|_{T} \cong  \left( \oplus_{i=1}^c  \lambda_i \right)^{\oplus d}, \text{ for some } c, d,$$ 
with the $\lambda_i$ non-isomorphic, and $(\Z/2\Z)^m \rtimes P_2(S_{m})$ acts transitively on the isomorphism classes of the $\lambda_i$, so the $\lambda_i$ have the same dimension and the number of $\lambda_i$, $c$, divides $|(\Z/2\Z)^m \rtimes P_2(S_{m})|$.

For $m = 2^t$, since $\varphi = \rho$ is faithful, it is non-trivial on $Z(P)$, and thus one of the $\lambda_i$ must be non-trivial on $Z(P)$. Without loss of generality assume that $\lambda_1$ is non-trivial on $Z(P)$. For $m \neq 2^t$, $\varphi$ is non-trivial on $T_j \cap C[2] \subset T[2]$, so one of the $\lambda_i$ must be non-trivial on $T_j \cap C[2]$. Without loss of generality assume the $\lambda_1$ is non-trivial on $T_j \cap C[2]$.

Note that the irreducible representations of $T$ are in bijection with irreducible representations of $(\mu_{2^s})^m$ which are trivial on $\langle (-1,\dots,-1) \rangle$.  By Lemma \ref{corrlemma}, the irreducible reprsentations of $(\mu_{2^s})^m$ are given by $\Psi_{\mbf{a}}$ with $\mbf{a} \in (\Z/2^s\Z)^m/\Gamma$, for $\Gamma = \text{Gal}(k(\zeta_{2^s})/k)$, and if $\Psi_{\mbf{a}}$ is non-trivial on $(\mu_{2^s})^m[2]$, then $\Psi_\mbf{a}$ has dimension $[k(\zeta_{2^s}):k]$.  So since $\lambda_1$ is non-trivial on 
$$\begin{cases} Z(P) = \langle (-1,\dots,-1,1,\dots,1)\rangle, &m = 2^t\\
T_j \cap C[2], &m \neq 2^t \end{cases} \subset (\mu_{2^s})^m[2]/\langle (-1,\dots,-1) \rangle,$$
 we must have $\dim(\lambda_1) = [k(\zeta_{2^s}):k]$, and so $\dim(\lambda_i) = [k(\zeta_{2^s}):k]$ for all $i$. And $\Psi_\mbf{a}$ will be trivial on $\langle (-1,\dots,-1) \rangle$ if and only if $2 \divides \sum_{i=1}^n a_i$. So $\lambda_1 \cong \Psi_{\mbf{a}}$ for some $\mbf{a} \in (\Z/2^s\Z)^n/\Gamma$ with $2 \divides \sum_{i=1}^n a_i$. 

For $m = 2^t$, recall that $I_j$ denotes the $j$th sub-block of $2^{t-1}$ entries in $\{1,\dots,2^t\}$. So since $\lambda_1$ is non-trivial on $Z(P) = \langle (-1,\dots,-1,1,\dots,1)\rangle$, we must have that $0 \neq 2^{s-1}A_2$. Thus $2 \nmid A_2$ and so $A_2$ is invertible. And so since $2 \divides \sum_{i=1}^n a_i = A_1 + A_2$, we must have $2 \nmid A_1$ and so $A_1$ is invertible as well. And for $m \neq 2^t$, since $\lambda_1 \cong \Psi_{\mbf{a}}$ is non-trivial on $T_j \cap C[2]$, we must have
$$\text{ }\sum_{i \in I_j} a_i2^{s-1} \neq 0.$$ 
Since $2 \divides \sum_{i=1}^m a_i = 0$, we must have $\sum_{i=1}^m a_i2^{s-1} = 0$; hence since $\sum_{i \in I_j} a_i2^{s-1} \neq 0$, we must also have $\sum_{i \in I_{j'}} a_i2^{s-1} \neq 0$ for some $j' \neq j$. Therefore, we must have $A_j = \sum_{i \in I_j} a_i$ invertible and $A_{j'} \sum_{i \in I_{j'}} a_i$ invertible for some $j' \neq j$.

\textbf{Case 1:} For $[2^{s-1}-1] \notin \Gamma$, by Corollary \ref{irrCor}, the orbit of $\lambda_1 = \Psi_{\mbf{a}}$ under the action of $(\Z/2\Z)^m \rtimes P_2(S_n)$ has size at least 
$2^{2 + k_{j}+k_{j'}}$. This is equal to $2^{2+2t-2} = 2^{2t}$ for $m = 2^t$. And for $m \neq 2^t$, since $k_{j'} \geq v_2(m)$ for all $j'$, we have $2^{2 + k_{j}+k_{j'}}\geq 2^{2 + k_{j}+v_2(m)}$. So 
$$c \geq \begin{cases} 2^{2t}, &m = 2^t\\
 2^{2 + k_{j}+v_2(m)}, &m \neq 2^t \end{cases}.$$ Thus
$$\dim(\varphi) \geq \begin{cases} 2^{2t}[k(\zeta_{2^s}):k], &m = 2^t\\
2^{2 + k_{j}+v_2(m)}[k(\zeta_{2^s}):k], &m \neq 2^t, \lambda = \varphi_j \end{cases}.$$
Thus for $m = 2^t$, 
$$\dim(\rho) = 2^{2t}[k(\zeta_{2^s}):k].$$
And for $m \neq 2^t$,
\begin{align*}
\dim(\rho) &= \sum_{j=1}^{\xi_2(m)-1} \dim(\varphi_j)\\
&\geq \sum_{j=1}^{\xi_2(m)-1} 2^{2+k_{j}+v_2(m)}[k(\zeta_{2^s}):k]\\
&= 2^{2+v_2(m)}\sum_{j=1}^{\xi_2(m)-1} 2^{k_{j}}[k(\zeta_{2^s}):k]\\
&= 2^{2+v_2(m)}(m-2^{v_2(m)})[k(\zeta_{2^s}):k].
\end{align*}
And by Lemma \ref{faithrep1}, there exists a faithful representation of $P$ of dimension 
$$\begin{cases} 2^{2t}[k(\zeta_{2^s}):k], &m = 2^t\\
2^{2+v_2(m)}(m-2^{v_2(m)})[k(\zeta_{2^s}):k], &m \neq 2^t \end{cases}.$$
Therefore, in the case $n = 2m$, $[2^{s-1}-1] \notin \Gamma$,
\begin{align*}
\ed_k(PGL_{n}(\F_q),2)) &=  \begin{cases} 2^{2t}[k(\zeta_{2^s}):k], &m = 2^t\\
2^{2+v_2(m)}(m-2^{v_2(m)})[k(\zeta_{2^s}):k], &m \neq 2^t \end{cases}.
\end{align*}

\textbf{Case 2:} For $[2^{s-1}-1] \in \Gamma$, by Corollary \ref{irrCor}, the orbit of $\lambda_1 = \Psi_{\mbf{a}}$ under the action of $(\Z/2\Z)^m \rtimes P_2(S_n)$ has size at least 
$2^{1 + k_{j}+k_{j'}}$. This is equal to $2^{1+2t-2} = 2^{2t-1}$ for $m = 2^t$. And for $m \neq 2^t$, since $k_{j'} \geq v_2(m)$ for all $j'$, we have $2^{1 + k_{j}+k_{j'}}\geq 2^{2 + k_{j}+v_2(m)}$. So 
$$c \geq \begin{cases} 2^{2t-1}, &m = 2^t\\
 2^{1 + k_{j}+v_2(m)}, &m \neq 2^t \end{cases}.$$ Thus
$$\dim(\varphi) \geq \begin{cases} 2^{2t-1}[k(\zeta_{2^s}):k], &m = 2^t\\
2^{1 + k_{j}+v_2(m)}[k(\zeta_{2^s}):k], &m \neq 2^t, \lambda = \varphi_j \end{cases}.$$
Thus for $m = 2^t$, 
$$\dim(\rho) = 2^{2t-1}[k(\zeta_{2^s}:k].$$
And for $m \neq 2^t$,
\begin{align*}
\dim(\rho) &= \sum_{j=1}^{\xi_2(m)-1} \dim(\varphi_j)\\
&\geq \sum_{j=1}^{\xi_2(m)-1} 2^{1+k_{j}+v_2(m)}[k(\zeta_{2^s}):k]\\
&= 2^{1+v_2(m)}\sum_{j=1}^{\xi_2(m)-1} 2^{k_{j}}[k(\zeta_{2^s}):k]\\
&= 2^{1+v_2(m)}(m-2^{v_2(m)})[k(\zeta_{2^s}):k].
\end{align*}
And by Lemma \ref{faithrep2}, there exists a faithful representation of $P$ of dimension $$ \begin{cases} 
  2^{2t-1}[k(\zeta_{2^s}):k], &m = 2^t\\
  2^{1+v_2(m)}(m-2^{v_2(m)})[k(\zeta_{2^s}):k], &m \neq 2^t \end{cases}.$$ 
  Therefore, in the case $n = 2m$, $[2^{s-1}-1] \in \Gamma$, 
\begin{align*}
\ed_k(PGL_{n}(\F_q),2)) &= \begin{cases} 
  2^{2t-1}[k(\zeta_{2^s}):k], &m = 2^t\\
  2^{1+v_2(m)}(m-2^{v_2(m)})[k(\zeta_{2^s}):k], &m \neq 2^t \end{cases}.
  \qedhere
\end{align*}

\end{proof}

\section{\texorpdfstring{The Special Linear Groups - $n = 2$ or odd, $q \equiv 1 \mod 4$}{The Special Linear Groups - n = 2 or odd, q equiv 1 mod 4}}

\begin{theorem}\label{SLn2'} Let $p \neq 2$ be a prime and $q = p^r$. Let $k$ be a field with $\text{char } k \neq 2$. Assume that $q \equiv 1 \mod 4$, and let $s = v_2(q-1)$. Then 
\begin{align*}
&\ed_k(SL_n(\F_q),2)\\
&= \begin{cases}
\ed_k(GL_{n-1}(\F_q),2), &2 \nmid n\\ 
 2[k(\epsilon):k], &n=2, \text{ } [-1] \notin \Gamma\\
[k(\epsilon):k], &n=2, \text{  } [-1] \in \Gamma, \text{ } x^2 + y^2 = -1 \text{ has a solution in } k(\epsilon+\epsilon^{-1}) \\
 2[k(\epsilon):k], &n=2, \text{  }[-1] \in \Gamma, \text{ } x^2 + y^2 = -1 \text{ has no solutions in } k(\epsilon+\epsilon^{-1})
\end{cases}.
\end{align*}
\end{theorem}

By (\cite{Gr}, Proposition 1.1), 
$$|SL_n(\F_q)| = \frac{|GL_n(\F_q)|}{q-1}.$$
So
\begin{align*}
|SL_n(\F_q)|_2 &= \frac{|GL_n(\F_q)|_2}{2^{v_l(q-1)}} = 2^{s(n-1)} \cdot |S_{n}|_2\\
\end{align*}

The proof when $2 \nmid n$ is simple:
\begin{proof}[Proof of Theorem \ref{SLn2} for the case $2 \nmid n$]
Note that we can embed $GL_{n-1}(\F_q)$ in $SL_n(\F_q)$ by sending the matrix $A \in GL_{n-1}(\F_q)$ to 
$$\begin{pmatrix} A & 0\\
0 & \text{det}(A^{-1})\end{pmatrix}.$$
If $2 \nmid n$, then $|S_n|_2 = |S_{n-1}|_2$, thus 
$$|SL_n(\F_q)|_2 = 2^{s(n-1)} \cdot |S_n|_2 = 2^{s(n-1)} \cdot |S_{n-1}|_2 = |GL_{n-1}(\F_q)|_2.$$
Therefore, the Sylow $2$-subgroups of $SL_n(\F_q)$ are isomorphic to Sylow $2$-subgroups of $GL_{n-1}(\F_q)$. Thus 
\begin{align*} \ed_k(SL_n(\F_q),2) &= \ed_k(GL_{n-1}(\F_q),2) = (n-1)[k(\zeta_{2^s}):k].
\qedhere
\end{align*}
\end{proof}
For $n = 2$, we have 
\begin{align*}
P &= \{(\mbf{b},a) \in (\mu_{2^s})^2 \rtimes \mu_2 : b_1b_2 = a, \text{ } -1(b_1,b_2) = (b_2,b_1)\}\\
&= \{(b, ab^{-1}, a) \in (\mu_{2^s})^2 \rtimes \mu_2, \text{ } -1(b_1,b_2) = (b_2,b_1)\}\\
&= \langle (g,g^{-1},1), (1,-1,-1) :\\
&\qquad (g,g^{-1},1)^{2^{s-1}} = (-1,-1,1) = (1,-1,-1)^2,\\
&\qquad (g,g^{-1},1)^{2^{s}}  = (1,1,1),\\ &\qquad (1,-1,-1)(g,g^{-1},1)(1,-1,-1)^{-1} = (b^{-1},-b,-1)(-1,1,-1) = (g^{-1},g,1)\rangle\\
&\text{ for } g \text{ a generator } of \mu_{2^s}\\
&\cong \langle w,v : w^{2^{s-1}} = v^2, w^{2^{s}} = 1, vwv^{-1} = w^{-1} \rangle \text{ where } w = (g, g^{-1},1), \text{ } v = (1,-1,-1)\\
&\cong Q_{2^{s+1}}
\end{align*}

So to prove Theorem \ref{SLn2} in the case $n = 2$, it suffices to prove the folowing proposition.

\begin{proposition}\label{edQuat}
Let $k$ be a field with $\text{char } k \neq 2$. Let $s  > 2$ be an integer, let $\epsilon = \zeta_{2^{s}}$ in $k_\text{sep}$, and let $\Gamma = \text{Gal}(k(\epsilon)/k)$. Then 
\begin{align*}
\ed_k(Q_{2^{s+1}},2) &= \begin{cases}
 2[k(\epsilon):k], &[-1] \notin \Gamma\\
 [k(\epsilon):k], &[-1] \in \Gamma \text{ and }x^2 + y^2 = -1 \text{ has a solution in } k(\epsilon+\epsilon^{-1}) \\
2[k(\epsilon):k], &[-1] \in \Gamma \text{ and } x^2 + y^2 = -1 \text{ has no solutions in } k(\epsilon+\epsilon^{-1})
\end{cases}. 
\end{align*}

\end{proposition}

\subsubsection{\texorpdfstring{Character table of $Q_{2^{s+1}}$}{Character table of Q2s+1}}

We will first find the character table of $Q_{2^{s+1}}$. \begin{footnote}{Mathar found the irreducible representations of $Q_{8}$, $Q_{16}$, and $Q_{32}$ in \cite{Ma}.}\end{footnote} 

Note that for $i \in \mathbb{N}$, $[w^i,v] = w^ivw^{-i}v^{-i} = w^{2i}.$ So $\langle w^2 \rangle$ is contained in the commutator. And $\langle w^2 \rangle$ has index $4$ in $Q_{2^{s+1}}$; thus $Q^{2^{s+1}}/\langle w^2 \rangle$ is abelian since all groups of order $4$ are abelian. Therefore, the commutator is given by $\langle w^2 \rangle$, which has index $4$ in $Q_{2^{s+1}}$, so there are four $1$-dimensional irreducible representations of $Q_{2^{s+1}}$ over $k_\text{sep}$.

These are given by 

$$\begin{array}{c | c }
        & w^av^b \\
 \hline 
\text{triv} & 1  \\
 \hline 
\lambda_1 & (-1)^a \\
 \hline  
\lambda_2 & (-1)^b\\
 \hline 
\lambda_3 & (-1)^a(-1)^b
\end{array}$$

Since $\langle w \rangle$ is an Abelian subgroup of $Q_{2^{s+1}}$ of index $2$, we can conclude that the irreducible representations have dimension at most $2$. Since $|Q_{2^{s+1}}| = 2^{s+1}$, we must have $2^{s-1}-1$ $2$-dimensional irreducible representations (so that $4 + 4(2^{s-1}-1) = 2^{s+1}$).

Let $\epsilon = \zeta_{2^{s}}$ in $k_\text{sep}$ and let 
$$C = \begin{pmatrix} \epsilon & 0\\
0 & \epsilon^{-1} \end{pmatrix}, \text{ } D = \begin{pmatrix} 0 & 1\\
-1 & 0 \end{pmatrix}.$$
Note that 
\begin{align*}
DCD^{-1} &= \begin{pmatrix} 0 & 1\\
-1 & 0 \end{pmatrix}\begin{pmatrix} \epsilon & 0\\
0 & \epsilon^{-1} \end{pmatrix}\begin{pmatrix} 0 & -1\\
1 & 0 \end{pmatrix}\\
&= \begin{pmatrix} 0 & \epsilon^{-1}\\
-\epsilon & 0 \end{pmatrix}\begin{pmatrix} 0 & -1\\
1 & 0 \end{pmatrix}\\
&= \begin{pmatrix} \epsilon^{-1} & 0\\
0 & \epsilon \end{pmatrix}\\
&= C^{-1}.
\end{align*}

For $i \in \{1, 3, 5, \dots, 2^s-1\}$, the map $\lambda_i: Q_{2^{s+1}} \to GL_2(k_\text{sep})$ given by $w \mapsto C^i, v \mapsto D$ is a homomorphism. And for $i \in \{2,4, \dots, 2^{s-1}-2\}\backslash\{2^{s-2}\}$, the map $\lambda_i: Q_{2^{s+1}} \to GL_2(k_\text{sep})$ given by $w \mapsto C^i, v \mapsto D$ is a homomorphism. For $s > 2, i = 2^{s-2}$, the map $\lambda_{2^{s-2}}: Q_{2^{s+1}} \to GL_2(k_\text{sep})$ given by $w \mapsto \begin{pmatrix} 0 & -1\\
1 & 0 \end{pmatrix}, v \mapsto \begin{pmatrix} 1 & 0\\
0 & -1\end{pmatrix}$ is a homomorphism.

The corresponding characters are given by 
$$\begin{array}{c | c | c}
        & w^a & w^av \\
 \hline 
\chi_i & \epsilon^{ai} + \epsilon^{-ai} & 0 
\end{array}$$
By notes of Alexander Merkurjev, $\rho_i$ is irreducible if and only if 
$$1 = \frac{1}{2^s} \sum_{g \in Q_{2^{s+1}}} \chi_i(g^{-1})\chi_i(g).$$
Note that
\begin{align*}
\frac{1}{2^{s+1}}\sum_{g \in Q_{2^{s+1}}} \chi_i(g^{-1})\chi_i(g) &= \frac{1}{2^{s+1}} \sum_{a=0}^{2^{s}-1} (\epsilon^{-ai} + \epsilon^{ai})(\epsilon^{ai}+\epsilon^{-ai})\\
&= \frac{1}{2^{s+1}} \sum_{a=0}^{2^{s}-1} (2 + \epsilon^{-2ai} + \epsilon^{2ai})\\
&= \frac{1}{2^{s}} (2^s + \sum_{a=0}^{2^{s}-1} (\epsilon^{-2i})^a + \sum_{a=0}^{2^{s}-1} (\epsilon^{2i})^a)
\end{align*}

For $i \notin \{0,2^{s-1}\}$, $\epsilon^{2i}$ and $\epsilon^{-2i}$ are both $2^{s-1}$-th roots of unity not equal to $1$. So 
$$\sum_{a=0}^{2^{s}-1} (\epsilon^{-2i})^a = 0 = \sum_{a=0}^{2^{s}-1} (\epsilon^{2i})^a.$$  
Thus $\frac{1}{2^{s+1}}\sum_{g \in Q_{2^{s+1}}} \chi_i(g^{-1})\chi_i(g) = 1$ and hence $\rho_i$ is irreducible.

I claim that if $\lambda_i \cong \lambda_j$ for $i,j \in \{1, \dots, 2^{s-1}-1\}$, then $i = j$.  By notes of Alexander Merkurjev, two isomorphic irreducible representations will have the same character. For $a = 2^{s-1}$, we have $\chi_i(x^a) = 2(-1)^i$ and $\chi_j = 2(-1)^j$, so if $\lambda_i \cong \lambda_j$ (and so $\chi_i = \chi_j$), then we must have $i \equiv j \mod 2$.   And since $1 = \frac{1}{2^{s+1}} \sum_{g \in Q_{2^{s+1}}} \chi_i(g^{-1})\chi_i(g)$ and $\chi_i = \chi_j$, we must have
$$\frac{1}{2^{s+1}} \sum_{g \in Q_{2^{s+1}}} \chi_i(g^{-1})\chi_j(g) = 1.$$
Note that 
\begin{align*}
\frac{1}{2^{s+1}} \sum_{g \in Q_{2^{s+1}}} \chi_i(g^{-1})\chi_j(g) &= \frac{1}{2^{s+1}} \sum_{a=0}^{2^{s}-1} (\epsilon^{-ai} + \epsilon^{ai})(\epsilon^{aj}+\epsilon^{-aj})\\
&= \frac{1}{2^{s+1}} \sum_{a=0}^{2^{s}-1} (\epsilon^{a(j-i)} + \epsilon^{-2a(i+j)} + \epsilon^{a(i+j)}+\epsilon^{a(i-j)})
\end{align*}
Note that for $i \neq j \in \{1,2,\dots, 2^{s-1}-1\}$, we have $i \neq j \in \Z/2^{s}\Z$ and $i \neq -j \in \Z/2^{s}\Z$. So $\epsilon^{(i-j)}$, $\epsilon^{(j-i)}$, $\epsilon^{-(i+j)}$, and $\epsilon^{(i+j)}$  are $2^{s-1}$-th roots of unity not equal to $1$. So 
$$0 = \sum_{a=0}^{2^{s}-1} \epsilon^{a(j-i)} = \sum_{a=0}^{2^{s}-1} \epsilon^{a(i-j)} = \sum_{a=0}^{2^{s}-1} \epsilon^{-a(i+j)} = \sum_{a=0}^{2^{s}-1} \epsilon^{a(i+j)}.$$
Thus for $i \neq j \in \{1,2,\dots, 2^{s-1}-1\}$,
$$\frac{1}{2^{s+1}} \sum_{g \in Q_{2^{s+1}}} \chi_i(g^{-1})\chi_j(g) = 0 \neq 1.$$
So if $i \neq j \in \{1,2,3 \dots, 2^{s-1}-1\}$, then we can conclude that $\lambda_i \not \cong \lambda_j.$  Thus there are $2^{s-1}-1$ distinct $2$-dimensional irreducible representations over $k_\text{sep}$ given by $\lambda_i$ for $i = 1, 2, \dots, 2^{s-1}-1$.

Note that these representations are defined over $k(\epsilon)$ and $\lambda_i$ is faithful if and only if $2 \nmid i$. So the faithful irreducible representations of $Q_{2^{s}}$ over $k_\text{sep}$ are given by $\lambda_i$ for $1 \leq i < 2^{s-1},$ $2 \nmid i$. 

Note that $k_\text{sep}[Q_{2^s}] = k \times k \times k \times k \times A_1 \times \dots, A_{2^{s-1}-1}$, where the copies of $k$ correspond to the $1$-dimensional irreducible representations of $Q_{2^{s+1}}$ and $A_i$ correspond to the $2$-dimensional irreducible representations of $Q_{2^{s+1}}$ over $k_\text{sep}$ (and so have dimension $4$). 

The idempotents are given by 
\begin{align*}
&f_1 = \frac{1}{2^{s+1}} \sum_{a = 1}^{2^{s}} (w^a+w^av),\\
&f_2 = \frac{1}{2^{s+1}} \sum_{a = 1}^{2^{s}} (-1)^a(w^a+w^av),\\
&f_3 = \frac{1}{2^{s+1}} \sum_{a = 1}^{2^{s}} (w^a-w^av),\\
&f_4 = \frac{1}{2^{s+1}} \sum_{a = 1}^{2^{s}} (-1)^a(w^a-w^av)
\end{align*}
and
$$e_i = \frac{2}{2^{s+1}}  \sum_{a = 1}^{2^{s}} (\epsilon^{ai} + \epsilon^{-ai})w^a.$$

\subsubsection{Proof}

\begin{proof}[Proof of Proposition \ref{edQuat}]
Let $R = k[Q_{2^{s+1}}]$. Note that  $R = eR \times fR$ where $e = \frac{1}{2}(1-v^2)$ and $f = 1-e = \frac{1}{2}(1+v^2)$.  

$v^2$ acts on $A_i = e_iR$.  Write $e = \sum_{i \in I} e_i$.  Then $ee_i = \frac{1-v^2}{2}e_i = \frac{e_i - v^2e_i}{2} = \begin{cases} e_i, &i \in I\\
0, &i \notin I \end{cases}$.  So $v^2e_i = \begin{cases} -e_i, &i \in I\\
e_i, &i \notin I \end{cases}.$ In order for a product of the $A_i's$ to correspond to an irreducible representation that is non-trivial on the center, we must have $v^2$ acting non-trivially on $A_i$. So we only care about those $e_iR$ with $i \in I$ when $e = \sum_{i \in I} e_i$.  

I claim that $I = \{i \in [1,2^{s-1}-1] : 2 \nmid i\}$. Note that
\begin{align*} 
\sum_{i = 1, 2 \nmid i}^{2^{s-1}-1} e_i &= \sum_{i=1, 2 \nmid i}^{2^{s-1}-1} \left(\frac{2}{2^{s+1}}  \sum_{a = 1}^{2^{s}} (\epsilon^{ai} + \epsilon^{-ai})w^a\right)\\
&= \frac{2}{2^{s+1}}  \sum_{a = 1}^{2^{s}} \left(\sum_{i=1, 2 \nmid i}^{2^{s-1}-1} (\epsilon^{ai} + \epsilon^{-ai})w^a\right)\\
\end{align*}
And 
\begin{align*}
\sum_{i=1, 2 \nmid i}^{2^{s-1}-1} (\epsilon^{ai} + \epsilon^{-ai}) &= \sum_{i=1, 2 \nmid i}^{2^{s-1}-1} (\epsilon^{ai} + \epsilon^{(2^{s}-1)ai})\\
&= \sum_{i=1, 2 \nmid i}^{2^{s}-1} \epsilon^{ai}.
\end{align*}
For $a \neq 0, 2^{s-1}$,  $\epsilon^a$ is a $2^{s}$-th root of unity not equal to $-1$, so the sum of the odd powers of $\epsilon^{a}$ is $0$, that is ${\displaystyle \sum_{i=1, 2 \nmid i}^{2^{s}-1} \epsilon^{ai} = 0.}$ So 
\begin{align*}
\sum_{i = 1, 2 \nmid i}^{2^{s-1}-1} e_i &= \frac{2}{2^{s+1}}  \sum_{a = 1}^{2^{s}} \left(\sum_{i=1, 2 \nmid i}^{2^{s-1}-1} (\epsilon^{ai} + \epsilon^{-ai})w^a\right)\\
&= \frac{2}{2^{s+1}}  \left(\sum_{i=1, 2 \nmid i}^{2^{s-1}-1} 2 +\sum_{i=1, 2 \nmid i}^{2^{s-1}-1} -2w^{2^{s-1}} \right)\\
&= \frac{2}{2^{s+1}}(2^{s-1} - 2^{s-1}w^{2^{s-1}})\\
&= \frac{1}{2}(1-w^{2^{s-1}})\\
&= \frac{1}{2}(1-v^2)\\
&= e
\end{align*}
So $I = \{i \in [1,2^{s-1}-1] : 2 \nmid i\}$. So it suffices to consider  
$$e \cdot k[Q_{2^{s+1}}] = \prod_j B_j,$$
for $B_j$ simple, where the $B_j$ are products of the $A_i$ for $2 \nmid i$. Choose one of the $B_j$ and let $B = B_j$. Then
$$B \otimes_k k_\text{sep} = \prod_j A_{i_j} = \prod_j M_2(k_\text{sep}),$$
where $A_{i_j}$ are the simple modules corresponding to the irreducible representations in the $\Gamma$-orbit on the set of irreducible representations of $k_\text{sep}$.

Since $B$ is simple, by the Artin-Wedderburn theorem we can write $B \cong M_n(D)$ for some $n$ and some division ring $D$. The center $Z(M_n(D))$ is given by scalar matrices with entries in $Z(D)$. Since $Z = Z(D)$ is an abelian division ring, it is a field. Let $t = [Z:k]$.

Note that $D \otimes_Z \overline{Z}$ is a central simple $\overline{Z}$ algebra. and the only division algebra over $\overline{Z}$ is $\overline{Z}$. So by the Artin-Wedderburn theorem $D \otimes_Z \overline{Z} \cong M_d(\overline{Z})$ for some $d$. So $\dim_{\overline{Z}} (D \otimes_Z \overline{Z}) = d^2$.  So $\dim_Z(D) = \dim_{\overline{Z}}(D \otimes_Z \overline{Z}) = d^2$. 

Note that there is a simple module corresponding to $B \cong M_n(D)$ given by 
$$V = \{ \begin{pmatrix} v_1 & 0 & \dots & 0\end{pmatrix} : v_1 \in D\} \oplus \dots \oplus \{\begin{pmatrix} 0 & \dots & 0 & v_n\end{pmatrix} : v_n \in D\}.$$ 
The dimension of $V$ over $k$ is given by
$$\dim_k(V) = ntd^2.$$ 
Note that 
\begin{align*}
D \otimes_k Z &= D \otimes_Z (Z \otimes_k Z)\\
&= D \otimes_Z Z^{t}\\
&= (D \otimes_Z Z)^t\\
&= D^t 
\end{align*}
And so
\begin{align*}
M_n(D) \otimes_k Z &= M_n(D \otimes_k Z)\\
&= M_n(D^t)\\
&= M_n(D)^{t}
\end{align*}
So for $V$ a simple $M_n(D)$-module over $k$, we have 
$$V \otimes_k Z = U_1 \oplus \dots \oplus U_t,$$
for $U_j$ irreducible over $Z$, where $U_j$ is the simple module corresponding to the $i$th copy of $M_n(D)$. Note that
\begin{align*}
M_n(D) \otimes_Z k_\text{sep} &= M_n(D \otimes_Z k_{\text{sep}})\\
&= M_n(M_d(k_\text{sep}))\\
&= M_{nd}(k_\text{sep})
\end{align*}

So over $k_\text{sep}$, we have $(U_j)_{k_\text{sep}} = W_i^{\oplus d}$ for $W_i$ irreducible over $k_\text{sep}$. So since $\dim(U_j) = nd^2$, we must have $\dim(W_i) = nd$. If $V$ corresponds to a faithful representation, then one of the $W_i$ must be faithful and so will have dimension $2$. So we have $nd = 2$. Thus 
$$B \otimes_Z k_\text{sep} = M_n(D) \otimes_Z k_\text{sep} = M_2(k_\text{sep}).$$
Thus $\dim_Z(B) = \dim_{k_\text{sep}}(B \otimes_Z k_\text{sep}) = 4$. So $B$ is a $4$-dimensional algebra over $Z$.

Let $\varphi$ be the character of the irreducible representation corresponding to $B$. $\varphi$ will be a direct sum of the characters of irreducible representations $\lambda_i$ over $k_\text{sep}$ (defined in the previous section) with $2 \nmid i$. So we have $\varphi = \oplus_{i \in J} \chi_{i}$ for some set of indices $J$ with $2 \nmid i$ for $i \in J$, where 
$$\begin{array}{c | c | c}
        & w^a & w^av \\
 \hline 
\chi_i & \epsilon^{ai} + \epsilon^{-ai} & 0 
\end{array}.$$
So 
$$k(\varphi) = k(\{\epsilon^{ai}+\epsilon^{-ai} : a \in \Z/2^s\Z, i \in J\}.$$
Since $2 \nmid i$ for all $i \in J$, by Lemma \ref{Zform} we can conclude that
$$k(\varphi) = k(\epsilon+\epsilon^{-1}).$$

Let $\Gamma = \text{Gal}(k_\text{sep}/k)$. By \cite{Sz} (Theorem 1.5.4), the functor mapping a finite \'etale $k$-algebra $A$ to the finite set $\text{Hom}_k(A, k_\text{sep})$ gives an anti-equivalence between the category of finite \'etale $k$-algebras and the category of finite sets equipped with a continuous left $\Gamma$-action, and separable field extensions give rise to sets with transitive $\Gamma$-action.  And for $x$ in a finite set $X$ with transitive $\Gamma$-action, the corresponding separable field extension is $(k_\text{sep})^{\Gamma_0}$, where $\Gamma_0 = \text{Stab}(x)$.  

Let $G = Q_{2^{s+1}}$. Note that $\varphi: G \to GL(Z(D))$ and $\varphi(g): Z(D) \to Z(D) \subset k_\text{sep}$. So for any $g \in G$, 
$$\varphi(g) \in \text{Hom}_k(Z(D),k_\text{sep}).$$ 
Thus for any $g \in G$, $Z(D) = (k_\text{sep})^{\Gamma_g}$, where $\Gamma_g = \text{Stab}_\Gamma(\varphi(g))$. So 
$$Z(D) = (k_\text{sep})^{\Gamma_0}, \text{ where } \Gamma_0 = \text{Stab}_\Gamma(\varphi).$$ 
Note that $\gamma \in \Gamma_0$ if and only if $\gamma_{\varphi(g)} = \varphi(g)$ and hence $\gamma|_{\text{Im}(\varphi)} = \text{Id}$. Thus $$\gamma \in \Gamma_0 \text{ if and only if } \gamma \in \text{Gal}(k_\text{sep}/k(\varphi)).$$ 
So $\Gamma_0 = \text{Gal}(k_\text{sep}/k(\varphi))$. Thus 
$$Z(D) = (k_\text{sep})^{\Gamma_0} = k(\varphi) = k(\epsilon+\epsilon^{-1}).$$

So $B$ is a $4$-dimensional algebra over $Z = k(\epsilon+\epsilon^{-1})$.  Note that for $e_B$ the idempotent corresponding to $B$, we have $B = e_Bk[Q_{2^{s+1}}]$, so $e_BQ_{2^{s+1}}$ spans $B$.  Note that $e_Bw^{2^{s-2}},$ $e_Bv$ satisfy the conditions 
$$(e_Bw^{2^{s-2}})^2 = -e_B = (e_Bv)^2, \text{ } (e_Bw^{2^{s-2}})(e_Bv) = -(e_Bv)(e_Bw^{2^{s-2}}).$$ 
So by Proposition \ref{quatbasis}, $B = (-1,-1)_{k(\epsilon+\epsilon^{-1})}$. Then by Corollary \ref{genQcorr},  $B$ is split if and only if there exist $x,y \in k(\epsilon+\epsilon^{-1})$ such that $x^2 + y^2 = -1$.
Therefore,
$$\ed_k(Q_{2^{s+1}},2) = \dim(V) = \begin{cases} 2[k(\epsilon+\epsilon^{-1}):k], &x^2 + y^2 = -1 \text{ has a solution in } k(\epsilon+\epsilon^{-1})\\
4[k(\epsilon+\epsilon^{-1}):k], &x^2 + y^2 = -1 \text{ has no solutions in } k(\epsilon+\epsilon^{-1})
 \end{cases}.$$

Note that $x^2+y^2=-1$ has a solution in $k(\epsilon)$ given by $x = \epsilon^{2^{s-2}}$, $y = 0$ (since $\epsilon^{2^{s-1}} = -1$).   By Lemma \ref{Gamma}, if $[-1] \notin \Gamma$, then $k(\epsilon) = k(\epsilon+\epsilon^{-1})$. So if $[-1] \notin \Gamma$, then $x^2+y^2 = -1$ has a solution in $k(\epsilon+\epsilon^{-1})$.

If $[-1] \in \Gamma$, then by Lemma \ref{Gamma} $[k(\epsilon):k] = 2[k(\epsilon+\epsilon^{-1}):k]$. So we have

\begin{align*}
\ed_k(Q_{2^{s+1}},2) &= \begin{cases}
 2[k(\epsilon):k], &[-1] \notin \Gamma\\
 [k(\epsilon):k], &[-1] \in \Gamma \text{ and }x^2 + y^2 = -1 \text{ has a solution in } k(\epsilon+\epsilon^{-1}) \\
2[k(\epsilon):k], &[-1] \in \Gamma \text{ and } x^2 + y^2 = -1 \text{ has no solutions in } k(\epsilon+\epsilon^{-1})
\end{cases}.  \qedhere
\end{align*}

\end{proof}

\section{\texorpdfstring{The Special Linear Groups - $n = 2$ or odd, $q \equiv 3 \mod 4$}{The Special Linear Groups - n = 2 or odd, q equiv 3 mod 4}}

\begin{theorem}\label{SLn2''} Let $p \neq 2$ be a prime and $q = p^r$. Let $k$ be a field with $\text{char } k \neq 2$. Assume that $q \equiv 3 \mod 4$, and let $s = v_2(q+1) + 1$. Let $\xi = \zeta_{2^{s-1}}$ in $k_\text{sep}$, and let $\Gamma' = \text{Gal}(k(\xi)/k)$.  Then
\begin{align*}
&\ed_k(SL_n(\F_q), 2)\\
&= \begin{cases}
\ed_k(GL_{2m}(\F_q),l), &n = 2m + 1\\
2[k(\xi):k], &n=2,  [-1] \notin \Gamma'\\
[k(\xi):k], &n=2, \text{  } [-1] \in \Gamma', \text{ } x^2 + y^2 = -1 \text{ has a solution in } k(\xi+\xi^{-1}) \\
 2[k(\xi):k], &n=2, \text{  }[-1] \in \Gamma', \text{ } x^2 + y^2 = -1 \text{ has no solutions in } k(\xi+\xi^{-1})
 \end{cases}
\end{align*}
\end{theorem}

By Grove (\cite{Gr}, Proposition 1.1), 
$$|SL_n(\F_q)| = \frac{|GL_n(\F_q)|}{q-1}.$$
For $q \equiv 3 \mod{4}$, we know that $v_2(q-1) = 1$, and so 
$$|SL_n(\F_q)|_2 = \frac{|GL_n(\F_q)|_2}{2} = \begin{cases} 2^{v_2(m!)} \cdot (2^{s+1})^m \cdot 2^{-1}, &n = 2m\\
\cdot 2^{v_2(m!)} \cdot (2^{s+1})^m, &n = 2m+1\end{cases}.$$

\begin{definition} For $b = x^ay^c \in SD_{2^{s+1}}$, let $\det(b)$ be defined by $\det(b) = (-1)^{a+c}$.  And for $\mbf{b} \in (SD_{2^{s+1}})^{m}$, let $\det(\mbf{b}) = \prod_{i=1}^m \det(b_i)$.
\end{definition}

\begin{lemma} For $P \in \syl_2(SL_n(\F_q))$, 
$$P \cong \begin{cases} 
\{(\mbf{b}, \tau), \in (SD_{2^{s+1}})^m \rtimes P_2(S_m) : \det(\mbf{b}) = \text{sgn}(\tau)\}, &n = 2m\\
\{(\mbf{b}, \tau, z) \in (SD_{2^{s+1}})^m \rtimes P_2(S_m) \times \Z/2\Z :
(-1)^{z}\det(\mbf{b}) = \text{sgn}(\tau)\}, &n = 2m+1
 \end{cases},$$
where the action of $P_2(S_n)$ on $\mbf{a} \in T$ is given by permuting the $a_i$.
\end{lemma}

\begin{proof}
By Proposition \ref{GLsyl2'}, the Sylow $2$-subgroups of $GL_n(\F_q)$ are isomorphic to 
$$P' \cong \begin{cases} 
(SD_{2^{s+1}})^m \rtimes P_2(S_m), &n = 2m\\ (SD_{2^{s+1}})^m \rtimes P_2(S_m) \times \Z/2\Z, &n = 2m + 1
\end{cases}.$$ 
Let 
$$P = \begin{cases} 
\{(\mbf{b}, \tau), \in P' : \det(\mbf{b}) = \text{sgn}(\tau)\}, &n = 2m\\
\{(\mbf{b}, \tau, z) \in P' :
(-1)^{z}\det(\mbf{b}) = \text{sgn}(\tau)\}, &n = 2m+1
 \end{cases}.$$
Then $P \subset SL_n(\F_q)$ and 
$$|P| = \frac{|GL_n(\F_q)|_2}{2} = |SL_n(\F_q)|_2.$$
Thus $P$ is isomorphic to a Sylow $2$-subgroup of $SL_n(\F_q)$.
\end{proof}

The proof when $n = 2m+1$ is simple:

\begin{proof}[Proof of Theorem \ref{SLn2'} for the case $n = 2m+1$]

Let $S = (SD_{2^{s+1}})^m \rtimes P_2(S_m)$, $P' = S \times \Z/2\Z$, and $P = \{(\mbf{b},\tau,z) \in P' : (-1)^z\det(\mbf{b}) = \text{sgn}(\tau)\}$. Then $S$ is isomorphic to a Sylow $2$-subgroup of $GL_{2m}(\F_q)$ and $P$ is isomorphic to a Sylow $2$-subgroup of $SL_{2m+1}(\F_q)$.  We can construct an isomorphism $\phi: P \to S$ given by $(\mbf{b},\tau,z) \mapsto (\mbf{b},\tau)$.  This map is injective since if $(\mbf{b},\tau) = (\mbf{b}',\tau')$, then for $(\mbf{b},\tau,z) \in P$, we have
$$(-1)^z\det(\mbf{b}) = \text{sgn}(\tau) = \text{sgn}(\tau') = (-1)^{z'}\det{\mbf{b}}' = (-1)^{z'}\det(B),$$
and hence
$$(-1)^z = (-1)^{z'} \Rightarrow z = z' \text{ since } z,z' \in \Z/2\Z.$$
Therefore, the Sylow $2$-subgroups of $SL_{2m+1}(\F_q)$ are isomorphic to Sylow $2$-subgroups of $GL_{2m}(\F_q)$. Thus 
\begin{align*} \ed_k(SL_{2m+1}(\F_q),2) &= \ed_k(GL_{2m}(\F_q),2). \qedhere
\end{align*}
\end{proof}

Note that since $q \equiv 3 \mod 4$, we can write $q = 3 + 4a$ for some integer $a$, and so $q+1 = 4 + 4a = 4(1+a)$. Therefore, $v_2(q+1) \geq 2$ and hence $s = v_2(q+1)+1 > 2$.

For $n = 2,$ we have
\begin{align*}P  &=  \{x^cy^d \in SD_{2^{s+1}} : 2 \divides c+d\}\\
&= \langle x^2, xy : (x^2)^{2^{s-2}} = x^{2^{s-1}} = (xy)^2, (x^2)^{2^{s-1}} = 1, (xy)x^2(xy)^{-1} = x^{-2} \rangle\\
&= \langle w, v : w^{2^{s-2}} = v^2, w^{2^{s-1}} = 1, vwv^{-1} = w^{-1} \rangle\\
&= Q_{2^s}.
\end{align*}
So in the case $n = 2$, Theorem \ref{SLn2''} follows from Proposition \ref{edQuat}.

\section{\texorpdfstring{The Special Linear Groups - $q \equiv 1 \mod 4$, $\Gamma$ trivial}{The Special Linear Groups - q equiv 1 mod 4, Gamma trivial}}

\begin{theorem}\label{SLn2} Let $p$ be a prime, $q = p^r$, and $l$ a prime Let $p \neq 2$ be a prime and $q = p^r$. Let $k$ be a field with $\text{char } k \neq 2$.   Assume that $q \equiv 1 \pmod 4$, and let $s = v_2(q-1)$. Suppose that $\Gamma = \text{Gal}(k(\zeta_{2^s})/k)$ is trivial.   Then 
$$\ed_k(SL_n(\F_q),2) = \begin{cases}
 \ed_k(GL_{n-1}(\F_q),2), &2 \nmid n \\
\ed_k(GL_n(\F_q),2) &2 \divides n
\end{cases}$$
\end{theorem}

\smallskip

By (\cite{Gr}, Proposition 1.1), 
$$|SL_n(\F_q)| = \frac{|GL_n(\F_q)|}{q-1}.$$
So
\begin{align*}
|SL_n(\F_q)|_2 &= \frac{|GL_n(\F_q)|_2}{2^{v_l(q-1)}} = 2^{s(n-1)} \cdot |S_{n}|_2\\
\end{align*}

The proof when $2\nmid n$ is simple:
\begin{proof}[Proof of Theorem \ref{SLn2} for the case $2 \nmid n$]
Note that we can embed $GL_{n-1}(\F_q)$ in $SL_n(\F_q)$ by sending the matrix $A \in GL_{n-1}(\F_q)$ to 
$$\begin{pmatrix} A & 0\\
0 & \text{det}(A^{-1})\end{pmatrix}.$$
If $2 \nmid n$, then $|S_n|_2 = |S_{n-1}|_2$, thus 
$$|SL_n(\F_q)|_2 = 2^{s(n-1)} \cdot |S_n|_2 = 2^{s(n-1)} \cdot |S_{n-1}|_2 = |GL_{n-1}(\F_q)|_2.$$
Therefore, the Sylow $2$-subgroups of $SL_n(\F_q)$ are isomorphic to Sylow $l$-subgroups of $GL_{n-1}(\F_q)$. Thus 
\begin{align*} \ed_k(SL_n(\F_q),2) &= \ed_k(GL_{n-1}(\F_q),2) = (n-1)[k(\zeta_{l^s}):k].
\qedhere
\end{align*}
\end{proof}

For the remainder of this section, we will assume that $2 \divides n$.

\begin{lemma} For $P \in \syl_2(SL_n(\F_q))$, 
$$P \cong \{(\mbf{a}, \tau) \in (\mu_{2^s})^n \rtimes P_2(S_n) : \prod_{i=1}^n a_i = \text{sgn}(\tau)\},$$
where the action of $P_2(S_n)$ on $\mbf{a} \in T$ is given by permuting the $a_i$.
\end{lemma}

\begin{proof} 
The proof is identical to that in \cite{Kni2} for $l \neq 2$. 
\end{proof}

 \begin{lemma}\label{SLZ}  If $2  \divides n$, then for $P \in \syl_2(SL_n(\F_q))$,
 $$Z(P)[2] \cong (\mu_2)^{\xi_2(n)}.$$
\end{lemma}
\begin{proof} 
The proof is identical to that in \cite{Kni2} for $l \neq 2$. 
\end{proof}

\subsection{\texorpdfstring{Case 1: $l = 2 = n$}{Case 1: l = 2 = n}}

\begin{proof}[Proof of Theorem \ref{SLn2} for the case $l \divides q-1$, $l = 2 = n$]

For $l = 2, n = 2$, we have 
\begin{align*}
P &= \{(\mbf{a}, b) \in (\mu_{2^s})^2 \rtimes \mu_2 : a_1a_2 = b\}\\
&= \{(a, ba^{-1}, b) \in (\mu_{2^s})^2 \rtimes \mu_2\}\\
&= \langle (a,a^{-1},1), (1,-1,-1) : (a,a^{-1},1)^{2^{s-1}} = (-1,-1,1) = (1,-1,-1)^2, (a,a^{-1},1)^{2^{s}}  = (1,1,1),\\ &\qquad (1,-1,-1)(a,a^{-1},1)(1,-1,-1) = (a^{-1},-a,-1)(1,-1,-1) = (-a^{-1},-a,1)\rangle\\
&\cong \langle x,y : x^{2^{s-1}} = -1 = y^2, x^{2^{s}} = 1, yxy = -x^{-1} \rangle
\end{align*}
(Note this is neither semi-dihedral nor quaternion.)

Let $\rho$ be a faithful representation of $P$ of minimum dimension (and so it is also irreducible since the center has rank $1$). Note that $\mu_{2^s} \triangleleft P$ and so by Clifford's Theorem (Theorem \ref{cliff}), $\rho|_{\mu_{2^s}}$ decomposes into a direct sum of irreducibles in the following manner:
$$\rho|_{\mu_{2^s}} \cong  \left( \oplus_{i=1}^c  \lambda_i \right)^{\oplus d}, \text{ for some } c, d,$$
 and $P/\mu_{2^s}$ acts transitively on the isomorphism classes of the $\lambda_i$ (and so the $\lambda_i$ have the same dimension). Since $\rho$ is faithful, one of the $\lambda_i$ must be non-trivial on $\mu_{2^s}[2]$. 
 
By Lemma \ref{corrlemma}, since we are assuming that $\Gamma = \text{Gal}(k(\zeta_{2^s})/k)$ is trivial, the irreducible representations of  $T \cong \mu_{2^s}$ have dimension $1$ and are given by $\Psi_a$ with $a \in \Z/2^s\Z$.

Note that for $x \in \mu_{2^s}$, $y = (1,-1,-1)$, $y(x) = -x^{-1} = x^{2^{s-1}-1}$. So \begin{align*}
y(\Psi_a)(x) &= \Psi_a(y(x))\\
&= \Psi_a(x^{2^{s-1}-1})\\
&= \text{ multiplication by } (x^{2^{s-1}-1})^a\\
&= \text{ multiplication by } x^{(2^{s-1}-1)a}\\
&= \Psi_{(2^{s-1}-1)a}(x).
\end{align*}
So since we are assuming that $\Gamma$ is trivial, we can conclude that $c = 2$. Thus 
$$\ed_k(SL_2(\F_q),2) = \dim(\rho) = cd[k(\zeta_{2^s}):k] \geq  2 = \ed_k(GL_n(\F_q),2).$$
And since $SL_2(\F_q) \subset GL_2(\F_q)$, we know that 
$$\ed_k(SL_2(\F_q),2) \leq \ed_k(GL_n(\F_q),2).$$  
Therefore,
\begin{align*}
\ed_k(SL_{2}(\F_q),2)) &= \ed_k(GL_n(\F_q),2) =  2
\qedhere
\end{align*}

\end{proof}

\subsection{\texorpdfstring{Case 2: $n = 2^t, t > 1$}{Case 2: n = 2t, t > 1}}

\begin{proof}[Proof of Theorem \ref{SLn2} for the case $n = 2^t$, $t > 1$]

Let $P = \{(\mbf{a}, \tau) \in (\mu_{2^s})^n \rtimes P_2(S_n) : \prod_{i=1}^n a_i = \text{sgn}(\tau)\}$. Note that since $(\mu_{2^s})^{n-1} \subset SL_n(\F_q) \subset GL_n(\F_q)$, we have
 $$(n-1) \leq \ed_k(SL_n(\F_q),2) \leq n.$$

Let $\rho$ be a faithful representation of $P$ of minimum dimension (and so it is also irreducible since the center has rank $1$). Then $\dim(\rho) \geq (2^t-1)$. Let $T = \{\mbf{a} \in (\mu_{2^s})^n : \prod_{i=1}^n a_i = 1\} \subset P$.  Then $T \triangleleft P$ and so by Clifford's Theorem (Theorem \ref{cliff}), $\rho|_T$ decomposes into a direct sum of irreducibles in the following manner:
$$\rho|_{T} \cong  \left( \oplus_{i=1}^c  \lambda_i \right)^{\oplus d}, \text{ for some } c, d,$$ 
with the $\lambda_i$ non-isomorphic, and $P/T$ acts transitively on the isomorphism classes of the $\lambda_i$. So the $\lambda_i$ have the same dimension and the number of $\lambda_i$, $c$, divides $|P/T|$, which is a power of $l$. Since $\rho$ is faithful, one of the $\lambda_i$ must be non-trivial on $T[2]$.

 By Lemma \ref{corrlemmaPSLn}, since we are assuming that $\Gamma = \text{Gal}(k(\zeta_{2^s})/k)$ is trivial, the irreducible representations of $T$ have dimension $1$ and are given by $\Psi_{\mbf{a}}|_T$ , with $\mbf{a} \in H$.

If $c = 1$, then since $\rho|_T = \oplus_{d \text{ times}} \lambda$ is faithful, we must have $\lambda$ is faithful. Recall that $\ed_k(T) = \ed_k((\mu_{2^s})^{n-1}) = (n-1)[k(\zeta_{l^s}):k]$. Since $n = 2^t$ and $t > 1$, we must have $n > 2$ and so $n - 1 > 1$. Thus there are no $1$-dimensional faithful representations of $T$. But $\dim(\lambda) = 1$, so we can conclude that $\lambda$ is not faithful. So we cannot have $c = 1$, and thus since $c$ is a power of $2$ we can conclude that $c$ is a multiple of $2$.   Thus $\dim(\rho) $ is a multiple of $2$.  So since we know that 
$$(2^t-1) \leq \dim(\rho) \leq 2^t,$$
we can conclude that $\dim(\rho) = 2^t$. Thus 
\begin{align*} \ed_k(SL_{2^t}(\F_q),2)) &= 2^t = \ed_k(GL_{2^t}(\F_q),2).
\qedhere
\end{align*}
\end{proof}

\subsection{\texorpdfstring{Case 3: $2 \divides n$, $n \neq 2^t$}{Case 3: 2 divides n, n neq 2t}}

\begin{proof}[Proof of Theorem \ref{SLn2} for the case $2 \divides n$, $n \neq 2^t$] \text{ }

Let $P = \{(\mbf{a}, \tau) \in (\mu_{2^s})^n \rtimes P_2(S_n) : \prod_{i=1}^n a_i = \text{sgn}(\tau)\}$. Let $\rho$ be a faithful representation of $P$ of minimum dimension.  Let $\rho = \oplus_{j=1}^{\xi_2(n)} \rho_j$ be the decomposition into irreducibles. Let $C = Z(P)$. By Lemma \ref{BMKS3.5}, if $\chi_j$ are the central characters of $\rho_j$, then $\{\chi_j|_{C[2]}\}$ form a basis for $\widehat{C[l]}$. Let $\mbf{b}^j$  be the dual basis for $C[2]$ so that $\rho_j(\mbf{b}^i)$ is trivial for $i \neq j$.  

For $j \leq \xi_2(n)$, let
$$P_j = \{(\mbf{b,\tau}) \in P: b_i = 1 \text{ for } i \notin I_j, \text{ } \tau \text{ acts trivially on } i \text{ for } i \notin I_j\}.$$
For $j \leq \xi_2(n)$, define $\mbf{e}^j$ by 
$$(\mbf{e}^j)_i = \begin{cases} \zeta_2, &i \in I_j\\
1, &i \notin I_j \end{cases}.$$ Then $\{\mbf{e}^j\}$ is a basis for $C[2]$.  Write $\mbf{b}^j = \oplus_{i} a_{i,j}\mbf{e}^i$. Then $\rho_j$ will be non-trivial on $P_j \cap C[2]$ if and only if $a_{j,j} \neq 0$. Note that $(a_{i,j})$ is the change of basis matrix from $\{\mbf{e}^j\}$ to $\{\mbf{b}^j\}$. Since it is a change of basis matrix, it must be invertible. By Lemma \ref{claim1}, we can rearrange the $\mbf{b}^j$ such that for all $i$ $a_{i,i} \neq 0$ in the change of basis matrix from $\{\mbf{e}^j\}$ to $\{\mbf{b}^j\}$. And so we can rearrange the $\rho_j$ such that $\chi_j|_{C[2]}$ is non-trivial on $P_j \cap C[2]$ and thus $\rho_j$ is non-trivial on $P_j \cap C[2]$.

Note that 
$P_j$ is isomorphic to a Sylow $2$-subgroup of $SL_{2^{k_j}}(\F_q)$.  And  $P_j \cap C[2]$ is precisely $Z(P_j)[l]$, which has rank $1$.  Thus, since $\rho_j$ is non-trivial on $P_j \cap C[2]$, we can conclude that $\rho_j|_{P_j}$ is a faithful representation of $P_j$.  And we know by the cases $n = 2$, $n = 2^t$ that $\ed_k(SL_{2^{k_j}}(\F_q),2) = 2^{k_j}.$ So we can conclude that 
$$\dim(\rho_j) \geq 2^{k_j}.$$
Thus
\begin{align*}
\dim(\rho) &= \sum_{j=1}^{\xi_2(n)} \dim(\rho_j) \geq \sum_{j=1}^{\xi_2(n)} 2^{k_j} = n
\end{align*}
So
$$\ed_k(SL_n(\F_q),2) \geq n.$$
Thus, since we also have 
$$\ed_k(SL_{n}(\F_q),2) \leq \ed_k(GL_{n}(\F_q),2) = n $$
we can conclude that 
\begin{align*} \ed_k(SL_n(\F_q),2)) &= n = \ed_k(GL_n(\F_q),2).
\qedhere
\end{align*}

\end{proof}

\section{\texorpdfstring{The Projective Special Linear Groups and Quotients of $SL_n(\F_q)$ by cyclic subgroups of the center - $q \equiv 1 \mod 4$, $\Gamma$ trivial}{The Projective Special Linear Groups and Quotients of SLn(Fq) by cyclic subgroups of the center - q equiv 1 mod 4, Gamma trivial}}

$PSL_n(\F_q)$ is defined to be 
$$PSL_n(\F_q) = SL_n(\F_q)/Z(SL_n(\F_q)).$$
The center of $SL_n(\F_q)$ is given by 
$$Z(SL_n(\F_q)) = \{x\text{Id}_n : x \in \F_q, x^n = 1\}.$$ 
By looking at the Sylow $2$-subgroup calculated in the section on $SL_n(\F_q)$ and modding by $Z(SL_n(\F_q))$, we see that a Sylow $2$-subgroup of $PSL_n(\F_q)$ is isomorphic to
$$P \cong \{(\mbf{b}, \tau) \in (\mu_{2^s})^n/ \{(x,\dots,x) : x^n = 1\} \rtimes P_2(S_n) : \prod_{i=1}^n b_i = \text{sgn}(\tau)\}.$$

Note that for $n' | n$, we obtain a subgroup of $SL_n(\F_q)$ containing  $PSL_n(\F_q)$ of order $\frac{|SL_n(\F_q)|}{(n', \text{ }q-1)}$ by taking the quotient of $SL_n(\F_q)$ by the cyclic subgroup of order $n'$ given by $\{x\text{Id} : x \in \F_q, \text{ } x^{n'} = 1\}$. The Sylow $l$-subgroups will be given by

$$P \cong \{(\mbf{b}, \tau) \in (\mu_{2^s})^n/ \{(x,\dots,x) : x^{n'} = 1\} \rtimes P_2(S_n) : \prod_{i=1}^n b_i = \text{sgn}(\tau)\}.$$

\begin{theorem}\label{PSLn2} Let $p \neq 2$ be a prime and $q = p^r$. Let $k$ be a field with $\text{char } k \neq 2$.   Assume that $q \equiv 1 \pmod 4$, and let $s = v_2(q-1)$. Suppose that $\Gamma = \text{Gal}(k(\zeta_{2^s})/k)$ is trivial. Let $G = SL_n(\F_q)/\{x\text{Id} : x \in \F_q, \text{ } x^{n'} = 1\}$. Let $v = \min(v_2(n'),s)$. Then if $2 \nmid n'$, then $\ed_k(G,2) = \ed_k(SL_n(\F_q),2)$. And if $2 \divides n'$, then 
$$\ed_k(G,l) = \begin{cases} 
2, &n' = n = 2\\
2^{2t-2}, &n=2^t, \text{ } t > 2, \text{ } v = 1\\
\ed_k(PGL_n(\F_q),2) = 2^{2t-1}, &n = 2^t, \text{ } t > 2, \text{ } v > 1\\
\ed_k(PGL_n(\F_q),2) = 2^{v_2(n)}(n-2^{v_2(n)}), &n \neq 2^t
\end{cases}.$$
Note that for $n' = n,$ $G = PSL_n(\F_q).$
\end{theorem}

\smallskip

If $l \nmid q - 1$ or $l \nmid n'$, then the Sylow $l$-subgroups of $G$ are isomorphic to the Sylow $l$-subgroups of $SL_n(\F_q)$. So we need only prove the theorem when $l \divides q-1,$ $l \divides n'$. Thus in this section, we will assume $l \divides q - 1$ and $l \divides n'$ (and hence $l \divides n$ since $n' \divides n$).

Note that since $q \equiv 1 \mod 4$, we know that $s = v_2(q-1) > 1$.

\subsection{\texorpdfstring{Case 1: $2 \divides n'$, $n = 2^t$}{Case 1: 2 divides n', n = 2t}}

\begin{definition} For $n = 2^t$, $1 \leq j \leq l$, let $J_j$ denote the $j$th sub-block of $l^{2-1}$ entries in $\{1, \dots, 2^t\}$, and let  $A_j = \sum_{i \in J_j} a_i$. \end{definition}

\subsubsection{\texorpdfstring{Case 1a: $n' = n = 2$}{Case 1a: n' = n = 2}}
\begin{proof}[Proof of Theorem \ref{PSLn2} in the case $n' = n = 2$]

For $n' = n = 2$, we have 
$$P = \{(\mbf{b}, z) \in (\mu_{2^s})^2/\{(x,x)\} \rtimes \Z/2\Z : b_1b_2 = (-1)^z\}.$$

Note that $(\mu_{2^s})^2/\{(x,x)\} \cong \mu_{2^s} \times \mu_{2^{s-1}}$ via the isomorphism 
$$(b_1,b_2) \mapsto \begin{cases} (b_1, b_2), &\text{ if } \zeta_{2} = (\zeta_{2^{s}})^{2^{s-1}} \nmid b_2\\
(\zeta_2 b_1, \zeta_2 b_2), &\text{ if } \zeta_2 \divides b_2 \end{cases}.$$
So we have 
$$\ed_k(PSL_n(\F_q),2) \geq \ed_k(\mu_{2^s} \times \mu_{2^{s-1}}) = [k(\zeta_{2^s}):k] + [k(\zeta_{2^{s-1}}):k].$$

Let $\phi: (\mu_{2^s})^2/\{(x,x)\} \to V$ be a faithful representation of $(\mu_{2^s})^2/\{(x,x)\}$ of dimension $[k(\zeta_{2^s}):k] + [k(\zeta_{2^{s-1}}):k]$. Define $\rho: P \to V$ by
$$\rho(\mbf{b},z)= \phi(\mbf{b}).
$$
I claim that $\rho$ is a faithful representation. Proof: Suppose that 
$$\rho(\mbf{b},z) = \rho(\mbf{b}',z').$$
Then
$$\phi(\mbf{b}) = \phi(\mbf{b}'),$$
and since $\phi$ is faithful, this means that $\mbf{b} = \mbf{b}'$. So we must have $(b_1,b_2) = (b_1',b_2')$ and hence
$$b_1b_2 = b_1'b_2'.$$ And since $(\mbf{b},z), (\mbf{b}',z') \in P$, we know that 
$$b_1b_2 = (-1)^z \text{ and } b_1'b_2' = (-1)^{z'}.$$
So $(-1)^z = (-1)^{z'}$. Therefore, $z = z' \mod 2$. But since $z$ and $z'$ are either $0$ or $1$, this means that $z = z'$. Therefore, $\rho$ is faithful. Thus for $n' = n = 2$,
\begin{align*}
&\ed_k(PSL_2(\F_q),2) = [k(\zeta_{2^s}):k] + [k(\zeta_{2^{s-1}}):k] = 2\\
&\text{ since we are assuming that } \text{Gal}(k(\zeta_{2^s})/k) \text{ is trivial}. \qedhere
\end{align*}

\end{proof}

\subsubsection{\texorpdfstring{The center of a Sylow $2$-subgroup in the case $n = 2^t$, $t > 1$}{The center of a Sylow l-subgroup in the case n = 2t, t > 1}}

\begin{lemma}\label{ZPSLn2} For $P \in \syl_2(PSL_n(\F_q))$ in the case $n = 2^t$, $t > 1$
$$Z(P)[2] \cong \mu_2.$$
\end{lemma}
\begin{proof}

Let $P = \{(\mbf{b}, \tau) \in (\mu_{2^s})^n/ \{(x,\dots,x) : x^{n'} = 1\} \rtimes P_2(S_n) : \prod_{i=1}^n b_i = \text{sgn}(\tau)\}$. Fix $(\mbf{b},\tau) \in P$. Then for $(\mbf{b}', \tau') \in P$,
$$(\mbf{b},\tau)(\mbf{b}',\tau') = (\mbf{b} \tau(\mbf{b}'), \tau\tau') \text{ and } (\mbf{b}', \tau')(\mbf{b},\tau) = (\mbf{b}' \tau'(\mbf{b}), \tau'\tau).$$
Thus $(\mbf{b},\tau)$ is in the center if and only if $\tau \in Z(P_2(S_{n}))$ and
$$\mbf{b}\tau(\mbf{b}') = \mbf{b}'\tau'(\mbf{b}) \mod \{(x,\dots,x) : x^{n'} = 1\}$$ for all $\mbf{b}',\tau'$. Choosing $\tau' = \Id$, we see we must have $\mbf{b}\tau(\mbf{b}') = \mbf{b}'\mbf{b} \mod \{(x,\dots,x) : x^{n'} = 1\}$.  Thus we must have $\tau(\mbf{b}') = \mbf{b}' \mod \{(x,\dots,x) : x^{n'} = 1\}$ for all $\mbf{b}'$ with $(\mbf{b}',\text{Id}) \in P$.

If $\tau \neq \text{Id}$, then without loss of generality assume $\tau(1) = 2$ and $\tau(2) = 3$. Since $t > 1$, $n \geq 4$, so choosing 
$$b'_1 = \zeta_l, b'_2 = \zeta_l, b'_3 = 1,b'_4 = \zeta_l^{-2}, \text{ and all other entries } 1,$$ we have $\mbf{b}' \in T$. But 
$$\tau(\mbf{b}')_2 = \zeta_l = b'_2,$$ 
whereas 
$$\tau(\mbf{b'})_3 = \zeta_l \neq 1 = b'_3.$$ 
So $\tau(\mbf{b'}) \neq \mbf{b'} \mod \{(x,\dots,x) : x^{n'} = 1\}$.

So for any $\tau \neq \text{Id}$, we can choose a $\mbf{b}'$ for which $\tau(\mbf{b}') \neq \mbf{b}' \mod \{(x,\dots,x) : x^{n'} = 1\}$, so we can conclude that we must have $\tau = \Id$. 

We also need $\tau'(\mbf{b}) = \mbf{b} \mod \{(x,\dots,x)\}$ for all $\tau' \in P_2(S_n)$. Note that for each $i, i'$ in the same $J_j$, there exists $\tau' \in P_2(S_n)$ that sends $i$ to $i'$ and fixes some other index. Since there is an index that is fixed by $\tau$, in order for $\tau(\mbf{b})$ to equal $\mbf{b}\mbf{x}$ for $\mbf{x} = (x,x,\dots,x)$, we must have $x = 0$ and so $\tau(\mbf{b}) = \mbf{b}$. So $b_1 = \dots = b_{2^{t-1}}$, $b_{2^{t-1}+1} = \dots = b_{l^{t}}.$  If we consider the last generator, $\sigma_1^{t}$, we see that we must have $b_{i+2^{t-1}} = b_{i}x$ for some fixed $x = \zeta_2^{a}$.  Thus $\mbf{b}$ must be of the form
$$\mbf{b} = (b\zeta_2^a, \dots b\zeta_2^a, b, \dots, b).$$
In $PSL_{l^t}(\F_z)$, the set of all elements of this form is a cyclic group of order $2$ generated by
$$\mbf{b} = (\zeta_2, \dots, \zeta_2, 1, \dots, 1).$$
So we have 
\begin{align*} Z(P) &\cong \mu_2. \qedhere
\end{align*}
\end{proof}

\subsubsection{\texorpdfstring{Case 1b: $n = 2^t$ with $t > 1$}{Case 1b: n = 2t with t > 1}}

For the proof of Theorem \ref{PSLn2} in the case $2 \divides n'$, $n = 2^t$ with $t > 1$, we will need the following lemma.

\begin{lemma}\label{irrHPSLn3} Let $H = (\Z/2^s\Z)^{n}/\{(x,\dots,x)\}$, $n = 2^t$ with $t > 1$, $v = \min(v_2(n'),s)$, and $\mbf{a} \in H$ with 
$$\sum_{i=1}^{n} a_i = 0 \mod 2^v \text{ and } \Psi_{\mbf{a}} \text{ non-trivial on } Z(P)[2].$$ 
Then
$$|\text{orbit}(\mbf{a})| \geq \begin{cases} 
2^{2t-2}, &v = 1\\
2^{2t-1}, &v > 1
 \end{cases}$$ under the action of $P_2(S_n)$ on $H$. 
\end{lemma}

\begin{proof} 
For the proof, see the Appendix (\ref{App2}). \qedhere
\end{proof}

\begin{corollary}\label{irrHPSLn3cor} Suppose that $\Gamma = \text{Gal}(k(\zeta_{2^s})/k)$ is trivial. Let  $H = (\Z/2^s\Z)^{n}/\{(x,\dots,x)\}$. Assume that $\Psi_{\mbf{a}}$ is non-trivial on $Z(P)[2]$. Then for $n = 2^t$ with $t > 1$, $v = \min(v_2(n'),s)$, $\mbf{a} \in H/\Gamma$ with $\sum_{i=1}^{n} a_i = 0 \mod 2^v,$ we can conclude that $$|\text{orbit}(\mbf{a})| \geq \begin{cases} 
2^{2t-2}, &v = 1\\
2^{2t-1}, &v > 1
 \end{cases}$$
under the action of $P_l(S_n)$ on $\widehat{T'}$.  
\end{corollary}
\begin{proof}
By Lemma \ref{changeperspPSLn}, the orbit of $\Psi_{\mbf{a}}|_T$ under the action of $P_2(S_n)$ has the same size as the orbit of $\mbf{a}$ under the action of $P_2(S_n)$ on $H/\Gamma$. And if $\Gamma$ is trivial, then this is the same as the orbit of $\mbf{a}$ under the action of $P_2(S_n)$ on $H$. And by Lemma \ref{irrHPSLn3}, the orbit  of $\mbf{a}$ under the action of $P_2(S_n)$ on $H$ has size at least 
$$\begin{cases} 
2^{2t-2}, &v = 1\\
2^{2t-1}, &v > 1
 \end{cases}.$$ 
Therefore the orbit of $\Psi_{\mbf{a}}|_T$ has size at least 
\begin{align*}
&\begin{cases} 
2^{2t-2}, &v = 1\\
2^{2t-1}, &v > 1
 \end{cases}. \qedhere 
 \end{align*}
\end{proof}

Granting this lemma, we can complete the proof in the case $n = 2^t$ with $t>1$.
  
\begin{proof}[Proof of Theorem \ref{PSLn2} for the case $2 \divides n'$, $n = 2^t$ with $t > 1$.]

Let $\rho$ be a faithful representation of $P$ of minimum dimension (and so it is also irreducible since the center has rank $1$.)   Let $T' = \{\mbf{a} \in (\mu_{2^s})^n : \prod_{i=1}^n a_i = 1\}/\{(x,\dots,x)\} \subset P$.  Then $T' \triangleleft P$ and so by Clifford's Theorem (Theorem \ref{cliff}), $\rho|_{T'}$ decomposes into a direct sum of irreducibles in the following manner:
$$\rho|_{T'} \cong  \left( \oplus_{i=1}^c  \lambda_i \right)^{\oplus d}, \text{ for some } c, d,$$ 
with the $\lambda_i$ non-isomorphic, and $P_2(S_n)$ acts transitively on the $\lambda_i$, so the $\lambda_i$ have the same dimension and the number of $\lambda_i$, $c$, divides $|P_2(S_n)|$ (which is a power of $2$), so $c$ is a power of $2$. Also, since $\rho$ is faithful, it is non-trivial on $Z(P)[2]$, thus one of the $\lambda_i$ must be non-trivial on $Z(P)[2]$. Without loss of generality, assume that $\lambda_1$ is non-trivial on $Z(P)[2]$.
 
 Note that the irreducible representations of $T'$ are in bijection with irreducible representations of $T$ which are trivial on $\{(x,\dots,x) : x^{n'} = 1\}$.   
% By Lemma \ref{corrlemmaPSLn}, the irreducible representations of $T$ are given by $\Psi_{\mbf{a}}|_T$, with $\mbf{a} \in H/\Gamma$, for $\Gamma = \text{Gal}(k(\zeta_{2^s})/k)$, and if $\Psi_{\mbf{a}}$ is non-trivial on $T[2]$, then $\Psi_\mbf{a}$ has dimension $[k(\zeta_{2^s}):k]$. 
%  Since $\lambda_1$ is non-trivial on $Z(P)[2]$, it must be non-trivial on $T[2]$, so we must have $\dim(\lambda_1) = [k(\zeta_{2^s}):k]$, and so $\dim(\lambda_i) = [k(\zeta_{2^s}):k]$ for all $i$.
 By Lemma \ref{corrlemmaPSLn}, since we are assuming that $\Gamma = \text{Gal}(k(\zeta_{2^s})/k)$ is trivial, the irreducible representations of $T$ have dimension $1$ and are given by $\Psi_{\mbf{a}}|_T$ , with $\mbf{a} \in H$.
 
Note that for $\mbf{x} = (x,\dots,x)$, $\psi_{\mbf{a}}(\mbf{x}) = x^{\sum_{i=1}^n a_i}$.  So $\psi_{\mbf{a}}|_T \in \widehat{T}/\Gamma$ will be trivial on $\{(x,\dots,x) : x^{n'} = 1\}$ if and only if $\sum_{i=1}^n a_i = 0 \mod 2^v$, where $v = \min(v_2(n'),s)$. So $\lambda_1 \cong \Psi_{\mbf{a}}|_T$ for some $\mbf{a} \in H/\Gamma$ with $\sum_{i=1}^n a_i = \mod 2^v$.

Then since $\lambda_1$ is non-trivial on $Z(P)[2]$ and we are assuming that $\Gamma = \text{Gal}(k(\zeta_{2^s})/k)$ is trivial,  by Corollary \ref{irrHPSLn3cor} the orbit the orbit of $\lambda_1$ under the action of $P_2(S_n)$ has size at least $$\begin{cases} 
2^{2t-2}, &v = 1\\
2^{2t-1}, &v > 1
 \end{cases}.$$ Thus
$$\dim(\rho) \geq  \begin{cases} 
2^{2t-2}, &v = 1\\
2^{2t-1}, &v > 1
 \end{cases}.$$
% $$\dim(\rho) \geq  \begin{cases} 2[k(\zeta_{2^s}):k], &t = 2, \text{ } s = 1\\
%2^{2t-2}[k(\zeta_{2^s}):k], &t=2, \text{ } s > 1, \text{ } v = 1\\
%&t > 2, \text{ } v = 1\\
%2^{2t-1}[k(\zeta_{2^s}):k], \text{ } &t = 2, \text{ } v > 1\\
%&t > 2, \text{ } v > 1
% \end{cases}.$$
For $v > 1$, since $G \subset PGL_n(\F_q)$,
$$\ed_k(G,l) \leq \ed_k(PGL_n(\F_q,l) = 2^{2t-1}.$$
Therefore for $n = 2^t$ with $v > 1$,
\begin{align*}
\ed_k(G,l) &= \ed_k(PGL_n(\F_q,l) = 2^{2t-1}.
\end{align*}

For $v = 1$, we can construct a faithful representation dimension $$2^{2t-2}$$ 
%$$\begin{cases} 2[k(\zeta_{2^s}):k], &t = 2, \text{ } s = 1\\
%2^{2t-2}[k(\zeta_{2^s}):k], &t > 2 \text{ or } s > 1 \end{cases}$$
in the following manner.  Let $\mbf{a} = (1, 0, \dots, 0, 1, 0, \dots, 0)$ and consider
$$\Psi_{\mbf{a}}|_T:  T' \to GL(k(\zeta_{2^s})) = GL(k).$$
%$$\Psi_{\mbf{a}}|_T:  T' \to GL(k(\zeta_{2^s})) = GL_d(k),$$  
%where $d = [k(\zeta_{2^s}):k] = 1$. 

For $v = 1$, the orbit of $\mbf{a}$ under the action of $P_2(S_n)$ on $H$ has size $2^{2t-2}$ given by the images under the action of $P_2(S_{2^{t-1}}) \times P_2(S_{2^{t-1}})$.

So the orbit of $\Psi_{\mbf{a}}|_T$ under the action of $P_2(S_{n})$ on the irreducible representations of $T'$ (not isomorphism classes) has size 
$\text{orbit}(\Psi_{\mbf{a}}|_T) = 2^{2t-2}.$
Let $\text{Stab}_{\mbf{a}}$ be the stabilizer of $\Psi_{\mbf{a}}$ in $P_2(S_{n})$ (which has order $\frac{|P_l(S_{n})|}{|\text{orbit}(\Psi_{\mbf{a}}|_T)|}$). We can extend $\Psi_{\mbf{a}}$ to $T' \rtimes \text{Stab}_{\mbf{a}}$ by defining $\Psi_{\mbf{a}}(\mbf{b},\tau) = \tau_{\Psi_{\mbf{a}}}(\mbf{b}) = \Psi_{\mbf{a}}(\mbf{b})$ (since $\tau \in \text{Stab}_{\mbf{a}}$). Let $\rho = \text{Ind}_{T' \rtimes \text{Stab}_{\mbf{a}}}^P \Psi_{\mbf{a}}$. Then $\rho$ has dimension 
$$[P_2(S_{n}): \text{Stab}_{\mbf{a}}]\dim(\Psi_{\mbf{a}}) = 2^{2t-2},$$ and $\rho$ is non-trivial (and hence faithful) on $Z(P)$. So this is a faithful representation of $P$ of dimension  $$2^{2t-2}.$$
Therefore for $n = 2^t$ with $v > 1$,
\begin{align*}
\ed_k(G,2)) &= 2^{2t-2}.
\end{align*}
\qedhere

%Then $\rho$ has dimension 
%$$[P_2(S_{n}): \text{Stab}_{\mbf{a}}]\dim(\Psi_{\mbf{a}}) = \begin{cases} 2[k(\zeta_{2^s}):k], &t = 2, \text{ } s = 1\\
%2^{2t-2}[k(\zeta_{2^s}):k], &t > 2 \end{cases},$$ and $\rho$ is non-trivial (and hence faithful) on $Z(P)$. So this is a faithful representation of $P$ of dimension  $$\begin{cases} 2[k(\zeta_{2^s}):k], &t = 2, \text{ } s = 1\\
%2^{2t-2}[k(\zeta_{2^s}):k], &t > 2 \text{ or } s > 1 \end{cases}.$$
%Thus we have shown that for $\Gamma = \text{Gal}(k(\zeta_{2^s})/k)$ trivial, $n = 2^t$ with $t>1$ and $v = 1$,
%\begin{align*}
%\ed_k(G,2)) &= \begin{cases} 2[k(\zeta_{2^s}):k], &t = 2, \text{ } s = 1\\
%2^{2t-2}[k(\zeta_{2^s}):k], &t > 2, \text{ or } s > 1 \end{cases}.
%\qedhere

\end{proof}

\subsection{\texorpdfstring{Case 2: $2\divides n'$, $n \neq 2^t$}{Case 2: 2 divides n', n neq 2t}}

\begin{definition} For $n \neq 2^t$, $1 \leq j \leq \xi_l(n)$, let $A_j = \sum_{i \in I_j} a_i$. \end{definition}

For the proof of Theorem \ref{PSLn2} in the case $2 \divides n'$, $n \neq 2^t$, we will need the following lemmas.

\begin{lemma}\label{ZPSLn1} For $$P  = \{(\mbf{b}, \tau) \in (\mu_{2^s})^n/ \{(x,\dots,x) : x^{n'} = 1\} \rtimes P_2(S_n) : \prod_{i=1}^n b_i = \text{sgn}(\tau)\}$$ in the case $2 \divides n$, $n \neq l^t$,
$$Z(P)[l] \cong (\mu_l)^{\xi_l(n)-1}.$$
\end{lemma}

\begin{proof}[Proof of Lemma \ref{ZPSLn1}]
Fix $(\mbf{b},\tau) \in P$. Then for $(\mbf{b}', \tau') \in P$,
$$(\mbf{b},\tau)(\mbf{b}',\tau') = (\mbf{b} \tau(\mbf{b}'), \tau\tau') \text{ and } (\mbf{b}', \tau')(\mbf{b},\tau) = (\mbf{b}' \tau'(\mbf{b}), \tau'\tau).$$
Thus $(\mbf{b},\tau)$ is in the center if and only if $\tau \in Z(P_l(S_{n}))$ and
$$\mbf{b}\tau(\mbf{b}') = \mbf{b}'\tau'(\mbf{b}) \mod \{(x,\dots,x) : x^{n'} = 1\}$$ for all $\mbf{b}',\tau'$. Choosing $\tau' = \Id$, we see we must have $\mbf{b}\tau(\mbf{b}') = \mbf{b}'\mbf{b} \mod \{(x,\dots,x) : x^{n'} = 1\}$.  Thus we must have $\tau(\mbf{b}') = \mbf{b}' \mod \{(x,\dots,x) : x^{n'} = 1\}$ for all $\mbf{b}'$ with $(\mbf{b}',\text{Id}) \in P$.

If $\tau \neq \text{Id}$, then without loss of generality assume $\tau(1) = 2$ and $\tau(2) = 3$. Since $n \neq l^t$, we must have $n > 3$, so choosing 
$$b'_1 = \zeta_l, b'_2 = \zeta_l, b'_3 = 1,b'_4 = \zeta_l^{-2}, \text{ and all other entries } 1,$$ we have $(\mbf{b}',\text{Id}) \in P$. But 
$$\tau(\mbf{b}')_2 = \zeta_l = b'_2,$$ 
whereas 
$$\tau(\mbf{b'})_3 = \zeta_l \neq 1 = b'_3.$$ 
So $\tau(\mbf{b'}) \neq \mbf{b'} \mod \{(x,\dots,x) : x^{n'} = 1\}$. Thus for any $\tau \neq \text{Id}$, we can choose a $(\mbf{b}', \text{Id}) \in P$ for which $\tau(\mbf{b}') \neq \mbf{b}' \mod \{(x,\dots,x) : x^{n'} = 1\}$, so we can conclude that we must have $\tau = \Id$. 

We also need $\tau'(\mbf{b}) = \mbf{b} \mod \{(x,\dots,x) : x^{n'} = 1\}$ for any $(\mbf{b}, \tau') \in P$. And for any $\tau' \in P_l(S_n)$, we can find $\mbf{b}'$ such that $(\mbf{b}',\tau' \in P$  So we need  $\tau'(\mbf{b}) = \mbf{b} \mod \{(x,\dots,x) : x^{n'} = 1\}$ for any $\tau' \in P_l(S_n)$.  Since $n \neq l^t$, for each $i, i'$ in the same $I_j$, there exists $\tau' \in P_l(S_n)$ that sends $i$ to $i'$ and fixes some other index. Since there is an index that is fixed by $\tau'$, in order for $\tau'(\mbf{b})$ to equal $\mbf{b}\mbf{x}$ for $\mbf{x} = (x,\dots,x)$, we must have $x = 1$ and so $\tau'(\mbf{b}) = \mbf{b}$. So $b_i = b_{i'}$ for $i,i'$ in the same $I_j$. Let $\mbf{b}^j$ be given by 
$$(\mbf{b}^j)_i = \begin{cases} \zeta_l, &i \in I_j\\
1, &i \notin I_j \end{cases}.$$ 
Note that since $l \divides n$, $\prod_{i=1}^n (\mbf{b}^j)_i = 1 = \text{sgn}(\Id)$; so $(\mbf{b}^j, \text{Id}) \in P$. Then
\begin{align*} Z(P)[l] &= \langle \mbf{b}^j \rangle_{j=1}^{\xi_l(n)}/\{(x,\dots,x) : x^{n'} = 1\}\\
&\cong \langle \mbf{b}^j \rangle_{j=1}^{\xi_l(n)-1} \text{ since } l \divides n'\\
&\cong (\mu_{l})^{\xi_l(n)-1}. \qedhere
\end{align*}

\end{proof}

\begin{lemma}\label{irrHPSLn1} Let $H = (\Z/l^s\Z)^{n}/\{(x,\dots,x)\}$, $n \neq 2^t$, $v = \min(v_2(n'),s)$, $j_1 \in \{1, \dots, \xi_2(n)\}$, and $\mbf{a} \in H$ with $$\sum_{i=1}^{n} a_i = 0 \mod 2^v, \text{ } A_{j_1}\text{ invertible}.$$ 
Then 
$$|\text{orbit}(\mbf{a})| \geq 2^{k_{j_1}+v_l(n)}$$
under the action of $P_2(S_n)$ on $H$. 
\end{lemma}

\begin{proof} 
For the proof, see the Appendix (\ref{App3}). \qedhere
\end{proof}

\begin{corollary}\label{irrHPSLn1cor} Suppose that $\Gamma = \text{Gal}(k(\zeta_{2^s})/k)$ is trivial. Let $H = (\Z/l^s\Z)^{n}/\{(x,\dots,x)\}$, $n \neq 2^t$, $v = \min(v_2(n'),s)$, $j_1 \in \{1, \dots, \xi_2(n)\}$, and $\mbf{a} \in H$ with $$\sum_{i=1}^{n} a_i = 0 \mod 2^v, \text{ } A_{j_1}\text{ invertible}.$$ 
Then $$|\text{orbit}(\Psi_{\mbf{a}}|_T)| \geq 2^{k_{j_1}+v_2(n)}$$
under the action of $P_2(S_n)$ on $\widehat{T'}$. 
\end{corollary}
\begin{proof}
By Lemma \ref{changeperspPSLn}, the orbit of $\Psi_{\mbf{a}}|_T$ under the action of $P_l(S_n)$ has the same size as the orbit of $\mbf{a}$ under the action of $P_2(S_n)$ on $H/\Gamma$. And if $\Gamma$ is trivial, then this is the same as the orbit of $\mbf{a}$ under the action of $P_2(S_n)$ on $H$. And by Lemma \ref{irrHPSLn1}, the orbit  of $\mbf{a}$ under the action of $P_2(S_n)$ on $H$ has size at least $2^{k_{j_1}+v_2(n)}$. Therefore the orbit of $\Psi_{\mbf{a}}|_T$ has size at least $2^{k_{j_1}+v_2(n)}.$ \qedhere
\end{proof}

Granting these lemmas, we can complete the proof in the case $n \neq 2^t$.

\begin{proof}[Proof of Theorem \ref{PSLn2} for the case $2 \divides n'$, $n \neq 2^t$.]

Recall that
$$P = \{(\mbf{b}, \tau) \in (\mu_{2^s})^n/\\
 \{(x,\dots,x) : x^{n'} = 1\} \rtimes P_2(S_n) : \prod_{i=1}^n b_i = \text{sgn}(\tau)\}.$$
  Let $\rho$ be a faithful representation of $P$ of minimum dimension.  Let $\rho = \oplus_{j=1}^{\xi_2(n)-1} \varphi_j$ be the decomposition into irreducibles. Let $C = Z(P)$. For $j \leq \xi_l(n)-1$, let 
 $$T_j = \{\mbf{b} \in (\mu_{2^s})^n : \prod_{i=1}^n b_i = 1, \text{ }b_i = 1 \text{ for } i \notin I_j\}/\{(x,\dots,x) : x^{n'} = 1\}.$$
  By the same reasoning as for $PGL_n(\F_q)$, we can rearrange the $\rho_j$ such that $\chi_j|_{C[2]}$ is non-trivial on $T_j \cap C[2]$ and thus $\varphi_j$ is non-trivial on $T_j \cap C[2]$.

Fix $j \leq \xi_2(n)-1$ and let $\varphi = \varphi_j$. Let $T' = \{\mbf{b} \in (\mu_{2^s})^n : \prod_{i=1}^n b_i = 1\}/ \{(x,\dots,x) : x^{n'} = 1\}$. Then $T' \triangleleft P$. So by Clifford's Theorem (Theorem \ref{cliff}), $\varphi|_{T'}$ decomposes into a direct sum of irreducibles in the following manner:
$$\varphi|_{T'} \cong  \left( \oplus_{i=1}^c  \lambda_i \right)^{\oplus d}, \text{ for some } c, d,$$ 
with the $\lambda_i$ non-isomorphic, and $P_2(S_{n})$ acts transitively on the isomorphism classes of the $\lambda_i$, so the $\lambda_i$ have the same dimension and the number of $\lambda_i$, $c$, divides $|P_2(S_{n})|$. Since $\varphi$ is non-trivial on $T_j \cap C[2]$, one of the $\lambda_i$ must be non-trivial on $T_j \cap C[2]$. Without loss of generality, assume that $\lambda_1$ is non-trivial on $T_j \cap C[2]$.

 Note that the irreducible representations of $T'$ are in bijection with irreducible representations of $T = \{\mbf{b} \in (\mu_{2^s})^n : \prod_{i=1}^n b_i = 1\}$ which are trivial on $\{(x,\dots,x) : x^{n'} = 1\}$.   
% By Lemma \ref{corrlemmaPSLn}, the irreducible representations of $T$ are given by $\Psi_{\mbf{a}}|_T$, with $\mbf{a} \in H/\Gamma$, for $\Gamma = \text{Gal}(k(\zeta_{2^s})/k)$, and if $\Psi_{\mbf{a}}$ is non-trivial on $T[2]$, then $\Psi_\mbf{a}$ has dimension $[k(\zeta_{2^s}):k]$.  
% Since $\lambda_1$ is non-trivial on $T_j \cap C[2] \subset T[2]$, it must be non-trivial on $T[2]$, so we must have $\dim(\lambda_1) = [k(\zeta_{2^s}):k]$, and so $\dim(\lambda_i) = [k(\zeta_{2^s}):k]$ for all $i$. 
 By Lemma \ref{corrlemmaPSLn}, since we are assuming that $\Gamma = \text{Gal}(k(\zeta_{2^s})/k)$ is trivial, the irreducible representations of $T$ have dimension $1$ and are given by $\Psi_{\mbf{a}}|_T$ , with $\mbf{a} \in H$.
 
Note that for $\mbf{x} = (x,\dots,x)$, $\psi_{\mbf{a}}(\mbf{x}) = x^{\sum_{i=1}^n a_i}$.  So $\Psi_{\mbf{a}}|_T \in \widehat{T}/\Gamma$ will be trivial on $\{(x,\dots,x) : x^{n'} = 1\}$ if and only if $\sum_{i=1}^n a_i = 0 \mod 2^v$, where $v = \min(v_2(n'),s)$. So $\lambda_1 \cong \Psi_{\mbf{a}}|_T$ for some $\mbf{a} \in H/\Gamma$ with $\sum_{i=1}^n a_i = \mod 2^v$.

Also, since $\lambda_1 \cong \Psi_\mbf{a}|_T$ is non-trivial on $T_j \cap C[2] = \langle \mbf{b}^j \rangle$, where $(\mbf{b}^j)_i = \begin{cases} \zeta_2, &i \in I_j\\
1, &i \notin I_j \end{cases},$
we must have 
$$1 \neq \prod_{i\in I_{j}} \zeta_{2}^{a_i} = \zeta_{2}^{\sum_{i\in I_{j}} a_i} = \zeta_2^{A_{j}}.$$
Thus $2 \nmid A_{j}$ and so $A_{j}$ is invertible.  So by Corollary \ref{irrHPSLn1cor}, since we are assuming that $\Gamma$ is trivial, the orbit of $\lambda_i = \Psi_{\mbf{a}}|_{T}$ under the action of $P_2(S_n)$ has size at least $2^{k_j + v_l(n)}$. So $c \geq 2^{k_j + v_l(n)}$. Thus for $\varphi = \varphi_j$,
$$\dim(\varphi) \geq 2^{k_j + v_l(n)}.$$
Hence
\begin{align*}
\dim(\rho) &= \sum_{j=1}^{\xi_2(n)-1} \dim(\varphi_j)\\
&\geq \sum_{j=1}^{\xi_2(n)-1} 2^{k_j + v_l(n)}\\
&= 2^{v_l(n)}\left(\sum_{j=1}^{\xi_2(n)-1} 2^{k_j}\right)\\
&= 2^{v_l(n)}(n-2^{v_l(n)})\\
&= \ed_k(PGL_n(\F_q),2)
\end{align*}
Also, since $PSL_n(\F_q) \subset PGL_n(\F_q)$,
$$\ed_k(PSL_n(\F_q),2) \leq \ed_k(PGL_n(\F_q),2).$$
Therefore for $\Gamma = \text{Gal}(k(\zeta_{2^s})/k)$ trivial, $n \neq 2^t$,
\begin{align*}
\ed_k(PSL_n(\F_q),2) &= \ed_k(PGL_n(\F_q),2) = 2^{v_2(n)}(n-2^{v_2(n)}). \qedhere
\end{align*}

%$$\dim(\varphi) \geq 2^{k_j + v_l(n)}[k(\zeta_{2^s}):k].$$
%Hence
%\begin{align*}
%\dim(\rho) &= \sum_{j=1}^{\xi_2(n)-1} \dim(\varphi_j)\\
%&\geq \sum_{j=1}^{\xi_2(n)-1} 2^{k_j + v_l(n)}[k(\zeta_{2^s}):k]\\
%&= 2^{v_l(n)}\left(\sum_{j=1}^{\xi_2(n)-1} 2^{k_j}\right)[k(\zeta_{2^s}):k]\\
%&= 2^{v_l(n)}(n-2^{v_l(n)})[k(\zeta_{2^s}):k]\\
%&= \ed_k(PGL_n(\F_q),2)
%\end{align*}
%Also, since $PSL_n(\F_q) \subset PGL_n(\F_q)$,
%$$\ed_k(PSL_n(\F_q),2) \leq \ed_k(PGL_n(\F_q),2).$$
%Therefore for $\Gamma$ trivial, $n \neq 2^t$,
%\begin{align*}
%\ed_k(PSL_n(\F_q),2) &= \ed_k(PGL_n(\F_q),2) = 2^{v_2(n)}(n-2^{v_2(n)})[k(\zeta_{2^s}):k]. \qedhere
%\end{align*}
\end{proof}

\newpage

\section{Appendix}

In this appendix, we provide some details for the computations in this article.

\subsection{Proof of Lemma \ref{irrHPSLn3}}\label{App2}

\begin{lma}[Lemma \ref{irrHPSLn3}] Let $H = (\Z/2^s\Z)^{n}/\{(x,\dots,x)\}$, $n = 2^t$ with $t > 1$, $v = \min(v_2(n'),s)$, and $\mbf{a} \in H$ with 
$$\sum_{i=1}^{n} a_i = 0 \mod 2^v \text{ and } \Psi_{\mbf{a}} \text{ non-trivial on } Z(P)[2].$$ 
Then
$$|\text{orbit}(\mbf{a})| \geq 
\begin{cases} 2^{2t-2}, &v=1\\
2^{2t-1}, &v > 1
 \end{cases}$$ under the action of $P_2(S_n)$ on $H$. 
\end{lma}

\begin{proof}

$Z(P)[2]$ is generated by 
$$g = (\zeta_2, \dots, \zeta_2, 1, \dots, 1).$$  So since $\Psi_{\mbf{a}}$ is non-trivial on $Z(P)[2]$. we must have 
$$1 \neq \Psi_{\mbf{a}}(g) = \zeta_{2}^{\sum_{j=1}^{2^{t-1}} a_i} = \zeta_2^{A_1}.$$ 
Thus $2 \nmid A_1$ and so $A_1$ is invertible. Since $2^v \divides \sum_{i=1}^n a_i$, we have $2 \divides \sum_{i=1}^n a_i = A_1 + A_2$. So since $A_1$ is invertible and $A_1 + A_2$ is not, $A_2$ must also be invertible.

\noindent \textbf{Case 1: } $\mbf{n = 4 \text{ } (t = 2)}$  

%If $n = 4$, then we have an orbit of size at least $2$ given by permuting the indices $1$ and $2$.  Let $\sigma$ be the permutation that permutes the indices $3$ and $4$.  Suppose by way of contradition that $\sigma(\mbf{a})$ is the same as one of the $2$ elements in the orbit that we have already found (give by $\tau(\mbf{a})$ for $\tau \in P_2(S_2)$ acting on the first two indices).  In other words, $\sigma(\mbf{a}) = \tau(\mbf{a}) + (x,\dots,x)$ for some $\tau \in P_2(S_2), x \in \Z/2^s\Z$. Then we have $a_3 = a_4 + x$ and $a_4 = a_3 + x = a_4 + 2x$. So $x = 2^{z(s-1)}$ for some $z \in \Z/2\Z$. And 
%$$A_2 = a_3 + a_4 = a_4 + x + a_4 = 2a_4 + x,$$
%which is not invertible, a contradiction. Therefore, $\sigma(\mbf{a})$ is not equal to any of the $\tau(\mbf{a}) \mod \{(x,\dots, x)\}$ for $\tau \in P_2(S_{2^{t-1}})$ acting on the indices in $J_1$. Thus the size of the orbit of $\mbf{a}$ is at least $2 + 1$, and so it must be at least $4$ since it must divide $|P_2(S_4)|$ which is a power of $2$.

Let $\sigma_1$ be the permutation given by $\sigma_1(a_1,a_2,a_3,a_4) = (a_2,a_1,a_3,a_4)$. Suppose by way of contradiction that $\sigma_1$ is in the stabilizer of $\mbf{a}$. Then $(a_2,a_1,a_3,a_4) = (a_1,a_2,a_3,a_4) + (x,x,x,x)$ for some $x$. So $x = 0$ and $a_1 = a_2$. But then $A_1 = a_1 + a_2 = 2a_1$, which is not invertible, a contradiction. Therefore, $\sigma_1$ is not in the stabilizer of $\mbf{a}$. By similar reasoning $\sigma_2$ given by $\sigma_2(a_1,a_2,a_3,a_4) = (a_1,a_2,a_4,a_3)$ is not in the stabilizer of $\mbf{a}$.

Let $\sigma_3$ be the permutation given by $\sigma_3(a_1,a_2,a_3,a_4) = (a_3,a_4,a_2,a_1)$. Suppose by way of contradiction that $\sigma_3$ is in the stabilizer of $\mbf{a}$. Then $(a_3,a_4,a_2,a_1) = (a_1,a_2,a_3,a_4) + (x,x,x,x)$ for some $x$.  Thus $a_3 = a_1 + x$ and $a_2 = a_3 + x = a_1 + 2x$. So $A_1 = a_1 + a_2 = 2a_1 + 2x$, which is not invertible, a contradiction. Therefore $\sigma_3$ is not in the stabilizer of $\mbf{a}$. By similar reasoning $\sigma_4$ given by $\sigma_4(a_1,a_2,a_3,a_4) = (a_4,a_3,a_1,a_2)$ is not in the stabilizer of $\mbf{a}$.

Let $\sigma_5$ be given by $\sigma_5(a_1,a_2,a_3,a_4) = (a_2,a_1,a_4,a_3)$. Suppose by way of contradiction that $\sigma_5$ is in the stabilizer of $\mbf{a}$. Then $(a_2,a_1,a_4,a_3) = (a_1,a_2,a_3,a_4) + (x,x,x,x)$ for some $x$.
Thus $a_2 = a_1 + x$ and $a_1 = a_2 + x = a_1+2x$. So $2^{s-1} \divides x$ and hence $2 \divides x$ (since $s > 1$). Then $A_1 = a_1 + a_2 = 2a_1 + x$ is not invertible, which is a contradiction. Therefore $\sigma_5$ is not in the stabilizer of $\mbf{a}$.

Thus the stabilizer has at most $8-5=3$ elements. And since the size of the stabilizer must divide $8$, it can have size at most $2$. Thus the orbit of $\mbf{a}$ has size at least $\frac{8}{2} = 4$.

\textbf{Subsubcase 1a:} If $\mbf{v=1}$, then for $\mbf{a} = (1,0,1,0)$, the orbit has size $4 = 2^{2t-2}$, given by $(1,0,1,0), (0,1,1,0), (1,0,0,1),$ and $(0,1,0,1)$.

\textbf{Subsubcase 1b:} If $\mbf{v>1}$, then we must have $v = 2$, $n' = 4$, and $\sum_{i=1}^4 a_i = 0 \mod 4$.

Let $\sigma_6$ be the permutation given by $\sigma_6(a_1,a_2,a_3,a_4) = (a_3,a_4,a_1,a_2)$. Suppose by way of contradiction that $\sigma_6$ is in the stabilizer of $\mbf{a}$. Then $(a_3,a_4,a_1,a_2) = (a_1,a_2,a_3,a_4) + (x,x,x,x)$ for some $x$.  Thus $a_3 = a_1 + x$ and $a_1 = a_3 + x = a_1+2x$. So $2^{s-1} \divides x$. Then 
\begin{align*}
4 &\divides a_1 + a_2 + a_3 + a_4\\
&= a_1 + a_2 + (a_1 + x) + (a_2 + x)\\
&= 2a_1 + 2a_2 + 2x\\
&= 2a_1 + 2a_2\\
&= 2(a_1 + a_2).
\end{align*}
Thus $2 \divides a_1 + a_2$, which is a contradiction with the fact that $A_1 = a_1 + a_2$ is invertible.  Therefore, $\sigma_6$ is not in the stabilizer of $\mbf{a}$. By similar reasoning $\sigma_7$ given by $\sigma_7(a_1,a_2,a_3,a_4) = (a_4,a_3,a_2,a_1)$ is not in the stabilizer of $\mbf{a}$.

Thus all seven of the non-trivial elements of $P_2(S_4)$ are not in the stabilizer of $\mbf{a}$. Therefore, for $n = 4, v > 1$ the stabilizer is trivial, so the orbit has size $|P_2(S_4)| = 8 = 2^{2t-1}$.

\noindent \textbf{Case 2: } $\mbf{n > 4 \text{ } (t > 2)}$:

For $j = 1,2$, let $K_j$ denote the $j$th sub-block of $2^{t-2}$ entries in $J_2$. And let $B_j = \sum_{i \in K_j} a_i$. Then since $A_2 = \sum_{j=1}^2 B_j$ is invertible, $B_j$ must be invertible for some $j$.  Without loss of generality, assume that $B_1$ is invertible. Consider the copy of $P_2(S_{2^{t-1}}) \times P_2(S_{2^{t-2}}) \subset P_2(S_n)$ that acts on $J_1 \times K_1$. This copy of $P_2(S_{2^{t-1}}) \times P_2(S_{2^{t-2}})$ acts trivially on the entries in $J_2\backslash K_1$, so if $\tau(\mbf{a}) = \mbf{a} + (x,\dots,x)$, then $x = 0$. So the orbit of $\mbf{a}$ under the action of  $P_2(S_{2^{t-1}}) \times P_2(S_{2^{t-2}}) \subset P_2(S_n)$ on $H$ is the same as the orbit under the action of $P_2(S_{2^{t-1}}) \times P_2(S_{2^{t-2}}) \subset P_l(S_n)$ on $(\Z/2^s\Z)^n$, which is equal to the product of the orbits of $\mbf{a}$ under the actions of $P_2(S_{2^{t-1}})$ and $P_2(S_{2^{t-2}})$. So by Lemma \ref{irrH1}, we can conclude that the orbit has size at least $2^{2t-3}$.

Consider the action of $P_2(S_{2^{t-1}}) \times P_2(S_{2^{t-1}})$ on $\mbf{a}$. If $\sum_{i \in K_2} = \sum_{i \in K_1} a_i + 2^{t-2}x$ for some $x$, then we would have
$$A_2 = \sum_{i \in K_1} a_i + \sum_{i \in K_2} a_i = 2\sum_{i \in K_1} a_i + 2^{t-2}x,$$
which is not invertible since $t > 2$. So we can conclude that $\sum_{i \in K_2} a_i \neq \sum_{i \in K_1} a_i + 2^{t-2}z$ for any $x$. So for $\tau$ a permutation that maps $K_1$ to $K_2$, $\tau(\mbf{a})$ is not equal to any of the $\sigma(\mbf{a}) + (x,\dots,x)$ for $\sigma \in P_2(S_{2^{t-1}}) \times P_2(S_{2^{t-2}})$ acting on $J_1 \times K_1$ (since these $\tau$ map $K_1$ to itself). Thus the size of the orbit under the action of $P_2(S_{2^{t-1}}) \times P_2(S_{2^{t-1}})$ is at least $2^{2t-3} + 1$, and so it must be at least $2^{2t-2}$ since it must divide $|P_2(S_n)|$, which is a power of $2$.

\textbf{Subcase 2a: } If $\mbf{t>2}$, $\mbf{v = 1}$,  then for $\mbf{a} = (1, 0, \dots, 0, 1, 0, \dots, 0)$, the orbit has size $2^{2t-2}$ given by the images under the action of $P_2(S_{2^{t-1}}) \times P_2(S_{2^{t-1}})$.

\textbf{Subcase 2b: } $\mbf{t>2}$, $\mbf{v > 1}$,

Since $A_1 + A_2 = 0 \mod 2^v$, we can conclude that $2^v \divides A_1 + A_2$. If we had $A_2 = A_1 + 2^{t-1}x$ for some $x$, then we would have $2^v \divides 2A_1 + 2^{t-1}x$ and so $2 \divides A_1$ since $v > 1$ and $t > 2$, a contradiction with the fact that $A_1$ is invertible.  So we cannot have $A_2 = A_1 + 2^{t-1}x$ for any $x$. So for $\tau$ a permutation that maps $J_1$ to $J_2$, $\tau(\mbf{a})$ is not equal to any of the $\sigma(\mbf{a}) + (x,\dots,x)$ for $\sigma \in P_2(S_{2^{t-1}}) \times P_2(S_{2^{t-1}})$ acting on $J_1 \times J_2$ (since these $\tau$ map $J_1$ to itself). Thus the size of the orbit under the action of $P_2(S_n)$ is at least $2^{2t-2} + 1$, and so it must be at least $2^{2t-1}$ since it must divide $|P_2(S_n)|$, which is a power of $2$.

\end{proof}

\subsection{Proof of Lemma \ref{irrHPSLn1}}\label{App3}

\begin{lma}[Lemma \ref{irrHPSLn1}] Let $H = (\Z/2^s\Z)^{n}/\{(x,\dots,x)\}$, $n \neq 2^t$, $v = \min(v_2(n'),s)$, $j_1 \in \{1, \dots, \xi_2(n)\}$, and $\mbf{a} \in H$ with $$\sum_{i=1}^{n} a_i = 0 \mod 2^v, \text{ } A_{j_1}\text{ invertible}.$$ 
Then 
$$|\text{orbit}(\mbf{a})| \geq 2^{k_{j_1}+v_2(n)}$$
under the action of $P_2(S_n)$ on $H$. 
\end{lma}

\begin{proof}[Proof of Lemma \ref{irrHPSLn1}]

Since $2 \divides n'$ and $s \geq 1$, $2 \divides 2^v$. So since $2^v \divides \sum_{i=1}^n a_i$, we have $2 \divides \sum_{i=1}^n a_i$.  So since $A_{j_1}$ is invertible and $\sum_{i=1}^n a_i = \sum_{j=1}^{\xi_2(n)} A_j$ is not, we must also have $A_{j_2}$ invertible for some $j_2 \neq j_1$. Note that since $n \neq 2^t$, $\xi_2(n) \geq 2$.

\noindent \textbf{Case 1: } $\mbf{\xi_2(n) > 2}$

Consider the copy of $P_2(S_{2^{k_{j_1}}}) \times P_2(S_{2^{k_{j_2}}}) \subset P_2(S_n)$ that acts on $I_{j_1} \times I_{j_2}$. Since $\xi_2(n) > 2$, this copy of $P_2(S_{2^{k_{j_1}}}) \times P_2(S_{2^{k_{j_2}}})$ acts trivially on $I_{j_3}$ for some $j_3$, so if $\tau(\mbf{a}) = \mbf{a} + (x,\dots,x)$, then $x = 0$. So the orbit of $\mbf{a}$ under the action of  $P_2(S_{2^{k_{j_1}}}) \times P_2(S_{2^{k_{j_2}}}) \subset P_2(S_n)$ on $H$ is the same as the orbit under the action of $P_2(S_{2^{k_{j_1}}}) \times P_2(S_{2^{k_{j_2}}}) \subset P_2(S_n)$ on $(\Z/2^s\Z)^n$, which is equal to the product of the orbits of $\mbf{a}$ under the actions of $P_2(S_{2^{k_{j_1}}})$ and $P_2(S_{2^{k_{j_2}}})$. So by Lemma \ref{irrH1}, we can conclude that the orbit has size at least $2^{k_{j_1}+k_{j_2}}$. And for all $j_2$, $k_{j_2} \geq v_2(n)$; so $2^{k_{j_1} + k_{j_2}} \geq 2^{k_{j_1} + k_{v_2(n)}}$. Hence $|\text{orbit}(\mbf{a})| \geq 2^{k_{j_1}+v_2(n)}$. 

\noindent \textbf{Case 2: } $\mbf{\xi_2(n) = 2}$

If $\xi_2(n) = 2$, then both $A_1$ and $A_2$ are invertible. For $j = 1,2$, let $K_j$ denote the $j$th sub-block of $2^{k_2-1}$ entries in $I_2$. Then since $A_2 = \sum_{j=1}^2 (\sum_{i \in K_j} a_i)$ is invertible, $\sum_{i \in K_j} a_i$ must be invertible for some $j$.  Without loss of generality, assume that $\sum_{i \in K_1} a_i$ is invertible. Consider the copy of $P_2(S_{2^{k_1}}) \times P_2(S_{2^{k_2-1}}) \subset P_2(S_n)$ that acts on $I_1 \times K_1$. This copy of $P_2(S_{2^{k_1}}) \times P_2(S_{2^{k_2-1}})$ acts trivially on the entries in $I_2\backslash K_1$, so if $\tau(\mbf{a}) = \mbf{a} + (x,\dots,x)$, then $x = 0$. So the orbit of $\mbf{a}$ under the action of  $P_2(S_{2^{k_1}}) \times P_2(S_{2^{k_2-1}}) \subset P_2(S_n)$ on $H$ is the same as the orbit under the action of $P_2(S_{2^{k_1}}) \times P_2(S_{2^{k_2-1}}) \subset P_2(S_n)$ on $(\Z/2^s\Z)^n$, which is equal to the product of the orbits of $\mbf{a}$ under the actions of $P_2(S_{2^{k_1}})$ and $P_2(S_{2^{k_2-1}})$. So by Lemma \ref{irrH1}, we can conclude that the orbit has size at least $2^{k_1+k_2-1}$. 

\textbf{Subcase 2a: } $\mbf{k_2 = 1}$

If $k_2 = 1$, then we have $n = 2^{k_1} + 2$. Let $\sigma$ be given by $n \mapsto n - 1 \mapsto n$. Then $\sigma$ permutes the $a_i$ in $I_2$. Suppose by way of contradiction that $\sigma(\mbf{a})$ is the same as one of the $2^{k_1}$ elements in the orbit that we have already found (given by $\tau(\mbf{a})$ for $\tau \in P_2(S_{2^{k_1}})$). In other words, $\sigma(\mbf{a}) = \tau(\mbf{a}) + (x,\dots,x)$ for some $\tau \in P_2(S_{2^{k_1}}), x \in \Z/2^s\Z$. Then for $n-1 \leq i \leq n$, $\sigma(\mbf{a})_i = a_i + x$ (since $\tau$ fixes the indices in $I_2$). So $a_{n-1} = a_n + x$ and $a_n = a_{n-1} + x = a_n + 2x$. So $x = 2^{z(s-1)}$ for $z \in \Z/2\Z$. And 
$$A_2 = a_{n-1} + a_n = a_n + x + a_n = 2a_n + x,$$
which is not invertible, a contradiction. Therefore, $\sigma(\mbf{a})$ is not equal to any of the $\tau(\mbf{a}) \mod \{(x,\dots,x)\}$ for $\tau \in P_2(S_{2^{k_1}})$. Thus the size of the orbit of $\mbf{a}$ is at least $2^{k_1} + 1$, and so it must be at least $2^{k_1+1}$ since it must divide $|P_2(S_{n})|$ which is a power of $2$.

\textbf{Subcase 2b: } $\mbf{k_2 > 1}$

Let $\sum_{i \in K_1} a_i = y$. If $\sum_{i\in K_2} a_i = y + 2^{k_2-1}x_j$ for some $x$, then we would have 
$$\sum_{i\in I_2} a_i = \sum_{j=1}^2 \left(\sum_{i \in K_j}a_i\right) = y + (y + 2^{k_2-1}x) = 2y + 2^{k_2-1}x,$$
which is not invertible for $k_2 > 1$. So we can conclude that $\sum_{i \in K_2} a_i \neq \sum_{i \in K_1} a_i + 2^{k_2-1}x$ for any $x$. Then for $\sigma$ a permutation that maps $K_1$ to $K_2$, $\sigma(\mbf{a})$ is not equal to any of the $\tau(\mbf{a}) + (x,\dots,x)$ for $\tau \in P_2(S_{2^{k_1}}) \times P_2(S_{2^{k_2-1}})$ acting on $I_1 \times K_1$ (since these $\tau$ map $K_1$ to itself).  Thus the size of the orbit is at least $2^{k_1+k_2-1} + 1$, and so it must be at least $2^{k_1+k_2}$ since it must divide $|P_2(S_{n})|$ which is a power of $2$.\qedhere
\end{proof}

\bibliographystyle{plain}
\bibliography{references}

\end{document}